\documentclass[11pt]{amsart}
\usepackage[bookmarks=true]{hyperref}
\usepackage{amsmath}
\usepackage{amsthm,amssymb,amscd}
\usepackage[arrow,matrix]{xy}
\CompileMatrices
\usepackage{enumerate}
\usepackage{mathrsfs}
\usepackage{verbatim}
\usepackage{version}
\usepackage{pb-diagram}
\usepackage{color}
\newcommand{\flip}{\mathrm{flip}}
\renewcommand{\tilde}{\widetilde}
\theoremstyle{definition}
\numberwithin{equation}{section}
\newtheorem{theo}[equation]{Theorem}
\newtheorem{defn}[equation]{Definition}
\newtheorem{exa}[equation]{Example}
\newtheorem{prop}[equation]{Proposition}
\newtheorem{cor}[equation]{Corollary}
\newtheorem{lem}[equation]{Lemma}
\newtheorem{conj}[equation]{Conjecture}
\newtheorem{rem}[equation]{Remark}

\newenvironment{pf*}[1]{\proof[#1]}{\endproof}
\newcommand{\Irr}{\operatorname{Irr}}
\renewcommand{\proofname}{\textsf{Proof}}

\renewcommand{\mod}{\operatorname{mod}}
\newcommand{\Sub}{\operatorname{Sub}}

\def\dim{\mathop{\mathrm{dim}}\nolimits}
\def\tr{\mathop{\mathrm{tr}}\nolimits}

\def\Im{\mathop{\mathrm{Im}}\nolimits}
\def\Ker{\mathop{\mathrm{Ker}}\nolimits}

\def\Hom{\mathop{\mathrm{Hom}}\nolimits}
\def\Ext{\mathop{\mathrm{Ext}}\nolimits}
\def\End{\mathop{\mathrm{End}}\nolimits}
\newcommand{\dimv}{\underline{\mathrm{dim}}}

\newcommand{\mfr}[1]{{\mathfrak{#1}}}
\newcommand{\mbf}[1]{{\mathbf{#1}}}
\newcommand{\mbb}[1]{{\mathbb{#1}}}
\newcommand{\mca}[1]{{\mathcal{#1}}}

\newcommand{\mscr}[1]{{\mathscr{#1}}}
\newcommand{\thmref}[1]{Theorem~\ref{#1}}
\newcommand{\secref}[1]{\S\ref{#1}}
\newcommand{\subsecref}[1]{\S\ref{#1}}
\newcommand{\lemref}[1]{Lemma~\ref{#1}}
\newcommand{\propref}[1]{Proposition~\ref{#1}}
\newcommand{\corref}[1]{Corollary~\ref{#1}}

\newenvironment{aenumerate}{%
  \begin{enumerate}%
  }{\end{enumerate}}
\newenvironment{renumerate}{%
  \begin{enumerate}%
  }{\end{enumerate}}
\newcommand{\set}[1]{\left\{#1\right\}}
\newcommand{\bracket}[1]{\left\langle#1\right\rangle}

\newcommand{\+}{\oplus}

\newcommand{\fit}[1]{{\widetilde{f}}_{#1}}
\newcommand{\eit}[1]{{\widetilde{e}}_{#1}}
\newcommand{\Uq}{\mbf{U}_q}
\newcommand{\Qq}{\mbb{Q}(q)}

\newcommand{\gr}{\mathrm{gr}}
\newcommand{\tw}{\tilde{w}}
\newcommand{\Gup}{G^{\operatorname{up}}}
\newcommand{\Glow}{G^{\operatorname{low}}}
\newcommand{\Llow}{\mscr{L}^{\operatorname{low}}}
\newcommand{\Lup}{\mscr{L}^{\operatorname{up}}}

\newcommand{\wt}{\operatorname{wt}}
\newcommand{\Binfty}{\mscr{B}(\infty)}
\newcommand{\vep}{\varepsilon}
\newcommand{\vphi}{\varphi}

\newcommand{\beq}{\begin{equation}}
\newcommand{\eeq}{\end{equation}}

\newcommand{\Frac}{\operatorname{Frac}}

\renewcommand{\bar}{\overline{\phantom{x}}}
\newcommand{\Fl}{\mscr{F}\!\ell}
\newcommand{\jbar}{j}
\newcommand{\pibar}{\pi}
\newcommand{\up}{\operatorname{up}}
\newcommand{\low}{\operatorname{low}}

\newcommand{\Supp}{\operatorname{Supp}}

\newcommand{\U}{\mbf{U}}

\newenvironment{NB}{\color{red}{\bf N.B.} \footnotesize}{}
\excludeversion{NB}
\newenvironment{NB2}{\color{blue}{\bf N.B.} \tiny}{}
\excludeversion{NB2}
\usepackage{a4wide}

 \makeatletter
       \def\@makefnmark{%
               \leavevmode
               \raise.9ex\hbox{\check@mathfonts
                       \fontsize\sf@size\z@\normalfont%
                               \@thefnmark}%
       }
 \makeatother

\title[Quantum unipotent subgroup and dual canonical basis]
{Quantum unipotent subgroup 
and dual canonical basis%
}
\author[Yoshiyuki Kimura]
{Yoshiyuki Kimura}
\address{Research Institute for Mathematical Science\\
Kyoto University\\
Kyoto 606-8502\\
Japan}
\date{\today}
\email{ykimura@kurim.kyoto-u.ac.jp}
\thanks{This work is supported by Kyoto University Global COE Program `Fostering top leaders in mathematics'.}

\begin{document}
\begin{abstract}
In a series of works \cite{GLS:semican1, GLS:semican2, GLS:Verma, GLS:adaptable, GLS:KacMoody, GLS:chamber},
{Gei\ss-Leclerc-Schr\"{o}er} defined the cluster algebra structure on the coordinate ring $\mbb{C}[N(w)]$ of the unipotent subgroup,
associated with a Weyl group element $w$.
And they proved cluster monomials are contained in Lusztig's \emph{dual semicanonical basis} $\mca{S}^{*}$.
We give a set up for the quantization of their results and propose a conjecture which relates the quantum cluster algebras in \cite{BZ:qcluster} to the \emph{dual canonical basis} $\mbf{B}^{\up}$.
In particular, we prove that the quantum analogue $\mca{O}_{q}[N(w)]$ of $\mbb{C}[N(w)]$ has the induced basis from $\mbf{B}^{\up}$,
which contains quantum flag minors and satisfies a factorization property with respect to the `$q$-center' of $\mca{O}_{q}[N(w)]$.
This generalizes Caldero's results \cite{Cal:unity, Cal:adapted, Cal:qflag} from ADE cases to an arbitary symmetrizable Kac-Moody Lie algebra.

\end{abstract}

\maketitle
\tableofcontents

\section{Introduction}
\subsection{The canonical basis $\mbf{B}$ and the dual canonical basis $\mbf{B}^{\up}$}
Let $\mfr{g}$ be a symmetrizable Kac-Moody Lie algebra,
$\Uq(\mfr{g})$  its associated quantized enveloping algebra, and $\Uq^-(\mfr{g})$ its negative part.
In \cite{Lus:quiver}, Lusztig constructed the canonical basis $\mbf{B}$ of $\Uq^{-}(\mfr{g})$
by a geometric method when $\mfr{g}$ is symmetric.
In \cite{Kas:crystal}, Kashiwara constructed the (lower) global basis $\Glow(\mscr{B}(\infty))$ by a purely algebraic method.
Grojnowski-Lusztig \cite{GroLus} showed that the two bases coincide when $\mfr{g}$ is symmetric.
We call the basis the \emph{canonical basis}.
There are two remarkable properties of the canonical basis,
one is the positivity of structure constants of multiplication and comultiplications,
and another is Kashiwara's crystal structure $\mscr{B}(\infty)$, which is a combinatorial machinery useful for applications to representation theory, such as tensor product decomposition.

Since $\Uq^-(\mfr{g})$ has a natural pairing which makes it into a (twisted) self-dual bialgebra,
we consider the dual basis $\mbf{B}^{\up}$ of the canonical basis in $\Uq^-(\mfr{g})$.
We call it the \emph{dual canonical basis}.
\subsection{Cluster algebras}
Cluster algebras were introduced by Fomin and Zelevinsky \cite{FZ:cluster1} and intensively studied also with Berenstein \cite{FZ:cluster2, BFZ:cluster3, FZ:cluster4}
with an aim of providing a concrete and combinatorial setting for the study of Lusztig's (dual) canonical basis and total positivity.
Quantum cluster algebras were also introduced by Berenstein and Zelevinsky \cite{BZ:qcluster}, Fock and Goncharov \cite{FG:cluster1, FG:cluster2, FG:qcluster} independently.
The definition of (quantum) cluster algebra was motivated by Berenstein and Zelevinsky's earlier work \cite{BZ:string} where combinatorial and multiplicative structures of 
the dual canonical basis were studied for $\mfr{g}=\mfr{sl}_2$ and $\mfr{sl}_3$.
Let us quote from \cite{FZ:cluster1}:
\begin{quote}
We conjecture that the above examples can be extensively generalized: for any simply-connected connected semisimple group $G$,  the coordinate rings $\mbb{C}[G]$ and $\mbb{C}[G/N]$, as well as coordinate rings of many other interesting varieties related to $G$, have a natural structure of a cluster algebra.
This structure should serve as an algebraic framework for the study of ``dual canonical basis'' in these coordinate rings and their $q$-deformations.
In particular, we conjecture that all monomials in the variables of any given cluster (the cluster monomials) belong to this dual canonical basis.
\end{quote}
\begin{NB}
Here ``dual canonical basis'' can be considered as a conjectural analogue of the dual canonical basis of $\Uq^{-}(\mfr{g})$ and $V(\lambda)$.
\begin{itemize}
\item For $\mca{O}_{q}[G]$ with $G$ finite case, we have natural bar-involution, which can be defined as the ``dual bar involution'' as in \cite{BZ:qcluster}.
See also \cite[Chapter 29]{Lus:intro}  and \cite{Lus:Zform}.
\item For general $\mca{O}_{q}[G]$, we have definition as in \cite{Kas:global}, see also \cite{LZ:1002.4701}.
\item For $\mca{O}_{q}[G/N]$ and $\mca{O}_{q}[G/P]$, the definition of ``bar involution'' is not so clear for us.
These can be defined as direct sum of integrable highest weight module, so ``component wise bar involution'' make sense.
But it not clear for us that this i compatible with the ``dual canonical basis'' of quantum Schubert cell and various .... 
\end{itemize}
\end{NB}
In \cite{BFZ:cluster3}, it was shown that the coordinate ring of the double Bruhat cell has a part of structures of a cluster algebra.

A cluster algebra $\mscr{A}$ is a subalgebra of rational function field $\mbb{Q}(x_1, x_2, \cdots, x_r)$ of $r$ indeterminates which is equipped with a distinguished set of generators (\emph{cluster variables})
which is grouped into overlapping subsets (\emph{clusters}) consisting of precisely $r$ elements.
Each subset is defined inductively by a sequence of certain combinatorial operation (\emph{seed mutations}) from the initial seed.
The monomials in the variables of a given single cluster are called \emph{cluster monomials}. 
However, it is not known that a cluster algebra have a basis, related to the dual canonical basis, which includes all cluster monomials in general.
\subsection{Cluster algebra and the dual semicanonical basis}
In a series of works \cite{GLS:semican1, GLS:semican2, GLS:Verma, GLS:adaptable, GLS:KacMoody, GLS:chamber},
Gei\ss, Leclerc and Schr\"{o}er introduced a cluster algebra structure on the coordinate ring $\mbb{C}[N(w)]$ of the unipotent subgroup 
associated with a Weyl group element $w$.
Furthermore they show that the \emph{dual semicanonical basis} $\mca{S}^*$ is compatible with the inclusion $\mbb{C}[N(w)]\subset U(\mfr{n})_{\operatorname{gr}}^*$
and contains all cluster monomials.
Here the dual semicanonical basis is the dual basis of the semicanonical basis of $U(\mfr{n})$, introduced by Lusztig \cite{Lus:aff, Lus:semican},
and ``compatible'' means that $\mca{S}^*\cap \mbb{C}[N(w)]$ forms a $\mbb{C}$-basis of $\mbb{C}[N(w)]$.

It is known that canonical and semicanonical bases share similar combinatorial properties (crystal structure),
but they are different (examples can be found in \cite{Lampe}
\footnote{In \cite{Lampe}, $\underline{\mca{S}}$ is the specialization of the dual canonical basis, while $\underline{\Sigma}$ is the dual semicanonical basis thanks to \cite{GLS:chamber}.}).

\begin{NB}
they studied the relationship between the full-subcategory $\mca{C}_w$ of $\mod(\Lambda)$
which is associated with a Weyl group element $w$ and showed that the associated cluster algebra  with 
coefficients \cite{FZ:cluster4} $\mscr{A}(\mca{C}_w)$ has a natural realization by a subalgebra of the restricted dual
of the universal envelopping algebra $U(\mfr{n}_-)_{\gr}^*$ and compatible with the 
\emph{dual semicanonical basis} $\mca{S}^*$.

Let us give some details of their works.
Let $Q$ be an acylic quiver and $\Lambda_Q$ be an associated completed preprojective alebra $\Lambda=\Lambda_{Q}$.
If we consider the representation variety of $\Lambda_Q$, then we obtain Lusztig's quiver varieties.

For a vertex $i\in Q_0$, we consider a two-sided ideal $I_i=\Lambda(1-e_i)\Lambda$,
where $e_i$ is an idempotent which corresponds to $i\in Q_0$.
For a Weyl group element $w\in W$ and fix a reduced expression $\tw=(i_1, i_2, \cdots, i_l)\in R(w)$,
we define $\Lambda$-bimodule $I_{\tw_{\leq k}}$ by $I_{i_1}\otimes_{\Lambda}\cdots \otimes_{\Lambda}I_{i_k}$ for $1\leq k\leq l$
and consider $\Lambda_{\tw_{\leq k}}:=\Lambda/I_{\tw_{\leq k}}$ and $T_{\tw}:=\bigoplus_{1\leq k\leq l}\Lambda_{\tw_{\leq k}}$.
Let $\mca{C}_w:=\Sub(T_{\tw})$ be the full subcategory of $\mod\Lambda$ which consists of submodule of finite direct sums of $\Lambda_{w}$.
It is known that $\Lambda_w$ and $\mca{C}_w$ do not depend on a choice of reduced expression and $\mca{C}_w=\Sub(\Lambda_w)$.
The description in \cite{GLS:KacMoody} is obtained by taking dual.
\begin{prop}[{\cite{BIRSc} \cite{GLS:adaptable}}]
\textup{(1)}
The additive category $\mca{C}_w$ is a Frobenius category and its stable category $\underline{\mca{C}_w}$ is a $\Hom$-finite $2$-Calabi-Yau category.

\textup{(2)}
If $\tw$ satisfies the condition that the subsets $\{k ; i_k=i\}$ are not empty for any $i\in I$,
then there are precisely $n=\# Q_0$ indecomposable projective-injective object in $\mca{C}_w$.

\textup{(3)}
$T_{\tw}$ is a cluster-tilting object of $\underline{\mca{C}_w}$.
\end{prop}

By using this category, they gave following additive categorification of 
the coordinate ring of unipotent subgroup $\mbb{C}[N(w)]$ as follows:

For $\tw\in R(w)$, we set the Gabriel quiver $\Gamma_{\tw}:=\End_{\mca{C}_w}(T_{\tw})$ of cluster-tilting object.
We consider this quiver as a quiver with frozen vertices.
For $\Gamma_{\tw}$, we associate a cluster algebra with coefficient $\mscr{A}(\mca{C}_w)$
in the sense of \cite{FZ:cluster4}.
In the their additive categorification, the generating function of Euler number of quiver flag varieties of module of preprojecitve algebra
which is called GLS's $\varphi$ map has important role.

\begin{theo}[{\cite{GLS:KacMoody}}]\label{theo:GLS_additive}
\textup{(1)}
The algebras $\mbb{C}[N(w)]$ of the coordinate ring of unipotent subgroup have a cluster structure.
For each $\tw\in R(w)$, the image of $\varphi$-map
$\{\varphi(\Lambda_{i_1}), \cdots, \varphi(\Lambda_{i_r})\}$ provides an initial cluster of these cluster algebra structure and it can be considered as a generalized minors.
We have isomorphism of algebras

\[\Phi_{\tw}\colon \mscr{A}(\Gamma_{\tw})\simeq \mbb{C}[N(w)].\]

\textup{(2)} The subset $\mca{S}^*(w):=\mca{S}^*\cap \mbb{C}[N(w)]$ forms a $\mbb{C}$-basis of $\mbb{C}[N(w)]$

\textup{(3)} The set of cluster monomials $\mscr{M}(\Gamma_{\tw})$ of the cluster algebra $\mscr{A}(\Gamma_{\tw})$
is contained in the dual semicanonical base $\mca{S}^*(w)$

\textup{(4)} The subalgebra which is generated by projective-injective objects of $\mca{C}_w$
forms the coefficient of $\mscr{A}(\Gamma_{\tw})$ 
\end{theo}

\end{NB}
\subsection{Cluster algebra and the dual canonical basis}
Our main result is to give a set up of a quantum analogue of Gei\ss-Leclerc-Schr\"{o}er's results:

\begin{enumerate}
\item 
The dual canonical basis is compatible with the quantum unipotent subgroup $\mca{O}_q[N(w)]$ which is a quantum analogue of $\mbb{C}[N(w)]$,
that is $\mbf{B}^{\up}(w):=\mbf{B}^{\up}\cap \mca{O}_q[N(w)]$ forms a $\mbb{Q}(q)$-basis of $\mca{O}_q[N(w)]$.
(See \thmref{theo:unipgp}.)

\item
Quantum flag minors are mutually $q$-commuting and their monomials are contained in the dual canonical basis up to some $q$-shifts.
Here quantum flag minors are defined as certain matrix coefficients with respect to extremal vectors in integrable highest weight modules.
(See \thmref{theo:seed}.)
\item
The ``$q$-center''  of $\mca{O}_q[N(w)]$ is generated by some of the quantum flag minors.
Moreover any dual canonical basis element in $\mbf{B}^{\up}(w)$ can be factored into the product of an element in the ``$q$-center'' of $\mca{O}_q[N(w)]$
and an ``interval-free'' element.
(See \thmref{theo:unip}.)
\end{enumerate}
When $\mfr{g}$ is of type ADE, Caldero proved the above results in a series of works \cite{Cal:unity, Cal:adapted, Cal:qflag} (see also \cite[6.3]{BuaMar}).
($\mca{O}_{q}[N(w)]$ is denoted by $\U_q(\mfr{n}_w)$ in \cite{Cal:qflag}.)
We generalize them to an arbitary symmetrizable Kac-Moody Lie algebra.
Key tools are the Poincar\'{e}-Birkhoff-Witt basis of $\mca{O}_q[N(w)]$ and the crystal structures.
They are already used by Caldero, but the author cannot follow several claims, and give a self-contained proof in this paper.

\subsection{Quantization conjectures and its consequences}

The above properties (1), (2), (3) can be thought as a part of structures of a quantum cluster algebra.
The corresponding properties of the ``classical limit'' $\mbb{C}[N(w)]$ were shown in \cite{GLS:KacMoody}
if the dual canonical basis is replaced by the dual semicanoncial basis.
We conjecture that remaining structures of a quantum cluster algebra exist on $\mca{O}_{q}[N(w)]$ as in \cite{GLS:KacMoody}.
Let $\mca{O}_{q}[N(w)]_{\mca{A}}$ be the integral form defined by the dual canonical basis $\mbf{B}^{\up}(w)$ where $\mca{A}=\mbb{Q}[q^{\pm}]$.
\begin{conj}[Quantization conjecture]\label{conj:qconj}
\textup{(1)}
We take a reduced expression $\tw=(i_{1}, \cdots, i_{l})$ of the Weyl group element $w$,
then we have an isomorphism of algebras
\[\Phi_{\tw}\colon \mscr{A}^{q}(\Gamma_{\tw}, \Lambda_{\tw})\otimes_{\mbb{Z}[q^{\pm}]}\mbb{Q}[q^{\pm}]\simeq \mca{O}_{q}[N(w)]_{\mca{A}},\]
which sends the initial seed to the quantum flag minors $\{\Delta^q_{s_{i_1}\cdots s_{i_k}\varpi_{i_k}, \varpi_{i_k}}\}_{1\leq k\leq l}$,
defined as matrix coefficients of certain extremal vectors associated with $\tw$,
where  $\Gamma_{\tw}$ is the frozen quiver in \cite{BFZ:cluster3} and \cite{GLS:KacMoody} and $\Lambda_{\tw}$ is the compatible pair in \cite[\S 10.3]{BZ:qcluster}.

(2) Under this isomorphism, the quantum cluster monomials  of 
$\mscr{A}^q(\Gamma_{\tw}, \Lambda_{\tw})$ are contained in the dual canonical base $\mbf{B}^{\up}(w)$ up to some $q$-shifts.
\end{conj}
Let $\mca{A}\to \mbb{C}$ be the algebra homomorphism defined by $q\mapsto 1$.
If we specialize Conjecture \ref{conj:qconj} to $q=1$, we obtain the following ``weak'' conjecture.

\begin{conj}[Weak quantization conjecture]\label{conj:weakqconj}
\textup{(1)}
Let $\tw$ be as above,
We have an isomorphism of algebras
\[\Phi_{\tw}\colon \mscr{A}(\Gamma_{\tw})\otimes_{\mbb{Z}}\mbb{C}\simeq \mbb{C}[N(w)],\]
which sends the initial seed to the specialized quantum flag minors $\{\Delta_{s_{i_1}\cdots s_{i_k}\varpi_{i_k}, \varpi_{i_k}}\}_{1\leq k\leq l}$,
where  $\Gamma_{\tw}$ is the frozen quiver as above.

(2) Under this isomorphism, the cluster monomials  of 
$\mbb{C}[N(w)]$ are contained in the specialized dual canonical base $\mbf{B}^{\up}(w)$ at $q=1$.
\end{conj}
Some parts of Conjecture \ref{conj:qconj} were shown for $A_{2}, A_{3}, A_{4}$ cases with $w=w_0$ in \cite{BZ:string} and \cite[\S 12]{GLS:semican1} and $A_1^{(1)}$ with $w=c^2$ in \cite{Lampe}.

The definition of the quantum cluster algebra $\mscr{A}^{q}(\Gamma_{\tw}, \Lambda_{\tw})$ will not be explained.
So we explain the meaning of this conjecture as properties of the dual canonical basis without referring to the axiom of a quantum cluster algebra \cite{BZ:qcluster}.

An element $x\in \mbf{B}^{\up}\setminus \{1\}$ is called \emph{prime} if it does not have a non-trivial factorization
$x=q^{N}x_{1}x_{2}$ with $x_{1}, x_{2}\in \mbf{B}^{\up}$ and $N\in \mbb{Z}$.
A subset $\mbf{x}=\{x_{1}, \cdots, x_{l}\}\subset \mbf{B}^{\up}$
is called \emph{strongly compatible} if for any $m_{1}, \cdots, m_{l}\in \mbb{Z}_{\geq 0}$, the monomial $x_{1}^{m_{1}}\cdots x_{l}^{m_{l}}\in q^{\mbb{Z}}\mbf{B}^{\up}$,
that is $x_{1}^{m_{1}}\cdots x_{l}^{m_{l}}$ 
is contained in the dual canonical basis $\mbf{B}^{\up}$ up to some $q$-shifts.
In particular, $x$ is contained in a compatible family, then it satisfies $x^{m}\in q^{\mbb{Z}}\mbf{B}^{\up}$ for any $m\geq 1$.
A strongly compatible subset $\mbf{x}=\{x_{1}, \cdots, x_{l}\}$ is called \emph{maximal} in $\mbf{B}^{\up}(w)$
if $y\in \mbf{B}^{\up}(w)$ satisfies $yx_{i}\in q^{\mbb{Z}}\mbf{B}^{\up}(w)$ for any $x_{i}$, then there exists $m_{1}, \cdots, m_{l}$ and $N$ such  
$y=q^{N}x_{1}^{m_{1}}\cdots x_{l}^{m_{l}}$.

Our quantization conjecture means that there are lots of maximal strongly compatible subsets of $\mbf{B}^{\up}(w)$,
constructed recursively from $\{\Delta^{q}_{s_{i_{1}}\cdots s_{i_{k}}\varpi_{i_{k}}, \varpi_{i_{k}}}\}_{1\leq k\leq l}$.
For example, for  finite type $\mfr{g}$ with $w=c^{2}$ for a (bipartite) Coxeter element $c$,
it is expected that the dual canonical basis $\mbf{B}^{\up}(w)$ is covered by the (finite) union of the maximal compatible families.
But the union is not the whole $\mbf{B}^{\up}(w)$ in general.

Our quantization conjecture implies several conjectures on (quantum) cluster algebras.
Let us spell out  a few.

If $\mfr{g}$ is symmetric, we have the positivity result for the dual canonical base by the construction of 
\cite{Lus:quiver}.
This implies the positivity conjecture for the quantum cluster algebras $\mscr{A}^{q}(\Gamma_{\tw}, \Lambda_{\tw})$,
stating that cluster monomials are Laurent polynomials with positive coefficients in $q$ and cluster variables of any seed.
This conjecture is known only special cases:
\begin{itemize}
\item cluster algebras of finite type \cite{FZ:cluster2},
\item cluster algebras with bipartite seeds \cite{Nak:cluster},
\item cluster algebras coming from triangulated surfaces \cite{MSW:positivity},
\item acyclic cluster algebras at the initial seed \cite{Qin}.
\end{itemize}

In fact, these results apply only to cluster algebras, not quantum ones except \cite{Qin}.
Thus we have much stronger positivity.

The quantization conjecture also provides us a \emph{monoidal categorification} of $\mbb{C}[N(w)]$
in the sense of Hernandez-Leclerc \cite{HerLec}.
It roughly says that there is a monoidal abelian category $\mfr{N}(w)$ whose complexified Grothendieck ring $K_{0}(\mfr{N}(w))\otimes_{\mbb{Z}}\mbb{C}$
has the cluster algebra structure of $\mbb{C}[N(w)]$ so that the cluster monomials are classes of simple objects.
If the weak quantizaton conjecture is true (and $\mfr{g}$
 is symmetric), the category $\mfr{N}(w)$ is given as
the category of finite dimensional modules of the (equivariant) $\Ext$ algebras of the simple (equivariant) perverse sheaves belonging to $\mbf{B}^{\up}(w)$.
Thanks to \cite{VV:KLR}, $\mfr{N}(w)$ is also considered as the extension-closed subcategory of the module category of Khovanov-Lauda-Rouquier's algebra
\cite{KhoLau:I, KhoLau:II, Rouq:2KM}
consisting of finite dimensional modules whose composition factors are contained in $\mbf{B}^{\up}(w)$.
\begin{NB}More details will be explained in \subsecref{subsec:monoidal}.\end{NB}

When $\mfr{g}$ is symmetric,
Gei\ss, Leclerc and Schr\"{o}er  conjecture that certain dual semicanonical basis elements are specialization of the corresponding dual canonical basis elements.
This is called the \emph{open orbit conjecture}.
\begin{NB}
This property conjecturally holds for cluster monomials, and is true if the weak quantization conjecture is true.
Conversely the open orbit conjecture implies the weak quantization conjecture.
\end{NB}
This class of the dual semicanonical basis element contains all the cluster monomials.
(Conjecturally it exactly consists of the cluster monomials \cite[Conjecture II 5.3]{BIRSc}.)
The open orbit conjecture for the cluster monomials is equivalent to the weak quantization conjecture.
\begin{NB}Note also that this weak form is enough to provide the mononidal categorification explained before.\end{NB}
\begin{NB}More details will be explained in \subsecref{subsec:openorbit}.\end{NB}

This paper is organized as follows.
In \secref{sec:QEA}, we give a review  the quantized enveloping algebra and its canonical basis.
In \secref{sec:upper}, we give a review the dual canonical basis $\mbf{B}^{\up}$ and  its multiplicative properties.
In \secref{sec:unip}, we define the  quantum unipotent subgroup and prove its compatibility with the dual canonical basis.
In \secref{sec:Demazure}, we define the quantum closed unipotent cell and study its relationship with the quantum unipotent subgroup.
In \secref{sec:qflag}, we give quantum flag minors and prove their multiplicative properties.
\begin{NB}In \secref{sec:conj}, we explain the quantization conjecture and its consequences in more details. \end{NB}
\subsection*{Acknowledgement}
The author is grateful to his advisor Hiraku Nakajima for his valuable comments and his sincere encouragement.
\section{Preliminaries: Quantized enveloping algebras and the canonical bases}\label{sec:QEA}
We briefly recall the definition of the quantized enveloping algebra and its canonical base in this section.

\subsection{Definition of $\Uq(\mfr{g})$}
\subsubsection{}
A \emph{root datum} consists of
\begin{enumerate}
\item $\mfr{h}$: a finite-dimensional $\mbb{Q}$-vector space,
\item a finite index set $I$, 
\item $P\subset \mfr{h}^*$: a lattice (weight lattice),
\item $P^\vee=\Hom_\mbb{Z}(P,\mbb{Z})$ with natural pairing $\bracket{~, ~}\colon P\otimes P^\vee\to \mbb{Z}$,
\item $\alpha_i\in P$~for $i\in I$ (simple roots),
\item $h_i\in P^\vee$~for $i\in I$ (simple coroots),
\item $(\cdot ,\cdot)$ a $\mbb{Q}$-valued symmetric bilinear form on $\mfr{h}^*$
\end{enumerate}
satisfying following conditions:
\begin{aenumerate}
\item $\bracket{h_i,\lambda}=2(\alpha_i,\lambda)/(\alpha_i,\alpha_i)$ for $i\in I$ and $\lambda\in P$,
\item $a_{ij}=\bracket{h_i,\alpha_j}=2(\alpha_i,\alpha_j)/(\alpha_i,\alpha_i)$
gives a symmetrizable generalized Cartan matrix, i.e.,
$\bracket{h_i,\alpha_i}=2$,
and $\bracket{h_i,\alpha_j}\in \mbb{Z}_{\leq 0}$ and
$\bracket{h_i,\alpha_j}=0 \Leftrightarrow \bracket{h_j,\alpha_i}=0$ for $i\neq j$,
\item $(\alpha_i,\alpha_i)\in 2\mbb{Z}_{>0}$, i.e. $d_i:=(\alpha_i, \alpha_i)/2\in \mbb{Z}_{>0}$,
\item $\set{\alpha_i}_{i\in I}$ are linearly independent.
\end{aenumerate}
We call $(I, \mfr{h}, (~, ~))$ a \emph{Cartan datum}.
Let $Q=\bigoplus_{i\in I}{\mbb{Z}}\alpha_i\subset P$ be the root lattice.
Let $Q_{\pm}=\pm \sum_{i\in I}\mbf{Z}_{\geq 0}\alpha_i$.
Fot $\xi=\sum_{i\in I}\xi_i\alpha_i\in Q$, we define $\tr(\xi)=\sum_{i\in I}\xi_i$.
And we assume that there exists $\varpi_i\in P$  such that $\bracket{h_i, \varpi_{j}}=\delta_{i, j}$ for any $i,j \in I$.
We call $\varpi_{i}$ the \emph{fundamental weight} corrsponding to $i\in I$.
We say $\lambda\in P$ is \emph{dominant} if $\bracket{h_i, \lambda}\geq 0$ for any $i\in I$ and denote by $P_+$
the set of dominant integral weights.
We denote by $\overline{P}:=\bigoplus_{i\in I}\mbb{Z}\varpi_{i}$ and $\overline{P}_+:=\overline{P}\cap P_+=\bigoplus_{i\in I}\mbb{Z}_{\geq 0}\varpi_i$.
\begin{NB2}
We note that the generalized Cartan matrix $a_{ij}=\bracket{h_i, \alpha_j}$ is symmetrizable,
i.e. $d_ia_{ij}=d_ja_{ji}$.
\end{NB2}%
\subsubsection{}
Let $(I, \mfr{h}, (~, ~))$ be a Cartan datum.
Let ${\mfr{g}}$ be the symmetrizable Kac-Moody Lie algebra corresponding to  the generalized Cartan matrix $A=(a_{ij})$
with the Cartan subalgebra ${\mfr{h}}$,
i.e.,
$\mfr{g}$ is the Lie algebra generated by $\{h ; h\in \mfr{h}\}$, $e_i$, and $f_i$
~$(i\in I)$ with the following relations:
\begin{renumerate}
	\item $[h, h']=0$ for $h, h'\in \mfr{h}$,
	\item $[h, e_i]=\bracket{h,\alpha_i}e_i$, $[h, f_i]=-\bracket{h, \alpha_i}f_i$,
	\item $[e_i, f_j]=\delta_{ij}h_i$, and 
	\item $(\mathrm{ad} e_i)^{1-\bracket{h_i, \alpha_j}}e_j=(\mathrm{ad} f_i)^{1-\bracket{h_i, \alpha_j}}f_j=0$ for $i\neq j$.
\end{renumerate}
We denote the Lie subalgebra generated by $\{f_i\}_{i\in I}$ by $\mfr{n}$.
\subsubsection{}
Suppose a root datum is given.
\begin{NB}
We fix a positive integer $d$ such that $(\alpha_i, \alpha_i)/2\in d^{-1}\mbb{Z}$ for any $i\in I$. 
We introduce an indeterminate $q_s=q^{1/d}$ and set $q=(q_s)^d$.
\end{NB}
We introduce an indeterminate $q$.
For $i\in I$, we set $q_i=q^{(\alpha_i, \alpha_i)/2}$.
For $\xi=\sum_{i\in I}\xi_i\alpha_i \in Q$, we set $q_\xi:=\prod_{i\in I}(q_i)^{\xi_i}=q^{(\xi, \rho)}$, where $\rho$ is the sum of all fundamental weights.
\begin{NB2}
	$(\rho, \alpha_i)=\frac{(\alpha_i, \alpha_i)}{2}\bracket{h_i, \rho}=1$
	If $(I, (\cdot, \cdot))$ is symmetric, $q_i=q$ for any $i\in I$ and $q_\xi=(q)^{\tr(\xi)}$
\end{NB2}
We define $\mbb{Q}$-subalgebras $\mca{A}_0$, $\mca{A}_\infty$ and $\mca{A}$ of $\Qq$ by 
\begin{align*}
	\mca{A}_0&:=\{f\in \Qq; f \text{~is regular at~} q=0\}, \\
	\mca{A}_\infty&:=\{f\in \Qq; f \text{~is regular at~} q=\infty\}, \\
	\mca{A}&:=\mbb{Q}[q^\pm].
\end{align*}
\subsubsection{}
The \emph{quantized enveloping algebra} $\Uq(\mfr{g})$ associated with a root datum is the $\Qq$-algebra 
generated by $e_i,f_i~(i\in I)$, $q^h~(h\in d^{-1}P^*)$ with the following relations:
\begin{renumerate}
\item $q^0=1,q^{h}q^{h'}=q^{h+h'}$,
\item $q^he_{i}q^{-h}=q^{\bracket{h,\alpha_i}}e_i, q^hf_{i}q^{-h}=q^{-\bracket{h,\alpha_i}}f_i$,
\item $e_if_j-f_je_j=\delta_{ij}{(t_i-t_i^{-1})}/{(q_i-q_i^{-1})}$,
\item $\displaystyle \sum_{k=0}^{1-a_{ij}}(-1)^ke_i^{(k)}e_je_i^{(1-a_{ij}-k)}=%
\sum_{k=0}^{1-a_{ij}}(-1)^kf_i^{(k)}f_jf_i^{(1-a_{ij}-k)}=0$~($q$-Serre relations),
\end{renumerate}
where 
$t_i=q^{\frac{(\alpha_i, \alpha_i)}{2}h_i}$,
$[n]_i=(q_i^n-q_i^{-n})/(q_i-q_i^{-1})$,
$[n]_i!=[n]_i[n-1]_i\cdots [1]_i$ for $n>0$ and
$[0]!=1$,
$e_i^{(k)}=e_i^k/[k]_i!, f_i^{(k)}=f_i^k/[k]_i!$
for $i\in I$ and $k\in {\mbb{Z}}_{\geq 0}$.
\subsubsection{}
Let $\Uq^+(\mfr{g})$ (resp.\ $\Uq^-(\mfr{g})$) be the $\Qq$-subalgebra of $\Uq(\mfr{g})$
generated by $e_i$ (resp.\ $f_i$) for $i\in I$.
Then we have the triangular decomposition
\[\Uq(\mfr{g})\simeq \Uq^-(\mfr{g})\otimes_{\Qq}\Qq[P^\vee]\otimes_{\Qq}\Uq^+(\mfr{g}),\]
where $\Qq[P^\vee]$ is the group algebra over $\Qq$, i.e., $\bigoplus_{h\in P^\vee}\Qq q^h$.
\subsubsection{}
For $\xi\in Q$, we define its \emph{root space} $\Uq(\mfr{g})_\xi$ by 
\[\Uq(\mfr{g})_\xi=\{x\in \Uq(\mfr{g})| q^hxq^{-h}=q^{\bracket{h, \xi}}x\text{~for any~} h\in P^*\}.\]
Then we have the root space decomposition
\[\Uq^\pm(\mfr{g})=\bigoplus_{\xi\in Q_\pm}\Uq(\mfr{g})_\xi.\]
An element $x\in \Uq(\mfr{g})$ is \emph{homogenous} if $x\in \Uq(\mfr{g})_\xi$ for some $\xi\in Q$,
and we set $\wt(x)=\xi$.
\subsubsection{}\label{subsubsec:A-form}
Let $\Uq^-(\mfr{g})_{\mca{A}}$ be the $\mca{A}$-subalgebra of $\Uq^-(\mfr{g})$ 
generated by $f_i^{(k)}$ for $i\in I$ and $k\in {\mbb{Z}}_{\geq 0}$.
Let $(\Uq^-(\mfr{g})_{\mca{A}})_\xi:= \Uq^-(\mfr{g})_{\mca{A}} \cap \Uq^-(\mfr{g})_\xi$.
We have
\[\Uq^-(\mfr{g})_{\mca{A}}=\bigoplus_{\xi \in Q_-}(\Uq^-(\mfr{g})_{\mca{A}})_\xi.\]
\begin{NB}
In fact, we can consider $\mbf{A}=\mbb{Z}[q^\pm]$-from which is generated by $\{f_i^{(n)}\}_{i\in I, n\geq 1}$.
For the theory of balanced triple, $\mca{A}$-form is enough to consider.
\end{NB}
\subsubsection{}
We define a $\Qq$-algebra anti-involution $*\colon \Uq(\mfr{g})\to \Uq(\mfr{g})$ by
\begin{align}
*(e_i)=e_i, &&  *(f_i)=f_i,&&*(q^h)=q^{-h}. \label{eq:bar}
\end{align}
We call this the \emph{$*$-involution}.

We define a $\mbb{Q}$-algebra automorphism $\overline{\phantom{x}}\colon \Uq(\mfr{g})\to \Uq(\mfr{g})$ by
\begin{align}
\overline{e_i}=e_i, &&  \overline{f_i}=f_i, && \overline{q}=q^{-1}, &&\overline{q^h}=q^{-h}. \label{eq:star}
\end{align}
We call this the \emph{bar involution}.

We remark that these two involutions preserve $\Uq^+(\mfr{g})$ and $\Uq^-(\mfr{g})$,
and we have $\overline{\phantom{x}}\circ *=*\circ \overline{\phantom{x}}$.
\subsubsection{}\label{sec:coproduct}
In this article, we choose the coproduct on $\Uq(\mfr{g})$ following \cite{Kas:crystal}:
\begin{subequations}
\begin{align}
	\Delta(q^h)&=q^h\otimes q^h, \\
	\Delta(e_i)&=e_i\otimes t_i^{-1}+1\otimes e_i, \\
	\Delta(f_i)&=f_i\otimes 1+t_i\otimes f_i.
\end{align}
\end{subequations}
\begin{NB}
\begin{rem}
In our construction, we define $\Delta$ by  $\Delta_-$ as in \cite{Kas:crystal}.
This coproduct is different from \cite[Lemma 3.1.4]{Lus:intro}, although there is a simple relation between them.
For more details, see \cite[1.4]{Kas:crystal}.
\end{rem}
\end{NB}
\subsubsection{}\label{sec:Lus_form}
We introduce Lusztig's $\Qq$-valued symmetric nondegenerate bilinear form $(~, ~)_L$ on $\Uq^-(\mfr{g})$.
We first define a $\Qq$-algebra structure on $\Uq^-(\mfr{g})\otimes\Uq^-(\mfr{g})$ by
\[(x_1\otimes y_1)(x_2\otimes y_2)=q^{-(\wt(x_2), \wt(y_1))}x_1x_2\otimes y_1y_2,\]
where $x_i, y_i~(i=1, 2)$ are homoegenous elements.

Let 
$r\colon \Uq^-(\mfr{g})\to \Uq^-(\mfr{g})\otimes \Uq^-(\mfr{g})$ be a $\Qq$-algebra homomorphism defined by
\[r(f_i)=f_i\otimes 1+1\otimes f_i~(i\in I).\]
We call this the \emph{twisted coproduct}.

By \cite[1.2.5]{Lus:intro}, the algebra $\Uq^-(\mfr{g})$ has a  unique nondegenerate $\Qq$-valued symmetric bilinear form
$(~, ~)_L\colon \Uq^-(\mfr{g})\times \Uq^-(\mfr{g})\to \Qq$ which satisfies
\begin{subequations}
\begin{align}
	&(1, 1)_L=1, \\
	&(f_i, f_j)_L=\frac{\delta_{i, j}}{1-q_i^{2}}, \\
	&(x, yy')_L=(r(x), y\otimes y')_L, \\
	&(xx', y)_L=(x\otimes x', r(y))_L,
\end{align} 
\end{subequations}
where the form on $\Uq^-(\mfr{g})\otimes \Uq^-(\mfr{g})$ is defined by $(x_1\otimes y_1, x_2\otimes y_2)_L=(x_1, x_2)_L(y_1, y_2)_L$.
\subsubsection{}
The relation between the coproduct $\Delta$ and the twisted coproduct $r$ is given as follows:
\begin{lem}\label{lem:rcoproduct}
	For homogenous $x\in \Uq^-(\mfr{g})_\xi$, we have 
	\beq \Delta(x)=\sum x_{(1)}t_{-\wt(x_{(2)})}\otimes x_{(2)},\eeq
	where
	$r(x)=\sum x_{(1)}\otimes x_{(2)}$, 
	$t_\xi=q^{\nu(\xi)}$, and 
	$\nu(\xi)=\sum_i \frac{(\alpha_i, \alpha_i)}{2}\xi_ih_i$ for $\xi=\sum \xi_i\alpha_i\in Q$.
\end{lem}
\begin{NB}
\begin{proof}
First we remark, the following relation. 
For $\xi\in Q, x\in \Uq(\mfr{g})_{\wt(x)}$
\[t_\xi x t_{-\xi}=q^{\bracket{\nu(\xi), \wt(x)}}x=q^{(\xi, \wt(x))}\]

We prove the induction on $|\tr(\xi)|$ of $x\in \Uq^-(\mfr{g})$.
For the generators $F_i$, the statement is true.
For $x=x'x'$ with $x'\in \Uq^-(\mfr{g})_{\xi'}, x''\in \Uq^-(\mfr{g})_{\xi''}$, we assume that the statement holds.
Let $r(x')=\sum x'_{(1)}\otimes x'_{(2)}, r(x'')=\sum x''_{(1)}\otimes x''_{(2)}$.
Then we have 
\begin{align*}
	r(x)=r(x')r(x'')=&(\sum x'_{(1)}\otimes x'_{(2)})\cdot (\sum x''_{(1)}\otimes x''_{(2)}) \\
	=&\sum \sum q^{-(\wt(x'_{(2)}), \wt(x''_{(1)}))} x'_{(1)}x'_{(1)}\otimes x'_{(2)}x'_{(2)}\\
	\Delta(x)=&\Delta(x')\Delta(x') \\
	=&(\sum x'_{(1)}t_{-\wt(x'_{(2)})}\otimes x'_{(2)})\cdot (\sum x''_{(1)}t_{-\wt(x''_{(2)})}\otimes x''_{(2)}) \\
	=&\sum \sum x'_{(1)}t_{-\wt(x'_{(2)})}x''_{(1)}t_{-\wt(x''_{(2)})}\otimes x'_{(2)} x''_{(2)}) \\
	=&\sum \sum x'_{(1)}t_{-\wt(x'_{(2)})}x''_{(1)}t_{\wt(x'_{(2)})}t_{-\wt(x'_{(2)})}t_{-\wt(x''_{(2)})}\otimes x'_{(2)} x''_{(2)}) \\
	=&\sum \sum q^{-(\wt(x'_{(2)}), \wt(x''_{(1)}))}x'_{(1)}x''_{(1)}t_{-\wt(x'_{(2)}x''_{(2)})}\otimes x'_{(2)} x''_{(2)}).
\end{align*}
Then we have obtained the claim.
\end{proof}
\end{NB}
\subsubsection{}
For $i\in I$, we define the unique $\Qq$-linear map  ${_ir}\colon \Uq^-\to \Uq^-$ (resp.\ $r_i\colon \Uq^-\to \Uq^-$) given by
${_ir}(1)=0, {_ir}(f_j)=\delta_{i, j}$ (resp.\ $r_i(1)=0, r_i(f_j)=\delta_{i, j}$) for any $i, j\in I$ and 
\begin{subequations}
\begin{align}
	{_ir}(xy)&={_ir}(x)y+q^{-(\wt x, \alpha_i)}x{_ir}(y), \label{eq:ir}\\
	r_i(xy)&=q^{-(\wt y, \alpha_i)}r_i(x)y+xr_i(y)\label{eq:ri}
\end{align}
\end{subequations}
for homogenous $x, y\in \Uq^-$.
From the definition, we have 
\begin{subequations}
\begin{align}
	(f_ix, y)_L=\frac{1}{1-q_i^{2}}(x,{_ir}y)_L, \label{eq:Lformadj}\\
	(xf_i, y)_L=\frac{1}{1-q_i^{2}}(x,{r_i}y)_L.
\end{align}
\end{subequations}

\begin{NB}
\begin{rem}
	Since we use the coproduct $\Delta_-$ on $\Uq(\mfr{g})$,
	the sign of power of $q$ is different from Lusztig's original bilinear form which is defined in \cite[Chapter 1]{Lus:intro}.
	Our $r\colon \Uq^-\to \Uq^-\otimes \Uq^-$ corresponds to Lusztig's 
	$\overline{r}:=(\bar{~}\otimes \bar{~}) \circ r \circ \bar{~} \colon \Uq^-\xrightarrow{\bar{~}} \Uq^- \xrightarrow{r}\Uq^-\otimes \Uq^-\xrightarrow{\bar{~}\otimes \bar{~}} \Uq^-\otimes \Uq^-$,
	and $\{~, ~\}=\bar{~} \circ (~, ~)\circ (\bar{~}\otimes \bar{~})\colon \Uq^-\otimes \Uq^-\xrightarrow{\bar{~}\otimes \bar{~}} \Uq^-\otimes \Uq^-\xrightarrow{(~, ~)} \Qq\xrightarrow{\bar{~}} \Qq$.
	This normalization comes from the choice of the coproduct $\Delta_-$.
\end{rem}
\end{NB}

\subsection{Canonical basis of $\Uq^-(\mfr{g})$}\label{sec:lower}
In this subsection, we give a brief review of the theory of the canonical base following Kashiwara \cite{Kas:crystal, Kas:bases}.
Note that Kashiwara call it the \emph{lower global base}.
\subsubsection{}
\begin{lem}[{\cite[Lemma 3.4.1]{Kas:crystal}, \cite{Nak:CBMS}}]\label{lem:qBoson}
For $x\in \Uq^-(\mfr{g})$ and any $i\in I$, we have
\[[e_i, x]=\frac{r_i(x)t_i-t_i^{-1}{_ir}(x)}{q_i-q_i^{-1}}.\]
\end{lem}
\subsubsection{}\label{sec:Kas_form}
Kashiwara \cite[\S 3.4]{Kas:crystal} has proved that there is a unique non-degenerate symmetric bilinear form $(\cdot, \cdot)_{K}$ on $\Uq^-(\mfr{g})$
such that 
\begin{subequations}
	\begin{align}
	(f_ix, y)_K&=(x, {_ir}y)_K,\label{eq:Kformadj}\\
	(1, 1)_K&=1.
	\end{align}
\end{subequations}
\begin{lem}[{\cite[Lemma 3.4.7]{Kas:crystal}, \cite[Lemma 1.2.15]{Lus:intro}}]\label{lem:indcrystal}
For $x\in \Uq^-(\mfr{g})$ with ${_ir}(x)=0$ for any $i\in I$ and $\wt(x)\neq 0$, then we have $x=0$.
\end{lem}
\subsubsection{}
We have the following relation between Kashiwara's bilinear form $(~, ~)_K$ and Lusztig's one $(~, ~)_L$.

\begin{lem}[{\cite[2.2]{Lec:qchar}}]\label{lem:KLforms}
For homogenous $x, y\in \Uq^-(\mfr{g})_{\xi}$ with $\xi=-\sum n_i\alpha_i\in Q_-$,
we have
\[(x, y)_K=\prod_{i\in I}(1-q_i^2)^{n_i}(x, y)_L.\]
\end{lem}
This can be proved by an induction on $\wt(x)$ by using \lemref{lem:indcrystal}, \eqref{eq:Kformadj} and \eqref{eq:Lformadj}.
\begin{NB}
\begin{proof}
The proof is straightforward.
For convenience of the reader, we give a proof.
We prove by induction on $|\xi|=\sum_{i\in I}n_i$.
For $\xi=0$ case, the statement is true.
So we assume that $|\xi|>0$.
Here, by \lemref{lem:indcrystal}, we have that there exists $i\in I$ and $P'\in \Uq^-(\mfr{g})_{\xi+\alpha_i}$ such that $P=f_iP'$. Here we have 
\begin{align*}
	(x, y)_K&=(f_ix', y)_K =(x',{_ir}(y))_{K}\\
	&=\frac{1}{1-q_i^2}(x',{_ir}(y))_{L}\prod_{i\in I}(1-q_i^2)^{n_i} \\
	&=(f_ix',y)_{L}\prod_{i\in I}(1-q_i^2)^{n_i}=(x,y)_{L}\prod_{i\in I}(1-q_i^2)^{n_i}.
\end{align*}
Hence we have obtained the assertion.
\end{proof}
\end{NB}

\begin{lem}[{\cite[Lemma 1.2.8(b)]{Lus:intro}}] \label{lem:*form}
For any homogenous $x, y\in \Uq^-(\mfr{g})$,
we have
\[(x, y)_K=(x^*, y^*)_K.\]
\end{lem}
\subsubsection{}
The \emph{reduced $q$-analogue} $\mscr{B}_q(\mfr{g})$ of a symmetrizable Kac-Moody Lie algebra $\mfr{g}$
is the $\Qq$-algebra generated by ${_ir}$ and $f_i$ with the $q$-Boson relations ${_ir}f_j=q^{-(\alpha_i, \alpha_j)}f_j{_ir}+\delta_{i, j}$ for $i, j\in I$
and the $q$-Serre relations for ${_ir}$ and $f_i$ for $i\in I$.
Then $\Uq^-(\mfr{g})$ becomes a $\mscr{B}_q(\mfr{g})$-modules by \lemref{lem:qBoson}.

By the $q$-Boson relation, any element $x\in \Uq^-(\mfr{g})$ can be uniquely written as $x=\sum_{n\geq 0}f_i^{(n)}x_n$ with 
${_ir}(x_n)=0$ for any $n\geq 0$.
So we define Kashiwara's \emph{modified root operators} $\fit{i}$ and $\eit{i}$ by
\begin{align*}
	\eit{i}x&=\sum_{n\geq 1}f_i^{(n-1)}x_n,\\
	\fit{i}x&=\sum_{n\geq 0}f_i^{(n+1)}x_n.
\end{align*}
By using these operators, Kashiwara introduced the crystal basis $(\mscr{L}(\infty), \mscr{B}(\infty))$ of $\Uq^-(\mfr{g})$:
\begin{theo}[{\cite{Kas:crystal}}]
	Let 
	\begin{align*}
	\mscr{L}(\infty)&:=\sum_{l\geq 0, i_1, i_2, \cdots, i_l\in I}\mca{A}_0\fit{i_1}\cdots \fit{i_l}1\subset \Uq^-(\mfr{g}), \\
	\mscr{B}(\infty)&:=\{\fit{i_1}\cdots \fit{i_l}1 \mod q\mscr{L}(\infty); {l\geq 0, i_1, i_2, \cdots, i_l\in I}\}\subset \mscr{L}(\infty)/q\mscr{L}(\infty).
	\end{align*}
Then we have the followings:
\begin{enumerate}
	\item $\mscr{L}(\infty)$ is a free $\mca{A}_0$-module with $\Qq\otimes_{\mca{A}_0}\mscr{L}(\infty)=\Uq^-(\mfr{g})$.
	\item $\eit{i}\mscr{L}(\infty)\subset \mscr{L}(\infty)$ and $\fit{i}\mscr{L}(\infty)\subset \mscr{L}(\infty)$.
	\item $\mscr{B}(\infty)$ is a $\mbb{Q}$-basis of $\mscr{L}(\infty)/q\mscr{L}(\infty)$.
	\item $\fit{i} \colon \mscr{B}(\infty)\to \mscr{B}(\infty)$ and $\eit{i} \colon \mscr{B}(\infty)\to \mscr{B}(\infty)\cup \{0\}$.
	\item For $b\in\Binfty$ with $\eit{i}(b)\neq 0$, we have $\fit{i}\eit{i}b=b$.
\end{enumerate}
\end{theo}
We call $(\mscr{L}(\infty), \mscr{B}(\infty))$ the \emph{(lower) crystal basis} of $\Uq^-(\mfr{g})$,
and $\mscr{L}(\infty)$ the \emph{(lower) crystal lattice}.
We denote $1\mod q\mscr{L}(\infty)\in \mscr{B}(\infty)$ by $u_{\infty}$ hereafter.
For $b\in \mscr{B}(\infty)$, we set $\vep_{i}(b):=\max\{n\in \mbb{Z}_{\geq 0}; \eit{i}^nb\neq 0\}<\infty$,
and $\eit{i}^{\max}(b):=\eit{i}^{\vep_{i}(b)}b\in \mscr{B}(\infty)$.
\subsubsection{}
We have the following compatibility of the $*$-involution with the crystal lattice $\mscr{L}(\infty)$.
\begin{theo}[{\cite[Proposition 5.2.4]{Kas:crystal}, \cite[Theorem 2.1.1]{Kas:Demazure}}]
	We have 
	\begin{subequations}
	\begin{align}
		*(\mscr{L}(\infty))=\mscr{L}(\infty), \\
		*(\Binfty)=\Binfty.
	\end{align}
	\end{subequations}
\end{theo}
For $i\in I$ and $b\in \mscr{B}(\infty)$, we set
\begin{subequations}\begin{align}
\fit{i}^{*}(b):=(*\circ \fit{i}\circ *)(b), \\
\eit{i}^{*}(b):=(*\circ \eit{i}\circ *)(b).
\end{align}\end{subequations}
For $b\in \mscr{B}(\infty)$, we set $\vep_i^*(b):=\max\{n\in \mbb{Z}_{\geq 0}; \eit{i}^{*n}b\neq 0\}<\infty$.
We have $\vep_i^*(b)=\vep_i(*b)$.
\subsubsection{}
We recall some results on relationship between the crystal lattice $\mscr{L}(\infty)$ and Kashiwara's form $(\cdot, \cdot)_K$.

\begin{prop}[{\cite[Proposition 5.1.2]{Kas:crystal}}]\label{prop:crylat}
	We have
	 \[(\mscr{L}(\infty), \mscr{L}(\infty))_K\subset \mca{A}_0.\]
	Therefore  the $\mbb{Q}$-valued inner product on 
	$\mscr{L}(\infty)/q\mscr{L}(\infty)$
	given by $(\cdot, \cdot)|_{q=0}$ is well-defined, which we denote by $(\cdot, \cdot)_0$.
	Then we have the following properties:
	\begin{enumerate}
	\item $(\eit{i}u, u')_0=(u, \fit{i}u')_0$ for $u, u'\in \mscr{L}(\infty)/q\mscr{L}(\infty)$,
	\item $\mscr{B}(\infty)\subset \mscr{L}(\infty)/q\mscr{L}(\infty)$ is an orthonormal basis with respect to $(~, ~)_0$.
	\end{enumerate}
	Moreover we have 
	\begin{equation}\mscr{L}(\infty)=\{x \in \Uq^-(\mfr{g}); (x, \mscr{L}(\infty))_K\subset \mca{A}_0\},\label{eq:crylat}\end{equation}
	that is the crystal lattice $\mscr{L}(\infty)$ is a self-dual lattice with respect to $(\cdot, \cdot)_K$.
\end{prop}
\begin{NB}
By the above properties, we have the following characterization of $\mscr{L}(\infty)$ 
\begin{prop}[{\cite[Proposition 5.1.3]{Kas:crystal}}]
We have 
\begin{equation}\mscr{L}(\infty)=\{x\in \Uq^-(\mfr{g});  (x, x)_K\in \mca{A}_0\}.\end{equation}
\end{prop}
\end{NB}
\subsubsection{}
Let $\overline{\phantom{A}}\colon \Qq\to \Qq$ be the $\mbb{Q}$-algebra involution sending $q$ to $q^{-1}$.
Let $V$ be a vector space over $\Qq$,
$\mscr{L}_0$ be an $\mca{A}_0$-submodule of $V$,
$\mscr{L}_\infty$ be an $\mca{A}_\infty$-submodule of $V$, 
and $V_{\mca{A}}$ be an $\mca{A}$-submodule of $V$.
We set $E:=\mscr{L}_0\cap \mscr{L}_\infty \cap V_{\mca{A}}$.
\begin{defn}
We say that a triple $(\mscr{L}_0, \mscr{L}_{\infty}, V_{\mca{A}})$ is \emph{balanced} if 
each $\mscr{L}_0, \mscr{L}_{\infty}$, and $V_{\mca{A}}$ generates $V$ as $\Qq$-vector space and if one of the following equivalent conditions is satisfied
\begin{enumerate}
	\item $E\to \mscr{L}_0/q\mscr{L}_0$ is an isomorphism,
	\item $E\to \mscr{L}_\infty/q^{-1}\mscr{L}_\infty$ is an isomorphism,
	\item $(\mscr{L}_0\cap V_{\mca{A}})\+ (q^{-1}\mscr{L}_\infty\cap V_{\mca{A}})\to V_{\mca{A}}$
	is an isomorphism,
	\item $\mca{A}_0\otimes_{\mbb{Q}} E\to \mscr{L}_0$,
	$\mca{A}_\infty\otimes_{\mbb{Q}} E\to \mscr{L}_\infty$,
	$\mca{A}\otimes_{\mbb{Q}} E\to V_{\mca{A}}$,
	and $\Qq\otimes_{\mbb{Q}}E\to V$ are isomorphisms.
\end{enumerate}
\end{defn}

\begin{theo}[{\cite[Theorem 6]{Kas:crystal}}]
	The triple $(\mscr{L}(\infty), \overline{\mscr{L}(\infty)}, \Uq^-(\mfr{g})_{\mca{A}})$
	is balanced.
\end{theo}
Let $\Glow\colon \mscr{L}(\infty)/q\mscr{L}(\infty)\to E:=\mscr{L}(\infty)\cap \overline{\mscr{L}(\infty)} \cap \Uq^-(\mfr{g})_{\mca{A}}$
be the inverse of $E\xrightarrow{\sim} \mscr{L}(\infty)/q\mscr{L}(\infty)$.
Then $\{\Glow(b) ; b\in \mscr{B}(\infty)\}$ forms an $\mca{A}$-basis of $\Uq^-(\mfr{g})_{\mca{A}}$.
This basis is called the \emph{canonical basis} of $\Uq^-(\mfr{g})$.
Using this characterization, we obtain the following compatibility of the canonical basis and the $*$-involution.
\begin{prop}\label{prop:*lower}
We have 
\[*\Glow(b)=\Glow(*b).\]
\end{prop}

\subsubsection{}
For integrable highest weight modules, we can define the (lower) crystal basis and the canonical basis of them as for $\Uq^-(\mfr{g})$,
see \cite[Theorem 2, Theorem 6]{Kas:crystal} for more details.
Let $M$ be an integrable $\U_{q}(\mfr{g})$-module and $M=\bigoplus_{\lambda\in P}M_\lambda$ be its weight decomposition.
For $u\in \Ker(e_i)\cap M_\lambda$ and $0\leq n\leq \bracket{h_i, \lambda}$, we define 
Kashiwara's modified operators or (lower) crystal operators $\eit{i}^{\low}$ and $\fit{i}^{\low}$ by
\begin{align*}
\eit{i}^{\low}(f_i^{(n)}u)&=f_i^{(n-1)}u,\\
\fit{i}^{\low}(f_i^{(n)}u)&=f_i^{(n+1)}u.
\end{align*}
Here we understand $f_{i}^{(-1)}u$ and $f_i^{(\bracket{h_i, \lambda}+1)}u$ as $0$.
Note that we denote the operators $\fit{i}$ and $\eit{i}$ in \cite[2.2]{Kas:crystal} by $\fit{i}^{\low}$ and $\eit{i}^{\low}$ following \cite{Kas:global}.

Let $\lambda\in P_+$ and $V(\lambda)$ be the integrable highest weight $\Uq(\mfr{g})$-module generated by a highest weight vector $u_\lambda$ of weight $\lambda$.
Let $\mscr{L}^{\low}(\lambda)$ be the $\mca{A}_{0}$-submodule spanned by $\fit{i_{1}}^{\low}\cdots \fit{i_{l}}^{\low}u_{\lambda}$.
Let $\mscr{B}^{\low}(\lambda)$ be the subset of $\mscr{L}^{\low}(\lambda)/q_{s}\mscr{L}^{\low}(\lambda)$ consisting of
the non-zero vectors of the form $\fit{i_{1}}^{\low}\cdots \fit{i_{l}}^{\low}u_{\lambda}$, that is 
\begin{align*}
\mscr{L}^{\low}(\lambda)&:=\sum \mca{A}_{0}\fit{i_{1}}^{\low}\cdots \fit{i_{l}}^{\low}u_{\lambda}\subset V(\lambda), \\
\mscr{B}^{\low}(\lambda)&:=\{\fit{i_{1}}^{\low}\cdots \fit{i_{l}}^{\low}u_{\lambda} \mod q\mscr{L}^{\low}(\lambda)\}\setminus \{0\}\subset \mscr{L}^{\low}(\lambda)/q_{s}\mscr{L}^{\low}(\lambda).
\end{align*}

\begin{theo}[{\cite[Theorem 2]{Kas:crystal}}]

\textup{(1)} $\mscr{L}^{\low}(\lambda)$ is a free $\mca{A}_{0}$-submodule with $\Qq\otimes_{\mca{A}_{0}}\mscr{L}^{\low}(\lambda)\simeq V(\lambda)$
and  $\mscr{L}^{\low}(\lambda)=\bigoplus_{\mu\in P}\mscr{L}^{\low}(\lambda)_{\mu}$ where $\mscr{L}^{\low}(\lambda)_{\mu}=\mscr{L}^{\low}(\lambda)\cap M_{\mu}$.

\textup{(2)} 
$\eit{i}^{\low}\mscr{L}^{\low}(\lambda)\subset \mscr{L}^{\low}(\lambda)$
and  $\fit{i}^{\low}\mscr{L}^{\low}(\lambda)\subset \mscr{L}^{\low}(\lambda)$.

\textup{(3)}
$\mscr{B}^{\low}(\lambda)$ is a $\mbb{Q}$-basis of $\mscr{L}^{\low}(\lambda)/q\mscr{L}^{\low}(\lambda)$
and $\mscr{B}^{\low}(\lambda)=\bigsqcup_{\mu\in P}\mscr{B}^{\low}(\lambda)_{\mu}$
where $\mscr{B}^{\low}(\lambda)_{\mu}=\mscr{B}^{\low}(\lambda)\cap \mscr{L}^{\low}(\lambda)_{\mu}/q\mscr{L}^{\low}(\lambda)_{\mu}$.

\textup{(4)}
For $i\in I$, we have $\eit{i}\mscr{B}(\lambda)\subset \mscr{B}(\lambda)\cup \{0\}$ and 
$\fit{i}\mscr{B}(\lambda)\subset \mscr{B}(\lambda)\cup \{0\}$.

\textup{(5)}
For $b, b'\in \mscr{B}^{\low}(\lambda)$, 
$b'=\fit{i}^{\low}b$ is equivalent to $b=\eit{i}^{\low}b'$.

\end{theo}
We call $(\mscr{L}^{\low}(\lambda), \mscr{B}^{\low}(\lambda))$ the \emph{lower crystal basis} of $V(\lambda)$,
and $\mscr{L}^{\low}(\lambda)$ the \emph{lower crystal lattice}.

Let $\overline{\phantom{x}}$ be the bar-involution defined by $\overline{Pu_\lambda}=\overline{P}u_\lambda$.
Set $V(\lambda)_{\mca{A}}:=\Uq^-(\mfr{g})_{\mca{A}}u_{\lambda}$. 
\begin{theo}[{\cite[Theorem 6]{Kas:crystal}}]
The triple $(\mscr{L}^{\low}(\lambda), \overline{\mscr{L}^{\low}(\lambda)}, V(\lambda)_{\mca{A}})$
is balanced.
\end{theo}

Let  $\Glow_{\lambda}$ be the inverse of 
$\Llow(\lambda)\cap \overline{\Llow(\lambda)}\cap V(\lambda)_{\mca{A}}\xrightarrow{\sim} \Llow(\lambda)/q\Llow(\lambda)$.
We call $\Glow_{\lambda}(\mscr{B}^{\low}(\lambda))$ the \emph{canonical basis} of $V(\lambda)$.
\subsubsection{}
We have a compatibility of the (lower) crystal basis of $\Uq^-(\mfr{g})$ and the integrable modules $V(\lambda)$.
Let $\pi_\lambda\colon \Uq^-(\mfr{g})\to V(\lambda)$ be the $\Uq^-(\mfr{g})$-module homomorphism
defined by $x\mapsto xu_{\lambda}$.
\begin{theo}[{\cite[Theorem 5]{Kas:crystal}}]\label{theo:grloop}
We have the following properties:
\begin{enumerate}
	\item $\pi_\lambda \mscr{L}(\infty)=\mscr{L}(\lambda)$,
	hence $\pi_\lambda$ induces a surjection homomorphism $\pibar_\lambda\colon \mscr{L}(\infty)/q\mscr{L}(\infty) \to \mscr{L}^{\low}(\lambda)/q\mscr{L}^{\low}(\lambda)$.
	\item $\pi_\lambda$ induces a bijection
	$\{b\in \mscr{B}(\infty) ; \pibar_\lambda(b)\neq 0\}\simeq \mscr{B}^{\low}(\lambda).$
	\item $\fit{i}^{\low}\circ \pibar_\lambda(b)=\pibar_\lambda\circ \fit{i}(b)$ if $\pibar_\lambda(b)\neq 0$.
	\item $\eit{i}^{\low}\circ \pibar_\lambda(b)=\pibar_\lambda\circ \eit{i}(b)$ if $\eit{i}\circ \pibar_\lambda(b)\neq 0$.
\end{enumerate}
\end{theo}
We denote the inverse of the bijection $\pi_\lambda$ by $\jbar_\lambda$.
\subsubsection{}
We also have a compatibility of the canonical basis of $\Uq^-(\mfr{g})$ and the integrable modules $V(\lambda)$ via $\pi_\lambda$.
\begin{theo}[{\cite[7.3 Lemma 7.3.2]{Kas:crystal}}]
For $\lambda\in P_+$ and $b\in \mscr{B}(\infty)$ with $\pibar_\lambda(b)\neq 0$,
we have 
\[\Glow(b)u_\lambda=\Glow_{\lambda}(\pibar_\lambda(b)).\]
\end{theo}
\subsubsection{}
For the canonical basis, we have the following expansion of left and right multiplication with respect to $f_i^{(m)}$.
\begin{theo}[{\cite[(3.1.2)]{Kas:Demazure}}]\label{thm:multlower}
	For $b\in \mscr{B}(\infty)$, we have
	\begin{subequations}
	\begin{align}
	f_i^{(m)}\Glow(b)&=\begin{bmatrix}\vep_i(b)+m \\ m\end{bmatrix}\Glow(\fit{i}^mb)+\sum_{\vep_i(b')>\vep_i(b)+m}f_{b b'; i}^{(m)}(q)\Glow(b'),\\
	\Glow(b)f_i^{(m)}&=\begin{bmatrix}\vep^*_i(b)+m \\ m\end{bmatrix}\Glow(\fit{i}^{*m}b)+\sum_{\vep^*_i(b')>\vep^*_i(b)+m}f^{*(m)}_{b b'; i}(q)\Glow(b'),
	\end{align}
	\end{subequations}
	where $f_{b b'; i}^{(m)}(q)=\overline{f_{b b'; i}^{(m)}(q)}, f^{*(m)}_{b b'; i}(q)=\overline{f^{*(m)}_{b b'; i}(q)}\in \mca{A}$.
\end{theo}
As a corollary of the above theorem, we have the following compatibilities  of 
the right and left ideals $f_{i}^{n}\Uq^{-}(\mfr{g})$ and $\Uq^{-}(\mfr{g})f_{i}^{n}$ with the canonical basis.
\begin{theo}[{\cite[Theorem 7]{Kas:crystal}}]\label{theo:lower_filt}
For $i\in I$ and $n\geq 1$, we have 
\begin{align*}
	f_i^n\Uq^-(\mfr{g})\cap \Uq^-(\mfr{g})_{\mca{A}}&=\bigoplus_{b\in \mscr{B}(\infty), \vep_i(b)\geq n}\mca{A}\Glow(b), \\
	\Uq^-(\mfr{g})f_i^n\cap \Uq^-(\mfr{g})_{\mca{A}}&=\bigoplus_{b\in \mscr{B}(\infty), \vep^*_i(b)\geq n}\mca{A}\Glow(b).
\end{align*}

\end{theo}

\begin{NB2}
\subsection{Lusztig's canonical basis}
For symmetric Kac-Moody case, we have a geometric construction of Lusztig's $\mbb{Z}[q^\pm]$-form
and the canonical basis $\Glow(\mscr{B}(\infty))$ of $\Uq^-(\mfr{g})$
\end{NB2}
\subsection{Abstract crystal}
The notion of a (abstract) crystal was introduced in ~\cite{Kas:Demazure}
by abstracting the crystal basis of $\Uq^-(\mfr{g})$ 
and of irreducible highest weight representations which are constructed in \cite{Kas:crystal}.
We recall it briefly.
For more detail, see \cite{Kas:bases}.
\subsubsection{}
\begin{defn}
A \emph{crystal} $\mscr{B}$ associated with a root datum is a set $\mscr{B}$ endowed 
with maps $\wt \colon \mscr{B}\to P,\vep_i,\varphi_i:\mscr{B}\to \mbb{Z}\sqcup \set{-\infty}$,
$\eit{i}, \fit{i}\colon \mscr{B}\to \mscr{B} \sqcup \set{0}~(i\in I)$
satisfying following conditions:
\begin{aenumerate}
\item $\varphi_i(b)=\vep_i(b)+\bracket{h_i,\wt(b)}$,
\item $\wt(\eit{i}b)=\wt(b)+\alpha_i$,
$\vep_i(\eit{i}b)=\vep_i(b)-1$, $\varphi_i(\eit{i}b)=\varphi_i(b)+1$,
if $\eit{i}b\in \mscr{B}$,
\item $\wt(\fit{i}b)=\wt(b)-\alpha_i$,
$\vep_i(\fit{i}b)=\vep_i(b)+1$, $\varphi_i(\fit{i}b)=\varphi_i(b)-1$,
if $\fit{i}b\in \mscr{B}$,
\item $b'=\fit{i}b \Leftrightarrow b=\eit{i} b'$ for $b,b'\in \mscr{B}$,
\item if $\varphi_i(b)=-\infty$ for $b\in \mscr{B}$,
then $\eit{i} b=\fit{i} b=0$.
\end{aenumerate}
Let $\wt_i(b)=\bracket{h_i,\wt(b)}$.
\end{defn}

\begin{defn}
For given two crystals $\mscr{B}_1,\mscr{B}_2$ and $h\in \mbb{Z}_{\geq 1}$ ,
 a map $\psi \colon \mscr{B}_1\sqcup\set{0}\to \mscr{B}_2\sqcup\set{0}$ 
 is called a \emph{morphism of amplitude $h$} of crystals from $\mscr{B}_1$ to $\mscr{B}_2$
if it satisfies the following properties for $b\in \mscr{B}_1$ and $i\in I$:
 \begin{aenumerate}
\item $\psi(0)=0$,
 \item 
 $\wt(\psi(b))=h\wt(b)$,
 $\vep_i(\psi(b))=h\vep_i(b)$,
 $\varphi_i(\psi(b))=h\varphi_i(b)$ if $\psi(b)\in \mscr{B}_2$,
 \item
 $\eit{i}^h\psi(b)=\psi(\eit{i}b)$ if  $\psi(b)\in \mscr{B}_2$, $\eit{i}b\in \mscr{B}_1$,
 \item
 $\fit{i}^h\psi(b)=\psi(\fit{i}b)$ if  $\psi(b)\in \mscr{B}_2$, $\fit{i}b\in \mscr{B}_1$.
 \end{aenumerate}
When $h=1$, it is simply called a \emph{morphism of crystal}. 
A morphism $\psi\colon\mscr{B}_1\to \mscr{B}_2$ is \emph{strict} if $\psi$ commutes with
$\eit{i},\fit{i}$ for all $i\in I$ without any restriction.
A strict morphism $\psi\colon\mscr{B}_1\to \mscr{B}_2$ is called an \emph{strict embedding} if $\psi$ is an injective map from $\mscr{B}_1\sqcup\set{0}$ to $\mscr{B}_2\sqcup\set{0}$.
\end{defn}

\begin{defn}
The \emph{tensor product} $\mscr{B}_1\otimes \mscr{B}_2$ of crystals $\mscr{B}_1$ and
$\mscr{B}_2$ is defined to be the set $\mscr{B}_1\times\mscr{B}_2$ with maps given by
\begin{subequations}
\begin{align}
   & \wt(b_1\otimes b_2) = \wt(b_1)+\wt(b_2),
\\
   & \vep_{i}(b_1\otimes b_2) 
   = \max(\vep_{i}(b_1), \vep_{i}(b_2)-\wt_{i}(b_1)),
\\
   & \vphi_{i}(b_1\otimes b_2) 
   = \max(\vphi_{i}(b_2), \vphi_{i}(b_1)+ \wt_{i}(b_2)), 
\\
   & \eit{i}(b_1\otimes b_2)
   = 
   \begin{cases}
      \eit{i}b_1 \otimes b_2 &
        \text{if $\vphi_{i}(b_1)\ge \vep_{i}(b_2)$},
   \\
      b_1\otimes \eit{i}b_2 & \text{otherwise},
   \end{cases}\label{eq:eit}
\\
   & \fit{i}(b_1\otimes b_2)
   = 
   \begin{cases}
      \fit{i}b_1 \otimes b_2 &
        \text{if $\vphi_{i}(b_1) > \vep_{i}(b_2)$},
   \\
      b_1\otimes \fit{i}b_2 & \text{otherwise}.
   \end{cases}\label{eq:fit}
\end{align}
\end{subequations}
Here $(b_1, b_2)$ is denoted by $b_1\otimes b_2$ and $0\otimes b_2$,
$b_1\otimes 0$ are identified with $0$.
\end{defn}
Iterating \eqref{eq:eit} and \eqref{eq:fit}, we obtain the followings:
\begin{subequations}
\begin{align}
\eit{i}^{n}(b_{1}\otimes b_{2})&=
\begin{cases}
\eit{i}^{n}b_{1}\otimes b_{2} & \text{~if~} \varphi_{i}(b_{1})\geq \vep_{i}(b_{2}), \\
\eit{i}^{n-\vep_{i}(b_{2})+\vphi_{i}(b_{2})}b_{1}\otimes \eit{i}^{\vep_{i}(b_{2})-\vphi_{i}(b_{1})} & \text{~if~} \vep_{i}(b_{2})\geq \vphi_{i}(b_{1})\geq \vep_{i}(b_{2})-n, \\
b_{1}\otimes \eit{i}^{n}b_{2} & \text{~if~} \vep_{i}(b_{2})-n\geq \vphi_{i}(b_{1}).
\end{cases}\label{eq:eitn}\\
\fit{i}^{n}(b_{1}\otimes b_{2})&=\begin{cases}
\fit{i}^{n}b_{1}\otimes b_{2} & \text{~if~} {\vphi_{i}(b_{1})\geq \vep_{i}(b_{2})+n},\\
\fit{i}^{\vphi_{i}(b_{1})-\vep_{i}(b_{2})}b_{1}\otimes \fit{i}^{n-\vphi_{i}(b_{1})+\vep_{i}(b_{2})}b_{2} & \text{~if~} \vep_{i}(b_{2})+n\geq \vphi_{i}(b_{1}) \geq \vep_{i}(b_{2}),\\
b_{1}\otimes \fit{i}^{n}b_{2} & \text{~if~} \vep_{i}(b_{2})\geq \vphi_{i}(b_{1}),
\end{cases}\label{eq:fitn}
\end{align}
\end{subequations}
\begin{NB}
\subsubsection{}
For crystals $\mscr{B}_{1}$, $\mscr{B}_{2}$ and $\mscr{B}_{3}$,
it is known that $(\mscr{B}_{1}\otimes \mscr{B}_{2})\otimes \mscr{B}_{3}$ is isomorphic to 
$\mscr{B}_{1}\otimes (\mscr{B}_{2}\otimes \mscr{B}_{3})$ by $(b_{1}\otimes b_{2})\otimes b_{3}\mapsto b_{1}\otimes (b_{2}\otimes b_{3})$,
see \cite[Proposition 1.3.1]{Kas:Demazure}.
We denote it by $\mscr{B}_{1}\otimes \mscr{B}_{2}\otimes \mscr{B}_{3}$.

Similarly, we can define $\mscr{B}_{1}\otimes \cdots \otimes \mscr{B}_{r}$ for crystal $\mscr{B}_{1}, \cdots, \mscr{B}_{r}$.
For the crystal $\mscr{B}_{1}\otimes \cdots \otimes \mscr{B}_{r}$, we can determine its structure inductively as follows.
For $b_{1}\otimes \cdots \otimes b_{r}\in \mscr{B}_{1}\otimes \cdots \otimes \mscr{B}_{r}$,
we set 
\begin{subequations}
\begin{align}
\vep_{i}^{k}(b_{1}\otimes \cdots \otimes b_{r})&:=\vep_{i}(b_{k})-\sum_{k':k'<k}\wt_{i}(b_{k'}),\\
\vphi_{i}^{k}(b_{1}\otimes \cdots \otimes b_{r})&:=\vphi_{i}(b_{k})+\sum_{k':k'>k}\wt_{i}(b_{k'}).
\end{align}
\end{subequations}
Then we can show the following by induction: 
\begin{subequations}
\begin{align}
\wt(b_{1}\otimes \cdots \otimes b_{r})&=\sum_{k} \wt(b_{k}), \\
\vep_{i}(b_{1}\otimes \cdots \otimes b_{r})&=\max_{k}\vep_{i}^{k}(b_{1}\otimes \cdots \otimes b_{r}), \label{eq:vepmult}\\
\vphi_{i}(b_{1}\otimes \cdots \otimes b_{r})&=\max_{k}\vphi_{i}^{k}(b_{1}\otimes \cdots \otimes b_{r}),\\
\eit{i}(b_{1}\otimes \cdots \otimes b_{r})&=b_{1}\otimes \cdots \otimes \eit{i}b_{k}\otimes \cdots \otimes b_{r}\\
&\text{~where~}k=\min\{k'; \vep_{i}^{k'}(b_{1}\otimes \cdots \otimes b_{r})=\vep_{i}(b_{1}\otimes \cdots \otimes b_{r})\},\\
\fit{i}(b_{1}\otimes \cdots \otimes b_{r})&=b_{1}\otimes \cdots \otimes \fit{i}b_{k}\otimes \cdots \otimes b_{r} \\
&\text{~where~}k=\max\{k'; \vphi_{i}^{k'}(b_{1}\otimes \cdots \otimes b_{r})=\vphi_{i}(b_{1}\otimes \cdots \otimes b_{r})\}.
\end{align}
\end{subequations}
\end{NB}
%
\subsubsection{}
The (lower) crystal basis $\mscr{B}(\infty)$ of $\Uq^-(\mfr{g})$ is an example of an abstract crystal.
The same is true for $\mscr{B}^{\low}(\lambda)$ of $V(\lambda)$ for $\lambda\in P_+$.
We may also write $\mscr{B}(\lambda)$ instead of $\mscr{B}^{\low}(\lambda)$, when it is considered as an abstract crystal.

\begin{exa}
For $i\in I$, let $\mscr{B}_i=\{b_i(n); n\in \mbb{Z}\}$.
We can endow it with a structure of the abstract crystal by 
$\wt(b_i(n))=n\alpha_i$, $\vep_i(b_i(n))=-n$, $\varphi_i(b_i(n))=n$, 
$\vep_j(b_i(n))=\varphi_{j}(b_i(n))=-\infty$, for $j\neq i$, and 
\begin{align*}
	\fit{j}b_i(n)=\begin{cases}b_i(n-1) & \text{~~if $j=i$}, \\ 0 &\text{~~if $j\neq i$},\end{cases} \\
	\eit{j}b_i(n)=\begin{cases}b_i(n+1) & \text{~~if $j=i$}, \\ 0 & \text{~~if $j\neq i$}.\end{cases}
\end{align*}
\end{exa}

\subsubsection{}
For the crystal $\mscr{B}(\infty)$, we have the following strict embedding.
\begin{theo}[{\cite[Theorem 2.2.1]{Kas:Demazure}}]\label{theo:Kasemb}
\textup{(1)}
For each $i\in I$, there exits a strict embedding $\Psi_i\colon \mscr{B}(\infty)\to \mscr{B}(\infty)\otimes \mscr{B}_i$
which satisfies $\Psi_i(u_{\infty})=u_{\infty}\otimes b_i$.

\textup{(2)}
If $\Psi_i(b)=b'\otimes \fit{i}^nb_i$, we have 
\begin{align*}
	\Psi_i(\eit{i}^*b)&=\begin{cases}b'\otimes \fit{i}^{n-1}b_i & \text{~if~} n\geq 1, \\ 0 & \text{~if~} n=0, \end{cases}\\
	\Psi_i(\fit{i}^*b)&=b'\otimes \fit{i}^{n+1}b_i.
\end{align*}

\textup{(3)}
$\Im \Psi_{i}=\{b'\otimes \fit{i}^nb_i; \vep^*_i(b')=0, n\geq 0\}.$
\end{theo}
By the above theorem, we have $\Psi_i(b)=\eit{i}^{*\max}b\otimes \fit{i}^{\vep^*_i(b)}b_i$.
For a sequence $(i_1, i_2, \cdots, i_r)\in I^r$,
we have a strict embedding 
\[\Psi_{(i_1, i_2, \cdots, i_r)}:=(\Psi_{i_r}\otimes \cdots \otimes 1)\cdots (\Psi_{i_2}\otimes 1)\Psi_{i_1}\colon \mscr{B}(\infty)\hookrightarrow %
\mscr{B}(\infty)\otimes \mscr{B}_{i_r}\otimes \cdots \otimes \mscr{B}_{i_1}.\]

\subsubsection{}
For $m\geq 1$, we have the following crystal morphism of amplitude $m$ which is called \emph{inflation of order} $m$ in \cite[Definition 8.1.4]{Kas:bases}.
\begin{prop}[{\cite[Proposition 8.1.3]{Kas:bases}, \cite[Proposition 3.2]{NaiSag}}]
\textup{(1)}
For $m\in \mbb{Z}_{\geq 1}$,
there exists a unique crystal morphism $S_m\colon \mscr{B}(\infty)\to \mscr{B}(\infty)$ of amplitude $m$ satisfying
\begin{align*}
&\wt(S_mb)=m\wt(b), \quad \vep_i(S_mb)=m\vep_i(b), \quad \varphi_i(S_mb)=m\varphi_i(b),\\
&S_m(\eit{i}b)=\eit{i}^mS_m(b), \quad S_m(\fit{i}b)=\fit{i}^mS_m(b),\\
&S_m(u_\infty)=u_{\infty}.
\end{align*}

\textup{(2)}
Let $b\in \mscr{B}(\infty)$. Then we have $(*\circ S_m)(b)=(S_m\circ *)(b)$.
In particular, for any $b\in \mscr{B}(\infty)$, we have 
\begin{align*}
&\vep^*_i(S_mb)=m\vep^*_i(b), \quad \varphi^*_i(S_mb)=m\varphi^*_i(b),\\
&S_m(\eit{i}^*b)=\eit{i}^{*m}S_m(b), \quad S_m(\fit{i}^*b)=\fit{i}^{*m}S_m(b).\\
\end{align*}
\end{prop}
\section{The dual canonical basis}\label{sec:upper}
\subsection{}
In this subsection, we recall the definition of the dual canonical basis and its charactrization in terms of the \emph{dual bar involution} $\sigma$ with a balanced triple.
We define $\mbf{B}^{\up}\subset \Uq^-(\mfr{g})$ by the dual basis of $\mbf{B}$ under the 
Kashiwara's  bilinear form $(~, ~)_K$.
We define the \emph{dual bar involution} $\sigma \colon \Uq^-(\mfr{g})\to \Uq^-(\mfr{g})$ so that
 \[(\sigma(x), y)_K=\overline{(x, \overline{y})_K}\] 
holds for any $y$ (\cite[10.2]{BZ:qcluster}).
This is well-defined since $(\cdot, \cdot)_{K}$ is non-degenerate.
By its definition, we have $\sigma(x)=x$ for $x\in \mbf{B}^{\up}$
and this is a $\mbb{Q}$-linear involutive automorphism of $\Uq^{-}(\mfr{g})$ which satisfies $\sigma(fx)=\overline{f}\sigma(x)$
for any $f\in \Qq$ and $x\in \Uq^{-}(\mfr{g})$.

\begin{NB}
By its construction, the dual bar involution $\sigma$ depend on the non-degenerate bilinear forms.
We can also define the dual bar-involution by using Lusztig's bilinear form $(\cdot, \cdot)_{L}$.
If we fix the dual bar-involution by $(\sigma_{L}(x) ,y)_{L}=\overline{(x, \overline{y})_{L}}$, the normalization in \propref{prop:dualbar} (3)
will be changed.
To remove the sign in the formula \propref{prop:dualbar} (3), we choose Kashiwara's bilinear form $(\cdot, \cdot)_{K}$ in this paper.
We have a following isomorphism.
\begin{prop}
	Set $\Phi_K\colon \Uq^-(\mfr{g})\simeq \Uq^-(\mfr{g})_{\gr}^*$ by $\bracket{\Phi_K(P), Q}:=(P, Q)_K$ for 
	$P, Q\in \Uq^-(\mfr{g})_{\xi}$ with $\xi=-\sum_{i\in I} n_i\alpha_i\in Q_-$.
	Then we have $\Phi_K(P')*\Phi_K(P'')=\Phi_K(P'P'')$,
	here $\bracket{\phi'*\phi'', P)}:=\bracket{\phi'\otimes \phi'', r(P)}$.
\end{prop}
\begin{proof}
For convenience of the reader, we give a proof.
\begin{align*}
	&\bracket{\Phi_K(x')\Phi_K(x''), y} \\
	:=&\bracket{\Phi_K(x')\otimes \Phi_K(x''), r(y)}=\bracket{\Phi_L(x')\otimes \Phi_L(x''), \sum y_{(1)}\otimes y_{(2)}} \\
	=&\sum\bracket{\Phi_K(x'), y_{(1)}}\bracket{\Phi_K(x''), y_{(2)}}=(x'x'', r(y))_K ~(\text{by \lemref{lem:KLforms}})\\ 
	=&\bracket{\Phi_K(x'x''), y}.
\end{align*}
Thus we obtained the assertion.
\end{proof}
\end{NB}
\subsubsection{}
For $\xi=\sum \xi_i\alpha_i\in Q$, we define
\beq N(\xi):=\frac{1}{2}\left((\xi, \xi)+\sum \xi_i (\alpha_i, \alpha_i)\right)=%
\frac{1}{2}\left( (\xi, \xi)+ 2(\xi, \rho) \right). \eeq
We have $N(-\alpha_i)=0$ for any $i\in I$ and $N(\xi+\eta)=N(\xi)+N(\eta)+(\xi, \eta)$ for any $\xi, \eta\in Q$.
\begin{prop}\label{prop:dualbar}
We assume $x, y\in \Uq^{-}(\mfr{g})$ are homogenous.

\textup{(1)}
If $r(x)=\sum x_{(1)}\otimes x_{(2)}$,
we have 
\[r(\overline{x})=\sum q^{-(\wt x_{(1)}, \wt x_{(2)})}\overline{x_{(2)}}\otimes \overline{x_{(1)}}.\]
\textup{(2)}
We set $\{x, y\}_K:=\overline{(\overline{x}, \overline{y})_K}$,
then we have 
\[\{x, y\}_K=q^{N(\wt x)}(x, *y)_K.\]
\textup{(3)}
We have
\[\sigma(x)=q^{N(\wt x)}(*\circ \overline{\phantom{x}})(x).\]
\end{prop}

\begin{proof}
For convenience of the reader, we give a proof.

\textup{(1)}
We follow the argument in \cite[1.2.10]{Lus:intro}.
For generators of $\Uq^-(\mfr{g})$, we have $r(f_i)=f_i\otimes 1+1\otimes f_i=r(\overline{f_i})$.
We prove the assertion by the induction on $\wt$, so we assume that \textup{(1)} holds for homogenous $x', x''$ and show that it holds also for $x=x'x''$.
First we write $r(x')=\sum x'_{(1)}\otimes x'_{(2)}$ and $r(x'')=\sum x''_{(1)}\otimes x''_{(2)}$.
By assumption, we have 
$r(\overline{x'})=\sum q^{-(\wt x'_{(1)}, \wt x'_{(2)})}\overline{x'_{(2)}}\otimes \overline{x'_{(1)}}$ and 
$r(\overline{x''})=\sum q^{-(\wt x''_{(1)}, \wt x''_{(2)})}\overline{x''_{(2)}}\otimes \overline{x''_{(1)}}$.
We have 
$r(x'x'')=r(x')r(x'')=\sum q^{-(\wt x'_{(2)}, \wt x''_{(1)})}x'_1x''_1\otimes x'_2x''_2$ and
\begin{align*}
	r(\overline{x'})r(\overline{x''})=
	\sum q^{-(\wt x''_{(1)}, \wt x''_{(2)})-(\wt x'_{(1)}, \wt x'_{(2)})-(\wt x'_{(1)}, \wt x''_{(2)})}\overline{x'_{(2)}x''_{(2)}}\otimes \overline{x'_{(1)}x''_{(1)}}.
\end{align*}
Then the assertion follows.

\textup{(2)}
We follow the argument in \cite[Lemma 1.2.11 (2)]{Lus:intro}.
For the generators, we have 
 $\{f_i, f_i\}_K=(f_i, f_i)_K=q^{N(\wt f_i)}(f_i, f_i)_K$.
 
We prove the assertion by the induction on $\tr(\wt x)=\tr(\wt y)$.
We prove that \textup{(2)} holds for $y=y'y''$ and for any $x$ assuming it holds for $y', y''$.
First we write $r(x)=\sum x_{(1)}\otimes x_{(2)}$ with $x_{(1)}$ and $x_{(2)}$ homogenous.
We have
\begin{align*}
	\;&(\overline{x}, \overline{y})_K\\
	=\;&(r(\overline{x}), \overline{y'}\otimes \overline{y''})=\sum q^{-(\wt x_{(1)}, \wt x_{(2)})}(\overline{x_{(2)}}\otimes \overline{x_{(1)}}, \overline{y'}\otimes \overline{y''})_K\\
	=\;&\sum q^{-(\wt x_{(1)}, \wt x_{(2)})}(\overline{x_{(2)}}, \overline{y'})_K(\overline{x_{(1)}}, \overline{y''})_K\\
	=\;&\sum q^{-(\wt x_{(1)}, \wt x_{(2)})-N(\wt x_{(1)})-N(\wt x_{(2)})}\overline{(x_{(2)}, *y')_K(x_{(1)}, *y'')_K}\\
	=\;&\sum q^{-N(\wt x)}\overline{(x_{(2)}, *y')_K(x_{(1)}, *y'')_K},
\end{align*}
where we have used the induction hypothesis in the fourth equality.
On the other hand, we have 
\begin{align*}
\;&q^{-N(\wt x)}\overline{(x, *y)_K} \\
=\;&q^{-N(\wt x)}\overline{(r(x), *y''\otimes *y')_K}\\
=\;&q^{-N(\wt x)}\sum \overline{(x_{(1)}\otimes x_{(2)}, *y''\otimes *y')_K}.
\end{align*}
Hence we obtain the assertion.

\textup{(3)}
We have $(\sigma(x), y)=\overline{(x, \overline{y})}=q^{N(\wt(x))}(\overline{x}, *y)=q^{N(\wt(x))}((*\circ \overline{\phantom{x}})(x), y)$,
where we used \lemref{lem:*form}.
Since this holds for any $y$, assertion follows.
\end{proof}
\subsubsection{}
By its construction, we have a characterization of the dual canonical basis $\mbf{B}^{\up}$
in terms of the dual bar involution $\sigma$ and the crystal lattice $\mscr{L}(\infty)$ of $\Uq^-(\mfr{g})$.
We note that $\mscr{L}(\infty)$ is a self-dual $\mca{A}_0$ lattice, see \eqref{eq:crylat},
and hence we do not need to the introduce the dual lattice of $\mscr{L}(\infty)$.

\begin{prop}
We set 
\[\Uq^-(\mfr{g})_{\mca{A}}^{\up}:=\{x\in \Uq^-(\mfr{g}); (x, \Uq^-(\mfr{g})_{\mca{A}})_K\subset \mca{A}\}.\]
Then $(\mscr{L}(\infty), \sigma(\mscr{L}(\infty)), \Uq^-(\mfr{g})_{\mca{A}}^{\up})$
is a balanced triple for the dual canonical basis $\mbf{B}^{\up}$.
\end{prop}
Here we have the following isomorphism of $\mbb{Q}$-vector spaces:
\[\mscr{L}(\infty)\cap \sigma(\mscr{L}(\infty))\cap \Uq^-(\mfr{g})_{\mca{A}}^{\up}\xrightarrow{\sim} \mscr{L}(\infty)/q\mscr{L}(\infty).\]
We denote its inverse by $\Gup$.
Then we have $\mbf{B}^{\up}=\Gup(\mscr{B}(\infty))$.
\subsubsection{}
The above proposition gives a characterization of the dual canonical basis elements.
\begin{cor}[{\cite[Proposition 16]{LNT}}]\label{prop:dualtriple}
A homogenous $x\in \Uq^-(\mfr{g})_{\mca{A}}^{\up}\cap \mscr{L}(\infty) \cap \sigma(\mscr{L}(\infty))$ is an element of the dual canonical basis if and only if 
there exists $b\in \mscr{B}(\infty)$ such that 
\begin{align*}
	\sigma(x)&=x, \\
	x&\equiv b \mod q\mscr{L}(\infty). 
\end{align*}
\end{cor}
\subsubsection{}
We have the following compatibility of the dual canonical basis and the $*$-involution from \propref{prop:*lower}.
\begin{lem}\label{lem:*upper}
	For $b\in \Binfty$,
	we have 
	\[\Gup(*b)=*\Gup(b).\]
\end{lem}
\begin{NB}
\begin{proof}
	The proof is straightforward.
	\begin{align*}
	(*\Gup(b), \Glow(b'))_K
	&=(\Gup(b), *\Glow(b'))_K &&\text{by \lemref{lem:*form}}\\
	&=(\Gup(b), \Glow(*b'))_K &&\text{by \propref{prop:*lower}}\\
	&=\delta_{b, *b'}=\delta_{*b, b'} &
	\end{align*}
\end{proof}
\end{NB}
\subsection{Compatible subset}
In this subsection, we introduce the concept of \emph{compatible subsets} of $\mscr{B}(\infty)$.
Roughly speaking, they are closed under the multiplication up to $q$-shifts, considered as subsets of the dual canonical basis $\mbf{B}^{\up}$.
\subsubsection{}
By \propref{prop:dualbar} (3), we obtain the following.
\begin{prop}
For homogenous $x_1, x_2\in \mbf{U}_q^-(\mfr{g})$, we have
\begin{equation}
\sigma(x_1x_2)=q^{(\wt x_1, \wt x_2)}\sigma(x_2)\sigma(x_1).\label{eq:dualbar}
\end{equation}
\end{prop}
Then we obtain the following property.
\begin{cor}\label{cor:inverse}
Let $b_1, b_2\in \mscr{B}(\infty)$
and consider the following expansion
\[\Gup(b_1)\Gup(b_2)=\sum_{\wt(b)=\wt(b_1)+\wt(b_2)}d_{b_1, b_2}^b(q)\Gup(b).\]
Then we have 
$d_{b_1, b_2}^b(q^{-1})=q^{(\wt b_1, \wt b_2)}d_{b_2, b_1}^b(q).$
In particular, if we have $\Gup(b_1)\Gup(b_2)=q^N\Gup(b_1\circledast b_2)$ for $b_1\circledast b_2\in \mscr{B}(\infty)$ and $N\in \mbf{Z}$,
then we have 
\[\Gup(b_1)\Gup(b_2)=q^{-N-(\wt b_1, \wt b_2)}\Gup(b_2)\Gup(b_1).\]
\end{cor}
\begin{proof}
The first statement is clear from \eqref{eq:dualbar}.
Suppose that $\Gup(b_1)\Gup(b_2)=q^N\Gup(b_1\circledast b_2)$ for $b_1\circledast b_2\in \mscr{B}(\infty)$ and $N\in \mbf{Z}$,
i.e.,  $d_{b_1, b_2}^{b}(q)=q^N\delta_{b, b_1\circledast b_2}$ for $b_1\circledast b_2\in \mscr{B}(\infty)$.
Then we have 
\[d_{b_2, b_1}^b(q)=q^{-(\wt b_1, \wt b_2)}d_{b_1, b_2}^b(q^{-1})=q^{-(\wt b_1, \wt b_2)}q^{-N}\delta_{b, b_1*b_2}.\]
This implies that if $\Gup(b_1)$ and $\Gup(b_2)$ satisfies 
$d_{b_1, b_2}^b(q)=q^N\delta_{b, b_1\circledast b_2}$ for some $b_1 \circledast b_2\in \mscr{B}(\infty)$,
then $\Gup(b_1)$ and $\Gup(b_2)$ $q$-commutes.
\end{proof}
Motivated by this corollary, we introduce the following definition.

\begin{defn}
\textup{(1)}
We denote $x\simeq y$ for $x, y\in\Uq^-(\mfr{g})$ if there exists $N\in \mbb{Z}$ such that 
$x=q^Ny$.

\textup{(2)}
For $b_1, b_2\in \mscr{B}(\infty)$,
we call $b_1$ and $b_2$ are \emph{multiplicative} or \emph{compatible} if
 there exists a unique $b_1\circledast b_2\in \mscr{B}(\infty)$ such that 
\[\Gup(b_1\circledast b_2)\simeq \Gup(b_1)\Gup(b_2).\]

By \corref{cor:inverse} this condition is independent of the order on $b_{1}$ and $b_{2}$.
We write $b_{1}\bot b_{2}$ when this holds.

\textup{(3)}
Elements $b_1, \cdots, b_l\in \mscr{B}(\infty)$ are called \emph{compatible}  if the following holds
\[\Gup(b_1)\cdots \Gup(b_l)\simeq \Gup(b_1 \circledast\cdots \circledast b_l)\] 
for a unique $b_1\circledast\cdots\circledast b_l\in \mscr{B}(\infty)$.
This condition is also independent of the ordering on $b_{1}, \cdots, b_{l}$.

\textup{(4)}
An element $b\in \mscr{B}(\infty)$ is called \emph{real} if $\Gup(b)\Gup(b)\simeq \Gup(b^{[2]})$
for a unique $b^{[2]}\in \mscr{B}(\infty)$,
that is $b\bot b$.

\textup{(5)}
An element $b\in \mscr{B}(\infty)$ is called \emph{strongly real} if $\Gup(b)^m\simeq \Gup(b^{[m]})$
for a unique $b^{[m]}\in \mscr{B}(\infty)$ for any $m$,
that is $\underbrace{b, \cdots ,b}_{m \text{~times~}}$ is compatible for any $m$.

\textup{(6)}
Elements $b_{1}, \cdots, b_{l}$ is called \emph{strongly compatible}
if for any $m_{1}, \cdots, m_{l}\in \mbb{Z}_{\geq 0}$,
the product $\Gup(b_{1})^{m_{1}}\cdots \Gup(b_{l})^{m_{l}}\simeq \Gup(b_{1}^{[m_{1}]}\circledast\cdots \circledast b_{l}^{[m_{l}]})$
for a unique $b_{1}^{[m_{1}]}\circledast\cdots \circledast b_{l}^{[m_{l}]}\in \mscr{B}(\infty)$.

\end{defn}
\begin{NB}
As a corollary of the simple product conjecture, 
we can prove that 
\begin{enumerate}
\item real implies strongly real
\item compatible sequence of real elements are strongly compatible
\end{enumerate}
So, strongly real and strongly compatible are not so strong.
\end{NB}
\begin{rem}
For $b_1, b_2\in \mscr{B}(\infty)$, 
we say a pair $(b_1, b_{2})$ is \emph{quasi-commutative} if we have $\Gup(b_{1})\Gup(b_{2})\simeq \Gup(b_{2})\Gup(b_{1})$ following 
\cite{BZ:string} and \cite{Rei:mult}.
In \cite{BZ:string}, Berenstein and Zelevinsky conjectured that the quasi-commutativity and compatibility is equivalent.
The above corollary proves the Reineke's result that the compatibility for $b_{1}$ and $b_{2}$ implies the quasi-commutativity
generalizes Reineke's result from when $\mfr{g}$ is symmetric to arbitrary symmetrizable $\mfr{g}$.
\end{rem}
\begin{rem}
The relation $b_1\bot b_2$ is \emph{not} an equivalence relation,
as there exists $b$ which does not satisfies $b\bot b$.
In particular, such elements are counter-examples for Berenstein-Zelevinsky's conjecture in \cite{BZ:string}.
In \cite{Lec:imaginary}, Leclerc said that $b$ is \emph{real} if $b\bot b$ and \emph{imaginary} otherwise.
He constructed examples of imaginary elements in \cite{Lec:imaginary}.
Other examples closely related to this paper are given in \cite[Corollary 4.4]{Lampe}. 
\end {rem}
\begin{rem}
Even if $b_1\bot b_2$, we can not determine $N$ in $d_{b_1, b_2}^b=q^{N}\delta_{b, b_1 \circledast b_2}$
in terms of weight of $b_1, b_2$.
In \secref{sec:unip}, we have its explicit form in terms of the Lusztig data of $b$ and $b'$ associated with a reduced expression $\tw$.
\begin{NB}For the case which is treated in \secref{sec:Demazure}, we have explicit form in terms of string data?\end{NB}
\end{rem}

\begin{cor}
\textup{(1)}
If $b_1\bot b_2$, then $*b_1\bot *b_2$.

\textup{(2)}
If $b$ is real, then $*b$ is also real. 
\end{cor}
\subsubsection{}
Let  ${_ir}^{(m)}:={_ir}^m/[m]!$ and ${r_i}^{(m)}:={r_i}^m/[m]!$.
These operators are adjoint of the left and right multiplications of $f_{i}^{(m)}$ by \eqref{eq:Kformadj}.
From \thmref{thm:multlower}, we get the following expansions
for the actions of ${_ir}^{(m)}:={_ir}^m/[m]!$ and ${r_i}^{(m)}:={r_i}^m/[m]!$.
\begin{theo}\label{theo:multupper}
	For $b\in \mscr{B}(\infty)$, we have
	\begin{subequations}
	\begin{align}
	{_ir}^{(m)}\Gup(b)&=\begin{bmatrix}\vep_i(b) \\ m\end{bmatrix}\Gup(\eit{i}^mb)+\sum_{\vep_i(b')<\vep_i(b)-m}E_{b b'; i}^{(m)}(q)\Gup(b'),\\
	{r_i}^{(m)}\Gup(b)&=\begin{bmatrix}\vep^*_i(b)\\ m\end{bmatrix}\Gup(\eit{i}^{*m}b)+\sum_{\vep^*_i(b')<\vep^*_i(b)-m}E^{*(m)}_{b b'; i}(q)\Gup(b'),
	\end{align}
	\end{subequations}
	where $E_{b b'; i}^{(m)}(q)=\overline{E_{b b'; i}^{(m)}(q)}, E^{*(m)}_{b b'; i}(q)=\overline{E^{*(m)}_{b b'; i}(q)}\in \mca{A}$.
\end{theo}

As a special case, we have the following result.
\begin{cor}[{\cite[Lemma 5.1.1.]{Kas:global}}]\label{cor:string1}
Let $b\in \mscr{B}(\infty)$ with $\vep_i(b)=c$ (resp.\ $\vep_i^*(b)=c$).
Then we have ${_ir}^{(c)}\Gup(b)=\Gup(\eit{i}^{\max}b)$ (resp.\ ${r_i}^{(c)}\Gup(b)=\Gup(\eit{i}^{*\max}b)$).
\end{cor}
By the above corollary and \eqref{eq:ir}, we obtain the following result.
\begin{cor}[{\cite[Lemma 2.1]{Rei:mult}}]\label{cor:string2}
For $b_1, b_2\in \mscr{B}(\infty)$
with 
\[\Gup(b_1)\Gup(b_2)=\sum d_{b_1, b_2}^b(q)\Gup(b),\]
we have $\vep_i(b)\leq \vep_{i}(b_1)+\vep_{i}(b_2)$ for any $i\in I$ if $d_{b_1, b_2}^b(q)\neq 0$.
An equality holds at least one $b$.
\end{cor}
If fact, we can prove prove $d^{b}_{b_{1}, b_{2}}(q)=0$ if $\vep_{i}(b)>\vep_{i}(b_{1})+\vep_{i}(b_{2})$
by the descending induction on $\vep_{i}(b)$.
In particular, the positivity of $d^{b}_{b_{1}, b_{2}}$, assumed in \cite{Rei:mult}, is not used in the proof.
The second assertion follows from 
\begin{equation}
\begin{split}
\;&{_{i}r}^{(\vep_{i}(b_{1})+\vep_{i}(b_{2}))}(\Gup(b_{1})\Gup(b_{2})) \\
=\;&q^{N}\Gup(\eit{i}^{\max}b_{1})\Gup(\eit{i}^{\max}b_{2})\\
=\;&\sum_{\vep_{i}(b_{1})+\vep_{i}(b_{2})=\vep_{i}(b)}q^{N}d^{b}_{b_{1}, b_{2}}(q)\Gup(\eit{i}^{\max}b)
\end{split}\label{eq:string3}
\end{equation}
for some $N\in \mbb{Z}$.
\begin{NB}
In particular, we have 

\begin{align}
_ir^{\vep_i(b_1)+\vep_i(b_2)}(\Gup(b_1)\Gup(b_2))
=&q^{\vep_{i}(b_1)}q^{-\vep_{i}(b_2)(\wt b_1, \alpha_i)}\Gup(\eit{i}^{\max}b_1)\Gup(\eit{i}^{\max}b_2) \\
=&\sum_{\vep_i(b)=\vep_i(b_1)+\vep_i(b_2)}d_{b_1, b_2}^b(q)\Gup(\eit{i}^{\max}b).\label{eq:stringest}
\end{align}

\end{NB}
\begin{NB}
\begin{proof}
For the completeness, we give a proof.
Firstly, we have 
\[_ir^{(m)}(xy)=\sum_{m'+m''=m}q_i^{m'}q^{-m''(\wt x, \alpha_i)}~\!_ir^{(m')}(x)_ir^{(m'')}(y).\]
By the multiplication formula, we obtain:
\[_ir^{\vep_{i}(b_1)+\vep_{i}(b_2)}(\Gup(b_1)\Gup(b_2))=%
q_i^{\vep_{i}(b_1)}q^{-\vep_{i}(b_2)(\wt b_1, \alpha_i)}\Gup(\eit{i}^{\max}b_1)\Gup(\eit{i}^{\max}b_2),\]
and $_ir^{\vep_{i}(b_1)+\vep_{i}(b_2)+1}(\Gup(b_1)\Gup(b_2))=0$.
On the other hand, if we consider the expansion 
$\Gup(b_1)\Gup(b_2)=\sum d_{b_1, b_2}^b(q) \Gup(b)$ and apply $_ir^{\vep_{i}(b_1)+\vep_{i}(b_2)+1}$.
Here we have 
\begin{align*}
0= &_ir^{\vep_{i}(b_1)+\vep_{i}(b_2)+1}(\sum d_{b_1, b_2}^b(q) \Gup(b))\\
 =&\sum d_{b_1, b_2}^b(q) \left(\begin{bmatrix}\vep_i(b) \\ \vep_{i}(b_1)+\vep_{i}(b_2)+1\end{bmatrix}
 \Gup(\eit{i}^{\vep_i(b_1)+\vep_i(b_2)+1}b)+\sum_{\vep_i(b')<\vep(b)-(\vep_{i}(b_1)+\vep_{i}(b_2)+1)}E_{i, bb'}^*(q)\Gup(b')\right),
\end{align*}
where we used the \thmref{theo:multupper}.
Here we have $d_{b_1, b_2}^b(q)=0$ if $\vep_{i}(b)>\vep_{i}(b_{1})+\vep_{i}(b_{2})$ by the descending induction on $\vep_{i}(b)$,
so we obtain the assertion.
Then we 
In particular, we obtain
\[_ir^{\vep_i(b_1)+\vep_i(b_2)}(\Gup(b_1)\Gup(b_2))=\sum_{\vep_i(b)=\vep_i(b_1)+\vep_i(b_2)}d_{b_1, b_2}^b(q)\Gup(\eit{i}^{\max}b).\]
\end{proof}
\end{NB}

As a corollary of \corref{cor:string1} and \corref{cor:string2}, we obtain  the following criterion.
\begin{cor}\label{cor:string3}
\textup{(1)}
If $b_1\bot b_2$, then $\eit{i}^{\max}b_1 \bot \eit{i}^{\max}b_2$ for any $i\in I$.
In fact, we have $\vep_{i}(b_{1}\circledast b_{2})=\vep_{i}(b_{1})+\vep_{i}(b_{2})$ and 
$\eit{i}^{\max}(b_{1})\circledast \eit{i}^{\max}(b_{2})=\eit{i}^{\max}(b_{1}\circledast b_{2})$.
Similar statement holds for $\eit{i}^{*\max}$.

\textup{(2)}
If $b$ is (resp.\ strongly) real, $\eit{i}^{\max}(b)$ is (resp.\ strongly) real for any $i\in I$.
In fact, we have $\vep_{i}(b^{[m]})=m\vep_{i}(b)$ and 
$(\eit{i}^{\max}b)^{[m]}=\eit{i}^{\max}(b^{[m]})$ for $m=2$ (resp.\ any $m$).
Similar statement holds for $\eit{i}^{*\max}$.

\end{cor}

\begin{lem}
If $b$ is (resp.\ strongly) real, we have $b^{[2]}=S_{2}(b)$ (resp.\ $b^{[m]}=S_{m}(b)$).
\end{lem}
\begin{proof}
For any $b$ with $\tr(\wt(b))>0$, there exists $i\in I$ such that $\vep_{i}(b)>0$.
Therefore  we can connect $b$ to $u_{\infty}$ by a path consisting of (strongly) real elements by successive applications of $\eit{i}^{\max}$'s.
From the formula in \corref{cor:string3} (2), we get the assertion.
\end{proof}
\subsection{Compatibilities of the dual canonical basis}
In this subsection, we study the dual canonical basis of integrable highest weight modules and its compatibilities with tensor products.
\subsubsection{}
We recall the definition of the dual canonical base of the integrable highest weight module $V(\lambda)$ following \cite[4.2]{Kas:global}.
Kashiwara call it the \emph{upper global basis}.
Let $M$ be an integrable $\U_{q}(\mfr{g})$-module with a weight decomposition $M=\bigoplus_{\lambda\in P}M_{\lambda}$.
For $u\in \Ker(e_{i})\cap M_{\lambda}$ and $0\leq n\leq \bracket{h_{i}, \lambda}$, 
we define other modified root operators called the \emph{upper crystal operators}:
\begin{align*}
\eit{i}^{\up}(f_{i}^{(n)}u)&=\frac{[\bracket{h_{i}, \lambda}-n+1]_{i}}{[n]_{i}}f_{i}^{(n-1)}u, \\
\fit{i}^{\up}(f_{i}^{(n)}u)&=\frac{[n+1]_{i}}{[\bracket{h_{i}, \lambda}-n]_{i}}f_{i}^{(n+1)}u.
\end{align*}
We have a $\Qq$-linear anti-automorphism $\varphi$ on $\Uq(\mfr{g})$ defined by 
\begin{align}
\varphi(e_i)=f_i, && \varphi(f_i)=e_i, && \varphi(q^h)=q^h. \label{eq:varphi}
\end{align}
For $\lambda\in P_+$, we have a unique symmetric non-degenerate bilinear form $(~, ~)_{\lambda}\colon V(\lambda)\otimes V(\lambda)\to \Qq$
which satisfies 
\begin{subequations}
\begin{align}
(\varphi(x)u, v)_\lambda&=(u, xv)_\lambda,\label{eq:innerprod_rep} &&  \text{~for~} u, v\in V(\lambda) \text{~and~} x\in \Uq(\mfr{g}),\\
(u_\lambda, u_\lambda)_\lambda&=1.
\end{align}
\end{subequations}
Then we have
\begin{subequations}
\begin{align}
(\eit{i}^{\up}u, v)_{\lambda}&=(u, \fit{i}^{\low}v)_{\lambda}, \\
(\fit{i}^{\up}u, v)_{\lambda}&=(u, \eit{i}^{\low}v)_{\lambda}.
\end{align}
\end{subequations}
Using $(\cdot, \cdot)_{\lambda}$, we define the dual bar involution $\sigma_{\lambda}$ by 
\[(\sigma_{\lambda}u, v)_{\lambda}:=\overline{(u, \overline{v})_{\lambda}}.\]
This is well-defined since $(\cdot, \cdot)_{\lambda}$ is a non-degenerate bilinear form.
We set 
\begin{subequations}
\begin{align}
	V(\lambda)_{\mca{A}}^{\up}&:=\{u\in V(\lambda); (u, V(\lambda)_{\mca{A}})_\lambda\subset \mca{A}\}, \\
	\Lup(\lambda)&:=\{u\in V(\lambda); (u, \Llow(\lambda))_\lambda\subset \mca{A}_0\}.
\end{align}
\end{subequations}
Then we have $\sigma_{\lambda}(\Lup(\lambda))=\{u\in V(\lambda); (u, \overline{\Llow(\lambda)})_\lambda\subset \mca{A}_\infty\}$.
Kashiwara denote $\sigma_{\lambda}(\Lup(\lambda))$ by $\overline{\mscr{L}}^{\up}(\lambda)$.
The triple $(\Lup(\lambda), \sigma_{\lambda}(\Lup(\lambda)), V(\lambda)_{\mca{A}}^{\up})$ is balanced by \cite[Lemma 2.2.3]{Kas:global}.

\begin{prop}\label{prop:dualVlambda}
Let $\mscr{B}^{\up}(\lambda)$ be the dual basis of $\mscr{B}^{\low}(\lambda)$ with respect to the induced pairing  
$(\cdot, \cdot)_{\lambda}\colon \Lup(\lambda)/q\Lup(\lambda)\times \Llow(\lambda)/q\Llow(\lambda)\to \mbb{Q}$,
then the pair $(\Lup(\lambda), \mscr{B}^{up}(\lambda))$ is an \emph{upper crystal base},
that is 
\begin{enumerate}
	\item $\Lup(\lambda)$ is a free $\mca{A}_{0}$-module with $\Qq\otimes_{\mca{A}_{0}}\Lup(\lambda)\simeq V(\lambda)$,
	\item $\fit{i}^{\up}\Lup(\lambda)\subset \Lup(\lambda)$ and $\eit{i}^{\up}\Lup(\lambda)\subset \Lup(\lambda)$,
	\item $\mscr{B}^{\up}(\lambda)\subset \Lup(\lambda)/q\Lup(\lambda)$ is a $\mbb{Q}$-basis,
	\item $\eit{i}^{\up}\mscr{B}^{\up}(\lambda)\subset \mscr{B}^{\up}(\lambda)\sqcup\{0\}$ and 
	$\fit{i}^{\up}\mscr{B}^{\up}(\lambda)\subset \mscr{B}^{\up}(\lambda)\sqcup\{0\}$,
	\item For $b, b'\in \mscr{B}^{\up}(\lambda)$, $b=\fit{i}^{\up}b'$ is equivalent to $\eit{i}^{\up}b=b'$.
\end{enumerate}
\end{prop}
Let $\Gup_\lambda$ be the inverse of $V(\lambda)_{\mca{A}}^{\up}\cap \Lup(\lambda)\cap\sigma_{\lambda}(\Lup(\lambda))\xrightarrow{\sim} \Lup(\lambda)/q\Lup(\lambda)$.
The set $\Gup_{\lambda}(\mscr{B}^{\up}(\lambda))$ is called the \emph{dual canonical basis} of $V(\lambda)$.
By its construction, the dual canonical basis is the dual basis of the canonical basis with respect to $(~, ~)_\lambda$.
We also have
\begin{equation}\Lup(\lambda)_\mu=q^{(\lambda, \lambda)/2-(\mu, \mu)/2}\Llow(\lambda)_\mu \text{~for~} \mu\in P\label{eq:upplow},\end{equation}
see \cite[(4.2.9)]{Kas:global}.
By \eqref{eq:upplow}, we obtain an isomorphism $\Lup(\lambda)/q\Lup(\lambda)\simeq \Llow(\lambda)/q\Llow(\lambda)$.
Through this identification, we have a bijection $\mscr{B}^{\up}(\lambda)\simeq \mscr{B}^{\low}(\lambda)$, and this bijection is an
isomorphism of abstract crystals associated with upper and lower crystal basis.
Hence we can identify $\mscr{B}^{\up}(\lambda)$ with $\mscr{B}^{\low}(\lambda)$ and denote both by $\mscr{B}(\lambda)$ hereafter.
If $\mu\in W\lambda$, this identification is given by the identity as $(\lambda, \lambda)=(\mu, \mu)$.
We can also prove that the canonical base elements and the dual canonical base elements coincide in this case.

\begin{rem}\label{rem:upplow}
For $\Uq^-(\mfr{g})$, we consider the $\Qq$-linear anti-automorphism $a$ of the reduced $q$-analogue $\mscr{B}_q(\mfr{g})$ defined by 
\begin{align}
a({_ir})=f_i, && a(f_i)={_ir}.
\end{align}
Since the (lower) crystal lattice is self-dual by Kashiwara's bilinear form $(\cdot, \cdot)_{K}$,
we do not need to consider the dual lattice of $\mscr{L}(\infty)$.
\end{rem}

\subsubsection{}
Using the pairing $(~, ~)_\lambda$, we consider an $\Qq$-linear embedding $j_\lambda\colon V(\lambda)\to \Uq^-(\mfr{g})$ which is defined in the following commutative diagram:
\[\xymatrix{
V(\lambda)\ar[d]^{j_{\lambda}}\ar[r]^{\sim} %
&V(\lambda)^* \ar[d]^{\pi_\lambda^*}\\
\Uq^-(\mfr{g})\ar[r]^{\sim} & \Uq^-(\mfr{g})^*,
}\]
where the horizontal isomorphisms are induced by the non-degenerate inner products on $V(\lambda)$ and $\Uq^-(\mfr{g})$
and the right vertical homomorphism is the transpose of $\Uq^-(\mfr{g})$-module homomorphism $\pi_\lambda\colon \Uq^-(\mfr{g})\to V(\lambda)$
given by $P\mapsto Pu_{\lambda}$.
Then for $b\in \mscr{B}(\lambda)$, 
we have $j_\lambda\Gup_{\lambda}(b)=\Gup(\jbar_\lambda(b))$,
where $j_{\lambda}$ in the right hand side was defined just after \thmref{theo:grloop}.
Thanks to this equality, there is no fear of confusion even though we use the same symbol $j_{\lambda}$ for different maps.
\begin{NB}
\begin{proof}
For convenience, we give a proof.
For the proof of the statement, it suffices to compute the following inner product.
\begin{align*}
&(j_\lambda\Gup_{\lambda}(\pibar_\lambda(b)), \Glow(b')) =(\Gup_{\lambda}(\pibar_\lambda(b)), \pi_\lambda\Glow(b')) \\
=&(\Gup_{\lambda}(\pibar_\lambda(b)), \Glow_\lambda(\pibar_\lambda(b')))=\delta_{\pibar_\lambda(b), \pibar_\lambda(b')}
\end{align*}
Since $\pibar_{\lambda}(b)=\pibar_{\lambda}(b)\neq 0$ if and only if $b=b'$, we obtain the claim.
\end{proof}
\end{NB}
\subsubsection{}
We use the following result in \cite[7.3.2]{LNT}.
For $\lambda, \lambda_1, \lambda_2, \cdots, \lambda_r\in P_+$ with $\lambda=\sum_j \lambda_j$,
let $\Phi(\lambda_1, \cdots, \lambda_r)\colon V(\lambda_1+\lambda_2+\cdots+\lambda_r)\to V(\lambda_1)\otimes \cdots \otimes V(\lambda_r)$ be
the unique $\Uq(\mfr{g})$-module homomorphism with 
$\Phi(\lambda_1, \cdots, \lambda_r)(u_{\lambda})=u_{\lambda_1}\otimes \cdots \otimes u_{\lambda_r}$.
Then we have the corresponding embeddings
\begin{align*}
&\Phi(\lambda_1, \cdots, \lambda_r)\colon \mscr{B}(\lambda)\hookrightarrow \mscr{B}(\lambda_{1})\otimes \cdots \otimes \mscr{B}(\lambda_{r}), \\
&\Phi(\lambda_1, \cdots, \lambda_r)(\mscr{L}^{\low}(\lambda))\subset \mscr{L}^{\low}(\lambda_{1})\otimes_{\mca{A}_{0}} \cdots \otimes_{\mca{A}_{0}} \mscr{L}^{\low}(\lambda_{r}).
\end{align*}
(See \cite[\S 4.2]{Kas:crystal}.)
Hence we obtain  
\[\Phi(\lambda_1, \cdots, \lambda_r)(\Glow_\lambda(b))\equiv\Glow_{\lambda_{1}}(b_{1})\otimes \cdots \otimes \Glow_{\lambda_{r}}(b_{r}) \mod q(\Llow(\lambda_{1})\otimes \cdots \otimes \Llow(\lambda_{r}))\]
for $\Phi(\lambda_{1}, \cdots, \lambda_{r})(b)=b_{1}\otimes \cdots \otimes b_{r}$ for some $b_{j}\in \mscr{B}^{\low}(\lambda_{j})$.

Let $q_{\lambda_{1}, \cdots, \lambda_{r}}\colon V(\lambda_{1})\otimes \cdots \otimes V(\lambda_{r})\to V(\lambda)$ be 
the homomorphism defined by the commutative diagram
\[\xymatrix{
V(\lambda_1)\otimes \cdots \otimes V(\lambda_r)\ar[d]^{q_{\lambda_1, \cdots, \lambda_r}}\ar[r]^{\sim} %
&V(\lambda_1)^*\otimes \cdots \otimes V(\lambda_r)^*\ar[d]^{\Phi(\lambda_1, \cdots, \lambda_r)^*} \\
V(\lambda)\ar[r]^{\sim} & V(\lambda)^*,
}\]
where the upper horizontal isomorphism is induced by the non-degenerate inner product
$(\cdot, \cdot)_{\lambda_{1}, \cdots, \lambda_{r}}:=(\cdot, \cdot)_{\lambda_{1}}\cdots (\cdot, \cdot)_{\lambda_{r}}$
on $V(\lambda_{1})\otimes \cdots \otimes V(\lambda_{r})$,
the lower horizontal isomorphism is induced by the non-degenrate inner product $(\cdot, \cdot)_{\lambda}$ on $V(\lambda)$
and the right vertical homomorphism is the transpose of $\Phi(\lambda_{1}, \cdots, \lambda_{r})$.

\begin{prop}\label{prop:prod_cry}
Let  $\lambda_1, \cdots, \lambda_r\in P_+$ 
and $b_j\in \mscr{B}(\lambda_j)~(1\leq j\leq r)$.
Assume that there exists  $b_1\diamond \cdots\diamond b_r\in \mscr{B}(\sum \lambda_j)$ with
$\Phi(\lambda_1, \lambda_2, \cdots, \lambda_r)(b_1\diamond \cdots\diamond b_r)=b_1\otimes \cdots \otimes b_r\in \mscr{B}(\lambda_1)\otimes\cdots\otimes \mscr{B}(\lambda_r)$.
Then we have the following equality
\[q_{\lambda_1, \cdots, \lambda_r}(\Gup_{\lambda_1}(b_1)\otimes \cdots \otimes\Gup_{\lambda_r}(b_r))=\Gup_\lambda(b_1\diamond b_2\diamond \cdots\diamond b_r) \mod q\mscr{L}^{\up}(\lambda).\]
\end{prop}
We give the proof for a completeness.
\begin{proof}
We have $q_{\lambda_1, \cdots, \lambda_r}(\Lup(\lambda_1)\otimes_{\mca{A}_0}\cdots \otimes_{\mca{A}_0}\Lup(\lambda_r))\subset \Lup(\lambda)$,
in particular we have $q_{\lambda_1, \cdots, \lambda_r}(\Gup_{\lambda_1}(b_1)\otimes \cdots\otimes \Gup_{\lambda_r}(b_r))\in \Lup(\lambda)$.

Hence to show the statement, it suffices  to compute the following inner product 
\[(q_{\lambda_1, \cdots, \lambda_r}(\Gup_{\lambda_1}(b_1)\otimes \cdots\otimes \Gup_{\lambda_r}(b_r)), \Glow_{\lambda}(b))_\lambda|_{q=0}\]
for $b\in \mscr{B}(\lambda)$.
By its definition of $q_{\lambda_1, \cdots, \lambda_r}$, this is equal to $(b_1\otimes \cdots \otimes b_r, \Phi(\lambda_1, \cdots, \lambda_r)(b))_{\lambda_1, \cdots, \lambda_r}|_{q=0}$.
Since the tensor product of the dual canonical basis is the dual of the tensor product of the canonical basis, this is equal to 
$\delta_{b_1\otimes \cdots \otimes b_r, \Phi(\lambda_1, \cdots, \lambda_r)(b)}=\delta_{b_1\diamond b_2\diamond \cdots\diamond b_r, b}$.
Hence we obtained the assertion.
\end{proof}

\subsubsection{}
To compute a product of dual canonical basis elements of integble highest weight modules,
we need the following modification of the coproduct as in \cite[7.2.5, 7.2.6]{LNT}.
\begin{lem}
For $\lambda, \mu\in P_+$,
let $r_{\lambda, \mu}\colon \Uq^-(\mfr{g})\to \Uq^-(\mfr{g})\otimes \Uq^-(\mfr{g})$ be the $\Qq$-linear map 
defined by
\[r_{\lambda, \mu}\Glow(b)=\sum_{b_1, b_2}d_{b_1, b_2}^b(q)q^{-(\wt(b_2), \lambda)}\Glow(b_1)\otimes \Glow(b_2)\]
for $\Glow(b)\in \Uq^-(\mfr{g})$ with $r(\Glow(b))=\sum_{b_1, b_2}d_{b_1, b_2}^b(q)\Glow(b_1)\otimes \Glow(b_2)$.
Then we have the commutative diagram of $\Qq$-vector spaces
\[\xymatrix{
\Uq^-(\mfr{g}) \ar[r]^{\pi_{\lambda+\mu}} \ar[d]_{r_{\lambda, \mu}}& V(\lambda+\mu) \ar[d]^{\Phi(\lambda, \mu)}\\
\Uq^-(\mfr{g})\otimes \Uq^-(\mfr{g}) \ar[r]_{\pi_\lambda\otimes \pi_\mu}& V(\lambda)\otimes V(\mu)
.}\]
\end{lem}
\begin{NB}
\begin{proof}
First we note that 
\[\Phi(\lambda, \mu)(\pi_{\lambda+\mu}(P))=\Phi(\lambda, \mu)(Pu_{\lambda+\mu})=\Delta_-(P)(u_{\lambda}\otimes u_\mu).\]
For $P\in \Uq^-(\mfr{g})$ with $r(P)=\sum P_{(1)}\otimes P_{(2)}$,
we have $\Delta_-(P)=\sum P_{(1)}t_{-\wt(P_{(2)})}\otimes P_{(2)}$.
Then for $b\in \mscr{B}(\infty)$ with $r(\Glow(b))=\sum_{b_1, b_2}d_{b_1, b_2}^b(q)\Glow(b_1)\otimes \Glow(b_2)$,
we have $\Delta_-(\Glow(b))=\sum_{b_1, b_2}d_{b_1, b_2}^b(q)\Glow(b_1)t_{-\wt(b_2)}\otimes \Glow(b_2)$.
Then we obtain that 
\begin{align*}
&\Delta_-(\Glow(b))(u_{\lambda}\otimes u_{\mu})\\
=&\sum_{b_1, b_2}d_{b_1, b_2}^b(q)\Glow(b_1)t_{-\wt(b_2)}\otimes \Glow(b_2)(u_{\lambda}\otimes u_{\mu}) \\
=&\sum_{b_1, b_2}d_{b_1, b_2}^b(q)q^{-(\wt(b_2), \lambda)}\Glow(b_1)\otimes \Glow(b_2)(u_{\lambda}\otimes u_{\mu}).
\end{align*}
\end{proof}
\end{NB}
Using the above modification, we obtain the following formula.
\begin{prop}\label{prop:prod_rep}
For $b_1\in \mscr{B}(\lambda)$ and $b_2\in \mscr{B}(\mu)$,
we have 
\[q^{(\wt b_2-\mu, \lambda)}\Gup(\jbar_\lambda(b_1))\Gup(\jbar_\mu(b_2))=j_{\lambda+\mu}q_{\lambda, \mu}(\Gup_\lambda(b_1)\otimes \Gup_\mu(b_2)).\]
\end{prop}
\begin{NB}
\begin{proof}
For convenience of the reader, we give a proof.
\begin{align*}
	&(j_{\lambda+\mu}q_{\lambda, \mu}(\Gup_\lambda(b_1)\otimes \Gup_\mu(b_2)), \Glow(b'))_K \\
	=&(q_{\lambda, \mu}(\Gup_\lambda(b_1)\otimes \Gup_\mu(b_2)), \pi_{\lambda+\mu}\Glow(b'))_K \\
	=&(\Gup_\lambda(b_1)\otimes \Gup_\mu(b_2), \Phi(\lambda, \mu)\pi_{\lambda+\mu}\Glow(b')) \\
	=&(\Gup_\lambda(b_1)\otimes \Gup_\mu(b_2), (\pi_{\lambda}\otimes\pi_{\mu})r_{\lambda, \mu}\Glow(b')) \\
	=&(\Gup(\jbar_\lambda b_1)\otimes \Gup(\jbar_\mu b_2), r_{\lambda, \mu}\Glow(b')) \\
	=&q^{-(\wt b_2-\mu, \lambda)}(\Gup(\jbar_\lambda b_1)\otimes \Gup(\jbar_\mu b_2), r\Glow(b'))
\end{align*}
Hence, we have obtain the equality.
\end{proof}
\end{NB}
\subsubsection{}
Combining \propref{prop:prod_rep} with \propref{prop:prod_cry}, we obtain the following proposition.
\begin{NB}
This is false!
This is a generalization of \cite[1.9]{Cal:adapted}, in our argument we don't need the assumption
on the string parametrization.
\end{NB}
\begin{prop}
Let  $\lambda_1, \cdots, \lambda_r\in P_+$ 
and $b_j\in \mscr{B}(\lambda_j)~(1\leq j\leq r)$.
Assume that there exists  $b_1\diamond \cdots\diamond b_r\in \mscr{B}(\sum \lambda_j)$ with
$\Phi(\lambda_1, \lambda_2, \cdots, \lambda_r)(b_1\diamond \cdots\diamond b_r)=b_1\otimes \cdots \otimes b_r\in \mscr{B}(\lambda_1)\otimes\cdots\otimes \mscr{B}(\lambda_r)$.
Then there exists a unique $m\in \mbb{Z}$ such that
\[q^m\Gup(\jbar_{\lambda_1}(b_1))\cdots \Gup(\jbar_{\lambda_r}(b_r))=\Gup(\jbar_{\sum \lambda_i}(b_1\diamond \cdots\diamond b_r)) \mod q\mscr{L}(\infty).\]
\end{prop}
\begin{NB}
\begin{proof}
By iterating the \propref{prop:prod_rep}, for some $m$, 
we have
\[q^m\Gup(\jbar_{\lambda_1}(b_1))\cdots \Gup(\jbar_{\lambda_r}(b_r))=j_{\sum \lambda_j}q_{\lambda_1, \cdots, \lambda_r}%
(\Gup_{\lambda_1}(b_1)\otimes \cdots \otimes\Gup_{\lambda_r}(b_r)).\]
Here by \propref{prop:prod_cry} we have 
\[j_{\sum \lambda_j}q_{\lambda_1, \cdots, \lambda_r}(\Gup_{\lambda_1}(b_1)\otimes \cdots \otimes\Gup_{\lambda_r}(b_r)) %
\equiv \Gup(\jbar_{\sum \lambda_j}(b_1\diamond \cdots \diamond b_r))\mod q\mscr{L}(\infty).\]
Then we obtain
\begin{equation}
q^m \Gup(\jbar_{\lambda_1}(b_1))\cdots \Gup(\jbar_{\lambda_r}(b_r))=\Gup(\jbar_{\sum \lambda_j}(b_1\diamond \cdots \diamond b_r))+\sum d_{b_1, \cdots, b_r}^{b'}(q)\Gup(b'),
\label{eq:caladapted1}
\end{equation}
with $d_{b_1, \cdots, b_r}^{b'}(q)\in q\mbb{Z}[q]$.
Applying the dual bar involuion, we obtain 
\[q^{-m} \Gup(\jbar_{\lambda_r}(b_r))\cdots \Gup(\jbar_{\lambda_1}(b_1))=\Gup(\jbar_{\sum \lambda_j}(b_1\diamond \cdots \diamond b_r))+\sum d_{b_1, \cdots, b_r}^{b'}(q^{-1})\Gup(b').\]
By the assumption on $q$-commutativitiy, we have
\begin{equation}
q^{m'} \Gup(\jbar_{\lambda_1}(b_1))\cdots\Gup(\jbar_{\lambda_r}(b_r))=\Gup(\jbar_{\sum \lambda_j}(b_1\diamond \cdots \diamond b_r))+\sum d_{b_1, \cdots, b_r}^{b'}(q^{-1})\Gup(b').
\label{eq:caladapted2}
\end{equation}
By compareing \eqref{eq:caladapted1} and \eqref{eq:caladapted2}, we obtain the assertion.
\end{proof}
\end{NB}
\begin{NB}
\begin{rem}
Satoshi Naito informed me that the (ii) in \cite[1.8 Propositon]{Cal:adapted}  does not hold for $r\geq 3$ in general.
But the theorem itself can be proved by considering the induction on $w$.
\end{rem}
\end{NB}

\section{Quantum unipotent subgroup and the dual canonical basis}\label{sec:unip}
\subsection{The Lie algebra $\mfr{n}(w)$}
\subsubsection{}
\begin{NB}
\begin{defn}[{\cite[Section 6.1]{Kum}}]
Let $\Delta_+$ be the set of positive roots associated with the Lie algebra $\mfr{g}$.
A subset $\Theta$ of $\Delta^+$ is called \emph{bracket closed} if for all $\alpha, \beta\in \Theta$
with $\alpha+\beta\in \Delta^+$, we have $\alpha+\beta\in \Theta$.
A subset $\Theta$ of $\Delta^+$ is called \emph{bracket coclosed} if $\Delta^+\setminus \Theta$ is bracket closed.
\end{defn}
\end{NB}
Let $w\in W$ be an element of the Weyl group associated with $\mfr{g}$.
Let $\Delta^+(w):=\Delta^+\cap w\Delta^-=\{\alpha\in \Delta^+ | w^{-1}\alpha<0\}\subset \Delta^+$.
\begin{NB}It is well-known that $\Delta^+(w)$ is bracket closed and bracket coclosed (\cite[Example 6.1.5(b)]{Kum}).\end{NB}
We have the following description of $\Delta^+(w)$ as follows (\cite[Lemma 1.3.14]{Kum}).

For a Weyl group element $w$,
let  $\tw=(i_1, i_2, \cdots, i_l)\in R(w)$ be a reduced expression of $w$, 
where $R(w)$ is the set of reduced expression of $w$.
	For each $1\leq k\leq l=l(w)$, we set
	\[\beta_k:=s_{i_1}s_{i_2}\cdots s_{i_{k-1}}(\alpha_{i_k}).\]
	Then $\Delta^+(w)$ has cardinality exactly equal to $l=l(w)$ and we have
	\[\Delta^+(w)=\{\beta_k\}_{1\leq k\leq l}.\]

Let
\[\mfr{n}(w)=\bigoplus_{\alpha\in \Delta^+(w)}\mfr{g}_{-\alpha}.\]
Let $N(w)$ be the corresponding (pro-)unipotent (pro-) group in \cite[VI]{Kum}.
Then $N(w)$ is a unipotent algebraic group of dimension $l(w)$ and its Lie algebra is $\mfr{n}(w)$.
We can identify the restricted dual $U(\mfr{n}(w))_{\operatorname{gr}}^{*}$ of $U(\mfr{n}(w))$ with the coordinate ring of $N(w)$,
that is $U(\mfr{n}(w))_{\operatorname{gr}}^{*}\simeq \mbb{C}[N(w)]$, see \cite[5.2]{GLS:KacMoody} for more details.

\subsection{Braid group symmetry on $\Uq(\mfr{g})$} \label{sec:Lusbraid}
We define (quantum) root vectors, using Lusztig's braid group symmetry $\{T_i\}$ on $\Uq(\mfr{g})$.
See \cite[Chapter 32]{Lus:intro} for more details.
\subsubsection{}
Following \cite[37.1.3]{Lus:intro},
we define the $\Qq$-algebra automorphisms $T'_{i, e}\colon \Uq(\mfr{g})\to \Uq(\mfr{g})$ for $i\in I$ and $e\in \{\pm 1\}$ by
\begin{subequations}
\begin{align}
	T'_{i, e}(q^h)&=q^{s_i(h)}, \\
	T'_{i, e}(e_i)&=-t_i^{e}f_i, \\
	T'_{i, e}(f_i)&=-e_it_i^{-e}, \\
	T'_{i, e}(e_j)&=\sum_{r+s=-\bracket{h_i, \alpha_j}}(-1)^r q_i^{er}e_i^{(r)}e_j e_i^{(s)} \text{~for~} j\neq i, \\
	T'_{i, e}(f_j)&=\sum_{r+s=-\bracket{h_i, \alpha_j}}(-1)^r q_i^{-er}f_i^{(s)}f_j f_i^{(r)} \text{~for~} j\neq i.
\end{align}
\end{subequations}

For $i\in I$ and $e\in \{\pm 1\}$,
we also define the $\Qq$-algebra automorphisms $T''_{i, e}\colon \Uq(\mfr{g})\to \Uq(\mfr{g})$ by
\begin{subequations}
\begin{align}
	T''_{i, -e}(q^h)&=q^{s_i(h)}, \\
	T''_{i, -e}(e_i)&=-f_it_i^{-e}, \\
	T''_{i, -e}(f_i)&=-t_i^{e}e_i, \\
	T''_{i, -e}(e_j)&=\sum_{r+s=-\bracket{h_i, \alpha_j}}(-1)^r q_i^{er}e_i^{(s)}e_j e_i^{(r)} \text{~for~} j\neq i, \\
	T''_{i, -e}(f_j)&=\sum_{r+s=-\bracket{h_i, \alpha_j}}(-1)^r q_i^{-er}f_i^{(r)}f_j f_i^{(s)} \text{~for~} j\neq i.
\end{align}
\end{subequations}
We have 
\beq T'_{i, e}T''_{i, -e}=T''_{i, -e}T'_{i, e}=\text{id}.\eeq
In the following, we write $T_i=T''_{i, 1}$ and $T_i^{-1}=T'_{i, -1}$ as in \cite[Proposition 1.3.1]{Saito:PBW}.
\subsubsection{}
We define a $q$-analogue of the action of the Weyl group on integrable module following \cite[Chapter 5]{Lus:intro} and \cite{Saito:PBW}.
We use a $q$-analogue of exponential $\exp_q(x)$ defined by 
\[\exp_q(x):=\sum_{n\geq 0}\frac{q^{n(n-1)/2}}{[n]_q!}x^n.\]
We have 
\begin{equation}\exp_{q}(x)\exp_{q^{-1}}(-x)=1.\label{eq:expinv}\end{equation}
For $i\in I$, we define $S_i$ (\cite[(1.2.2), (1.2.13)]{Saito:PBW}) by
\begin{subequations}
\begin{align}
S_i:= \;&\exp_{q_i^{-1}}(q_i^{-1}e_it_{i}^{-1})\exp_{q_i^{-1}}(-f_{i})\exp_{q_i^{-1}}(q_ie_{i}t_{i})q_i^{h_i(h_i+1)/2} \\
=\;&\exp_{q_i^{-1}}(-q_i^{-1}f_it_{i})\exp_{q_i^{-1}}(e_i)\exp_{q_i^{-1}}(-q_if_it_i^{-1})q_i^{h_i(h_i+1)/2}.
\end{align}
\end{subequations}
For integrable $\Uq(\mfr{g})$-modules, the action of $S_i$ is well-defined.
It is known that the action of $\{S_i\}_{i\in I}$ satisfies the braid group relations for the Weyl group $W$.

The braid group symmetry $\{T_i\}_{i\in I}$ defined above is described as
\begin{equation}
T_i(x)=S_ixS_i^{-1} \label{eq:braidadj},
\end{equation}
where the elements are considered in the endomorphism ring of integrable modules,
see \cite[1.3]{Saito:PBW} for more details.

\subsubsection{}
We have the following relationship between $T_{i}, T_{i}^{-1}$ and the $*$-involution.
%
\begin{prop}[{\cite[37.2.4]{Lus:intro}}]
	We have 
	\begin{align}
	* \circ T_i\circ *&=T_i^{-1}.  \label{eq:Ti_inverse}
	\end{align}
\end{prop}
\begin{NB}

Moreover we have the following relation between $T'_{i, e}$ and $T''_{i, e}$ for homogenous elements in $\Uq(\mfr{g})$.
\begin{prop}\cite[37.2.4]{Lus:intro}
\begin{align}
	\;&\bar{~} \circ T'_{i, e} \circ \bar{~}= T'_{i, -e}, \\
	\;&\bar{~} \circ T''_{i, e} \circ \bar{~}= T''_{i, -e}, \\
	\;&T''_{i, e}(x)=(-1)^{\bracket{h_i, \xi}}q^{e(\alpha_i, \xi)}T'_{i, e}(x) \text{~for~} x\in \Uq(\mfr{g})_\xi, \\
	\;&T'_{i, e}(x)=(-1)^{\bracket{h_i, \xi}}q^{-e(\alpha_i, \xi)}T''_{i, e}(x) \text{~for~} x\in \Uq(\mfr{g})_\xi.
\end{align}
\end{prop}
\end{NB}
\begin{NB}
\subsubsection{}
It is known that the operators $\{T_{i}\}_{i\in I}$ satisfy the braid group relations corresponding to the Weyl group \cite[Chapter 39]{Lus:intro} and \cite[1.4]{Saito:PBW},
then we can extend to a braid group action on $\Uq(\mfr{g})$ by
\begin{subequations}
\begin{align}
	T_{w w'}&=T_w\circ T_{w'} ~~\text{if~~} \ell(ww')=\ell(w)+\ell(w'), \\
	T_{s_i}&=T_i.
\end{align}
\end{subequations}
\end{NB}
\subsection{Quantum nilpotent subalgebra $\Uq^-(w, e)$}
\subsubsection{}\label{Root vector}
We define root vectors associated with $\tw=(i_{1}, \cdots, i_{l})\in R(w)$ for $w\in W$.
See \cite[Proposition 40.1.3, Proposition 41.1.4]{Lus:intro} for more detail.
For $w\in W$ and $\tw\in R(w)$, we define $\beta_{k}$ as above.
We define the \emph{root vectors} $F_e(\beta_k)$ associated with $\beta_k\in \Delta(w)$ and $e\in \{\pm 1\}$
by 
\[F_e(\beta_k):=T_{i_1}^{e}\cdots T_{i_{k-1}}^{e}(f_{i_k}).\]
It is known that $F_e(\beta_k)\in \Uq^-(\mfr{g})$.
We note that $F_e(\beta_k)$ \emph{does} depend on the choice of $\tw\in R(w)$.
We define its divided power by 
\[F_e(c\beta_k):=T_{i_1}^{e}\cdots T_{i_{k-1}}^{e}(f_{i_k}^{(c)})\]
for $c\geq 1$.
It is known that $F_e(c\beta_k)\in \Uq^-(\mfr{g})_{\mca{A}}$.
\subsubsection{}\label{sec:PBWbasis}
\begin{theo}[{\cite[Proposition 40.2.1, Proposition 41.1.3]{Lus:intro}}]\label{thm:PBW}

\textup{(1)}
For $w\in W$, $\tw\in R(w)$, $e\in \{\pm 1\}$ and $\mbf{c}\in \mbb{Z}_{\geq 0}^l$, we set 
\[F_e(\mbf{c}, \tw):=\begin{cases}
F_e(c_1\beta_1)\cdots  F_e(c_l\beta_l)\;& \text{~if~} e=+1, \\
F_e(c_l\beta_l)\cdots  F_e(c_1\beta_1)\;& \text{~if~} e=-1.
\end{cases}\]
Then $\{F_e(\mbf{c}, \tw)\}_{\mbf{c}\in \mbb{Z}_{\geq 0}^l}$ forms a basis of a subspace $\Uq^{-}(w, e)$ of $\Uq^{-}(\mfr{g})$ which does not depend on $\tw$.

\textup{(2)}
We have $F_e(\mbf{c}, \tw)\in \Uq^-(\mfr{g})_{\mca{A}}$ for any $\mbf{c}\in \mbb{Z}_{\geq 0}^l$.
\end{theo}
\subsubsection{}\label{sec:LSformula}
We recall commutation relations for root vectors and its divied powers $\{F(c_k\beta_k)\}_{1\leq k\leq l, c_k\geq 1}$,
known as the \emph{Levendorskii-Soibelman formula}.
See \cite{LevSoi}, \cite{Beck:convex} or \cite{Nak:CBMS} for more details.

In this subsections, we give statements for the $e=+1$ case.
We can obtain the corresponding results for the $e=-1$ case, applying the $*$-involution \eqref{eq:Ti_inverse}.
So we denote $F_e(c\beta)$, $F_{e}(\mbf{c}, \tw)$ by $F(c\beta)$, $F(\mbf{c}, \tw)$ by omitting $e$.

Let $w\in W$, $\tw=(i_1, i_2, \cdots, i_l)\in R(w)$
and fix a total order on $\Delta^+(w)$ given by 
\[\beta_1<\beta_2<\cdots <\beta_l.\]
\begin{theo}[{\cite[Proposition 3.6]{Nak:CBMS}, \cite[5.5.2 Proposition]{LevSoi}}]\label{theo:LSformula}
For $j<k$, let us write
\[F(c_k\beta_k)F(c_j\beta_{j})-q^{-(c_j\beta_j, c_k\beta_k)} F(c_j\beta_{j})F(c_k\beta_{k})%
=\sum_{}f_{\mbf{c'}}F(\mbf{c}', \tw)\]
with $f_{\mbf{c}'}\in \Qq$.
If $f_{\mbf{c}'}\neq 0$, then $c'_{j}<c_{j}$ and $c'_{k}<c_{k}$
with $\sum_{j\leq m\leq k} c'_{m}\beta_m=c_{j}\beta_j+c_{k}\beta_k$.
\end{theo}
\begin{NB}
\begin{rem}
The integrality of the \thmref{theo:LSformula}, that is 
$\rho^{\mbf{c}'}_{c_j\beta_j, c_k\beta_k}\in \mca{A}$,
is not known except for finite cases and affine cases ?
But this should follow from the existence of the canonical base and its integrality? 
\end{rem}
\end{NB}
\subsubsection{}
The following proposition is  a consequence of \thmref{theo:LSformula}.
(cf.\ \cite[2.4.2 Proposition Theorem b)]{LevSoi:qWeyl} and \cite[2.2 Proposition]{DKP:solvable}.)
\begin{prop}\label{prop:DKP}
Let $\tw=(i_1, i_2, \cdots, i_l)$ be a reduced expression for $w\in W$
and $e\in\{\pm 1\}$.
Then the subspace $\Uq^-(w, e)$ is a $\Qq$-subalgebra generated by $\{F_e(\beta_k)\}_{1\leq k\leq l}$.
\end{prop}
We call it the \emph{quantum nilpotent subalgebra} associated with $w\in W$.

\subsubsection{}\label{subsec:lex}
We define a lexicographic order $\leq$ on $\mbb{Z}_{\geq 0}^l$ associated with $\tw\in R(w)$
by
\begin{align*}
&\mbf{c}=(c_1, c_2, \cdots, c_l)< \mbf{c'}=(c'_1, c'_2, \cdots, c'_l) \\
\Longleftrightarrow & \text{~~there exists~~} 1\leq p \leq l \text{~~such that~~} 
c_1=c'_1, \cdots, c_{p-1}=c'_{p-1}, c_p<c'_p.
\end{align*}
The following theorem is obtained as a consequence of the Levendorskii-Soibelmann formula.
\begin{theo}
	Let $w\in W$ and $\tw\in R(w)$ be its reduced expression.
	For $\mbf{c}\in \mbb{Z}_{\geq 0}^l$, we consider the following $\Qq$-subspace $\mca{F}_{\leq \mbf{c}}^{\tw}\Uq^-(w)$: 
	\[\mca{F}_{\leq \mbf{c}}^{\tw}\Uq^-(w):=\bigoplus_{\mbf{c'}\leq \mbf{c}}\Qq F(\mbf{c'}, \tw).\]
	Then 
	
	\textup{(1)} 
	$\{\mca{F}_{\leq \mbf{c}}^{\tw}\Uq^-(w)\}_{\mbf{c}\in \mbb{Z}_{\geq 0}^l}$ forms an increasing filtration on $\Uq^-(w)$.
	
	\textup{(2)}
	The associated graded algebra $\gr^{\tw}\Uq^-(w)$ is generated by $\{\gr^{\tw}(F(\beta_k))| 1\leq k\leq l\}$
	with relations:
	\[\gr^{\tw}(F(\beta_k))\gr^{\tw}(F(\beta_j))=q^{-(\beta_j, \beta_k)}\gr^{\tw}(F(\beta_j))\gr^{\tw}(F(\beta_k))%
	~(j<k).\]
\end{theo}
We call this the \emph{De Concini-Kac filtration}.
\subsection{PBW basis and the canonical base}
In this subsection, we recall  compatibilities between Lusztig's braid symmetry $\{T_i\}_{i\in I}$ and the canonical base.
For more details, see \cite[Chapter 38]{Lus:intro}, \cite{Lus:braid} and \cite{Saito:PBW}.
\begin{lem}[{\cite[Proposition 38.1.6, Lemma 38.1.5]{Lus:intro}}]\label{lem:orthdecomp}
	\textup{(1)}
	For $i\in I$, we have 
	\begin{align*}
	\Uq^-[i]:=&\{x\in \Uq^-; {_ir}(x)=0\}, \\
	=&\{x\in \Uq^-; T_i^{-1}(x)\in \Uq^-\}, \\
	{^*}\Uq^-[i]:=&\{x\in \Uq^-; {r_i}(x)=0\}, \\
	=&\{x\in \Uq^-; T_i(x)\in \Uq^-\}.
	\end{align*}
	
	\textup{(2)}
	For $i\in I$, we have the following orthogonal decompositions:
	\begin{align*}
	\Uq^-=\Uq^-[i]\+f_i\Uq^-={^*}\Uq^-[i]\+\Uq^-f_i.
	\end{align*}
\end{lem}
From \lemref{lem:orthdecomp} and \thmref{theo:lower_filt}, we obtain the following result.
\begin{prop}
	For $n\geq 0$ and $i\in I$, the subspaces 
	$\bigoplus_{k=0}^n f_i^k\Uq^-[i]$ and $\bigoplus_{k=0}^n {^*}\Uq^-[i]f_i^k$
	are compatible with the dual canonical base and we have 
	\begin{align*}
	\bigoplus_{k=0}^n f_i^k\Uq^-[i]=\bigoplus_{b\in \mscr{B}(\infty), \vep_i(b)\leq n}\Qq \Gup(b), \\
	\bigoplus_{k=0}^n {^*}\Uq^-[i]f_i^k=\bigoplus_{b\in \mscr{B}(\infty), \vep^*_i(b)\leq n}\Qq \Gup(b).
	\end{align*}
\end{prop}
Let ${^i}\pi\colon \Uq^-\to \Uq^-[i]$ (resp.\ $\pi{^i}\colon \Uq^-\to {^*}\Uq^-[i]$)
be the orthogonal projection whose kernel is $f_i\Uq^-(\mfr{g})$ (resp.\ $\Uq^-(\mfr{g})f_i$).
The following result is due to Saito and Lusztig.
\begin{theo}[{\cite[Prospoition 3.4.7, Corollary 3.4.8]{Saito:PBW}, \cite[Theorem 1.2]{Lus:braid}}]\label{theo:braidlower}
	For $b\in \mscr{B}(\infty)$ with $\vep_i^*(b)=0$, we have 
	\[T_i({\pi^i}\Glow(b))={^i}\pi(\Glow(\Lambda_i(b))) \in \mscr{L}(\infty),\]
	where $\Lambda_i \colon \{b\in \mscr{B}(\infty); \vep_i^{*}(b)=0\}\to \{b\in \mscr{B}(\infty); \vep_i(b)=0\}$
	is the bijection given by $\Lambda_i(b)=\fit{i}^{*\varphi_i(b)}\eit{i}^{\vep_i(b)}b$
	and its inverse is given by $\Lambda_i^{-1}(b)=\fit{i}^{\varphi^*_i(b)}\eit{i}^{*\vep^*_i(b)}b$.
\end{theo}
By \thmref{theo:braidlower}, we obtain the following result.
\begin{theo}[{\cite[Theorem 4.1.2]{Saito:PBW}, \cite[Proposition 8.2]{Lus:braid}}]
For $w\in W$, $\tw=(i_1, i_2, \cdots, i_l)\in R(w)$ and $e\in \{\pm 1\}$,

\textup{(1)}
we have $F_e(\mbf{c}, \tw)\in \mscr{L}(\infty)$
and  
\[b_e(\mbf{c}, \tw)=F_e(\mbf{c}, \tw) \mod q\mscr{L}(\infty)\in \mscr{B}(\infty).\]

\textup{(b)}
The map $\mbb{Z}_{\geq 0}^l \to \mscr{B}(\infty)$ which is defined by $\mbf{c}\mapsto b_e(\mbf{c}, \tw)$ is injective.
We denote the image by $\mscr{B}(w, e)$ and this does not depend on a choice of $\tw\in R(w)$.
\end{theo}

For fixed $\tw\in R(w)$, we denote  the inverse of $\mbf{c}\mapsto b_e(\mbf{c}, \tw)$ by $L_{e, \tw} \colon \mscr{B}(w, e)\to \mbb{Z}_{\geq 0}^l$.
This map is called \emph{Lusztig data} of $b$ associated with $\tw$.
\subsubsection{}
As a corollary of the above description, we have the following properties.
\begin{cor}
\textup{(1)}
We have $\Lambda_i S_m(b)=S_m \Lambda_i(b)$ 
for $b\in \{b\in \mscr{B}(\infty) ; \vep_i^{*}(b)=0\}$.

\textup{(2)}
We have $S_m \colon \mscr{B}(w, e)\to \mscr{B}(w, e)$ for any $m\geq 1$ and $S_m(b_{e}(\mbf{c}, \tw))=b_{e}(m\mbf{c}, \tw)$.

\textup{(3)}
We have $*(\mscr{B}(w, e))=\mscr{B}(w, -e)$ and 
$*b_e(\mbf{c}, \tw)=b_{-e}(\mbf{c}, \tw)$.
\end{cor}

\begin{NB}
\begin{proof}
For convenience of the reader, we give a proof.
\textup{(1)}
For $b\in \mca{B}(\infty)$, we have $m\vep_i(b)=\vep_i(S_mb)$ for any $m$.
Here the left hand side make sense.
\begin{align*}
	\Lambda_i S_m(b)=\fit{i}^{*\varphi_i(S_m(b))}\eit{i}^{\vep_i(S_mb)}S_mb=\fit{i}^{*m\varphi_i(b)}\eit{i}^{m\vep_i(b)}S_mb=S_m(\Lambda_i(b)),
\end{align*}

\textup{(2)}
This can be proven by the induction on the length $l$ of $w\in W$.
\end{proof}
\end{NB}


\subsection{Inner products of PBW basis}
By  \ref{lem:KLforms},
we have the following modification of \cite[Proposition 38.2.1]{Lus:intro}.
\begin{prop}\label{prop:Ti_Kform}
For $x,y\in \Uq^-(\mfr{g})_{\xi}$ with $x, y\in \Uq^-[i]$ (resp.\ with $x, y\in {^*}\Uq^-[i]$),
we have 
\[(x, y)_K=(1-q_i^2)^{-\bracket{h_i, \xi}}(T_i^{-1}x, T_i^{-1}y)_K ~(\text{resp.\ } (1-q_i^2)^{-\bracket{h_i, \xi}}(T_ix, T_iy)_K).\]
\end{prop}

\subsubsection{}
We have the following formula for inner product of PBW basis with respect to Lusztig's bilinear form $(~, ~)_{L}$.
For more details, see \cite[Propsition 38.2.3]{Lus:intro}.
\begin{prop}\label{prop:innerprod}
Let $w\in W$ and $\tw\in R(w)$ with $l=\ell(w)$.
We have 
\[(F(\mbf{c}, \tw), F(\mbf{c}', \tw))_L=\prod_{k=1}^l\delta_{c_k, c'_k}\prod_{s=1}^{c_k}\frac{1}{1-q_i^{2s}}=%
\prod_{k=1}^l\delta_{c_k, c'_k}(-1)^{c_k}\frac{q_{i_k}^{-\frac{c_k(c_k+1)}{2}}}{(q_{i_k}-q_{i_k}^{-1})^{c_k}[c_k]_{i_k}^!}.\]

\end{prop}
\begin{NB}
We have
\[\overline{(F^{\tw}(\mbf{c}), F^{\tw}(\mbf{c}))_L}=(-1)^{\sum c_k}\prod_{k=1}^l q_{i_k}^{\frac{c_k(c_k+1)}{2}}(F^{\tw}(\mbf{c}), F^{\tw}(\mbf{c}))_L.\]
For $\xi=\sum \xi_i\alpha_i\in Q$, we set $\xi(q)=\prod_{i\in I}(1-q_i^2)^{\xi_i}=\prod_{i\in I}\{q_i(q_i^{-1}-q_i)\}^{\xi_i}$,
then we have $\overline{\xi(q)}=\prod_{i\in I}\{q_i^{-1}(q_i-q_i^{-1})\}^{\xi_i}=q_{-2\xi}(-1)^{\tr \xi}\xi(q)=(-1)^{\tr \xi}q^{-2(\xi, \rho)}\xi(q)$.
Here we consider the inner product 
\[(F^{\tw}(\mbf{c}), F^{\tw}(\mbf{c}))_K=(F^{\tw}(\mbf{c}), F^{\tw}(\mbf{c}))_L(-\wt(c))(q)\]
then we have 
\begin{align*}
	&\overline{(F^{\tw}(\mbf{c}), F^{\tw}(\mbf{c}))_K}=\overline{(F^{\tw}(\mbf{c}), F^{\tw}(\mbf{c}))_L(-\wt(c))(q)} \\
	=&(-1)^{\sum c_k}\prod_{k=1}^l q_{i_k}^{\frac{c_k(c_k+1)}{2}}(-1)^{\tr\wt{\mbf{c}}}q^{2(\wt{\mbf{c}}, \rho)}(F^{\tw}(\mbf{c}), F^{\tw}(\mbf{c}))_K\\
	=&(-1)^{\sum c_k(1-\tr\beta_k))}\prod_{k=1}^l q_{i_k}^{\frac{c_k(c_k+1)}{2}}q^{2(\wt{\mbf{c}}, \rho)}(F^{\tw}(\mbf{c}), F^{\tw}(\mbf{c}))_K
\end{align*}
\end{NB}
\subsection{Compatibility with $T_i$ and the dual canonical base}
\subsubsection{}
By using the above results, we obtain the following compatibility between
the dual canonical basis and Lusztig's braid group symmetry $T_i$.
\begin{theo}\label{theo:braidupper}
For $b\in \mscr{B}(\infty)$ wiht $\vep_i^*(b)=0$,
we have 
\[(1-q_i^2)^{\bracket{h_i, \xi}}T_i \Gup(b)=\Gup(\Lambda_i b).\]
\end{theo}
\begin{proof}
We shall prove that $((1-q_i^2)^{\bracket{h_i, \xi}}T_i \Gup(b), \Glow(b'))_K=\delta_{b', \Lambda_i(b)}$.
By \lemref{lem:orthdecomp}, $(1-q_i^2)^{\bracket{h_i, \xi}}(T_i \Gup(b), \Glow(b'))_{K}$ is equal to 
$(1-q_i^2)^{\bracket{h_i, \xi}}(T_i \Gup(b), {{^i}\pi}\Glow(b'))_{K}$.
By the \propref{prop:Ti_Kform}, this is equal to $(\Gup(b), T_i^{-1}{{^i}\pi}\Glow(b'))_{K}$.
Using \thmref{theo:braidlower}, we have
\begin{align*}
\;&(\Gup(b), T_i^{-1}{{^i}\pi}\Glow(b'))_{K}=(\Gup(b), \pi^{i}\Glow(\Lambda_i^{-1}b'))_{K}\\
=\;&(\Gup(b), \Glow(\Lambda_i^{-1}b'))_{K}=\delta_{b, \Lambda_i^{-1}(b')}.
\end{align*}
Then we obtain the assertion.
\end{proof}
As a corollary, we obtain the following multiplicative properties.

\begin{cor}
\textup{(1)}
For $b_1, b_2\in \mscr{B}(\infty)$ with $\vep_i^*(b_1)=\vep_i^*(b_2)=0$
(resp.\ $\vep_i(b_1)=\vep_i(b_2)=0$)
and $b_1\bot b_2$,
we have $\Lambda_i(b_1)\bot \Lambda_i(b_2)$
(resp.\ $\Lambda_i^{-1}(b_1)\bot \Lambda_i^{-1}(b_2)$).

\textup{(2)}
If $b\in \mscr{B}(\infty)$ with $\vep^*_i(b)=0$ (resp.\ $\vep_{i}(b)=0$) is (strongly) real, then $\Lambda_i(b)$ (resp.\ $\Lambda_i^{-1}(b)$) is also (strongly) real.
\end{cor}
\begin{NB}
\begin{proof}
Let $b_1, b_2\in \mscr{B}(\infty)$ with $\wt(b_1)=\xi_1, \wt(b_2)=\xi_2$.
Here we assume that $\vep_i^*(b_1)=\vep_{i}^*(b_2)=0$ and $b_1\bot b_2$, that is  we have $\Gup(b_1)\Gup(b_2)=q^N\Gup(b_1*b_2)$.
Let $\xi_1=\wt(b_1)$ and $\xi_2=\wt(b_2)$.
If we apply $(1-q_i^2)^{\bracket{h_i, \xi_1+\xi_2}}T_i(\Gup(b_1)\Gup(b_2))=q^N(1-q_i^2)^{\bracket{h_i, \xi_1+\xi_2}}T_i(\Gup(b_1*b_2))$.
Hence we obtain the claim.
\end{proof}
\end{NB}
\subsection{Compatibility with the dual canonical basis}
In this subsection,
we prove the compatibility of the dual canonical basis with the $\Qq$-subalgebra $\Uq^-(w, e)$.
This is a straightforward generalization of \cite[2.2 Proposition]{Cal:qflag} and \cite[Theorem 4.1]{Lampe}.
Here we fix $w\in W$ and $\tw\in R(w)$.
\begin{NB}
\begin{theo}\label{theo:unipgp}
	For $w\in W$ and $e\in \{\pm 1\}$, the subalgebra $\Uq^-(w, e)$ is compatible with the dual canonical basis,
	that is we have 
	\[\Uq^-(w, e)=\bigoplus_{b\in \mscr{B}(w, e)}\Qq\Gup(b).\]
\end{theo}
\end{NB}
\begin{theo}\label{theo:unipgp}
	For $w\in W$ and $e\in \{\pm 1\}$,
	the triple $(\Uq^-(w, e)\cap \mscr{L}(\infty), \Uq^-(w, e)\cap \sigma(\mscr{L}(\infty)), \Uq^-(w, e)\cap \Uq^{-}(\mfr{g})_{\mca{A}}^{\up})$ is balanced,
	in particular we have 
	\[\Uq^-(w, e)\cap \Uq^{-}(\mfr{g})_{\mca{A}}^{\up}=\bigoplus_{b\in \mscr{B}(w, e)}\mca{A}\Gup(b).\]
\end{theo}
\begin{NB}
\begin{rem}
In contrast to the above theorem,  we have $\Glow(b_e(\mbf{c}, \tw))\notin \Uq^-(w, e)$ in general.
\end{rem}
\end{NB}
The proof of this theorem occupies the rest of this subsection.
As in \cite[3.4]{LNT}, \cite[Proposition 31, Corollary 41]{Lec:qchar} and \cite{Cal:qflag}, we first prove that the dual root vectors are 
contained in the dual canonical basis and then prove the unitriangular property of upper global basis with respect to the dual PBW basis.
The compatibility with the dual canonical basis is its direct consequence.

Our proof needs an extra step from ones in \cite{LNT, Lec:qchar, Cal:qflag}, as it is not known that the PBW basis is an $\mca{A}$-basis of 
$\Uq^{-}(\mfr{g})_{\mca{A}}\cap \Uq^{-}(w)$ unless $\mfr{g}$ is of finite or affine type.
\subsubsection{}
\begin{prop}\label{prop:dualmultroot}
\textup{(1)}
For $i\in I$ and $n\geq 1$, let
\[F^{\up}(n\alpha_i):=\frac{f_i^{(n)}}{(f_i^{(n)}, f_i^{(n)})_K}.\]
Then we have $F^{\up}(n\alpha_i)\in \mbf{B}^{\up}$, $(F^{\up}_{\alpha_{i}})^{n}\in q^{\mbb{Z}}\mbf{B}^{\up}$ and 
$F^{\up}(n\alpha_{i})F^{\up}(m\alpha_{i})=q_{i}^{mn}F^{\up}((m+n)\alpha_{i})$

\textup{(2)}
For $n\geq 1$ and $1\leq k\leq l$, 
let  
 \[F^{\up}(n\beta_k):=\frac{F(n\beta_k)}{(F(n\beta_k), F(n{\beta_k}))_K}.\]
Then we have $F^{\up}(n\beta_k)\in \mbf{B}^{\up}$, $F^{\up}(\beta_k)^n\in q^{\mbb{Z}}\mbf{B}^{\up}$ and
$F^{\up}(n\beta_{k})F^{\up}(m\beta_{k})=q_{i_{k}}^{mn}F^{\up}((m+n)\beta_{k})$

\end{prop}

\begin{proof}
Since $F_i^{(n)}$ are the canonical base elements and $\dim \Uq^-(\mfr{g})_{-n\alpha_i}=1$ for any $n\geq 1$, $F^{\up}(n\alpha_i)$ are the dual canonical base elements by its definition.
By \propref{prop:innerprod} and \lemref{lem:KLforms}, we have 
\[(F_i^{(n)}, F_i^{(n)})_K=(1-q_i^2)^n/\prod_{j=1}^n (1-q_i^{2j})=\frac{1}{[n]_i!}q_i^{\frac{n(n-1)}{2}}.\]
Therefore we have  $(F^{\up}(\alpha_{i}))^{n}=q_i^{\frac{n(n-1)}{2}}F^{\up}(n\alpha_{i})\in \mbf{B}^{\up}$,
in particular $F^{\up}(\alpha_{i})$ is a strongly real element.
Applying \thmref{theo:braidupper}, we obtain the result for $F^{\up}(n\beta_{k})$ for $1\leq k\leq l$.
\end{proof}
\begin{NB}
We have 
\[F^{\up}(n\alpha_{i})F^{\up}(m\alpha_{i})=q_{i}^{mn}F^{\up}((m+n)\alpha_{i}).\]
\begin{proof}
Proof is straightforward.
By $(m+n)(m+n-1)=m(m-1)+n(n-1)+2mn$, we have
\[-\frac{(m+n)(m+n-1)}{2}+mn=-\frac{m(m-1)}{2}-\frac{n(n-1)}{2}.\]
Hence we have 
\begin{align*}
F^{\up}(n\alpha_{i})F^{\up}(m\alpha_{i})=&q_i^{-\frac{n(n-1)}{2}-\frac{m(m-1)}{2}}F^{\up}(\alpha_{i})^{n+m},\\
=&q_{i}^{-\frac{(m+n)(m+n-1)}{2}+mn}F^{\up}(\alpha_{i})^{n+m}.
\end{align*}
\end{proof}
\end{NB}

\begin{NB}
The following is a 1st proof.
\subsubsection{Computation of root vectors}
Here we prove that the normalized root vector (dual root vector) is contained the dual canonical basis.
\begin{prop}\label{prop:dual root}
Let $w\in W$ and $\tilde{w}=(i_1, i_2, \cdots, i_l)\in R(w)$ be its reduced expression.
Then we have 
\[F_{\beta_k}^*:=\frac{1}{(F_{\beta_k}, F_{\beta_k})_K}F_{\beta_k}\in \mbf{B}^{\up},\]
i.e, we have the following two conditions:

\textup{(1)} $(* \circ \overline{\phantom{x}})(F_{\beta_k}^*)=q^{-N(-\beta_k)}F_{\beta_k}^*$

\textup{(2)} $F_{\beta_k}^*\in \Uq^-(\mfr{g})_{\mca{A}}^*$

\textup{(3)} $F_{\beta_k}^* \mod q\mscr{L}(\infty)\in \mscr{B}(\infty).$

\end{prop}
The proof of the above proposition is straightforward generalization of \cite[Proposition 2.1]{Cal:qflag}.
\subsubsection{}
The following computation is straightforward generalization of \cite[Proposition 31]{Lec:qchar} for general symmetrizable Kac-Moody case.
For convenience of readers, we give a proof.
\begin{lem}\label{lem:dual root}
\textup{(1)}
	Let $w\in W$ and $\tw=(i_1, i_2, \cdots, i_l)\in R(w)$.
We have 
	\[(* \circ\overline{\phantom{x}})(T_{i_1}T_{i_2}\cdots T_{i_{k-1}}(F_{i_k}))=(-1)^{A(k)}q^{B(k)}T_{i_1}T_{i_2}\cdots T_{i_{k-1}}(F_{i_k}),\]
	where
	\[A(k)=\sum_{j=1}^{k-1}\bracket{h_{i_j}, s_{i_{j+1}}s_{i_{j+2}}\cdots s_{i_{k-1}}(\alpha_{i_k})},~~ %
	B(k)=\sum_{j=1}^{k-1}(\alpha_{i_j}, s_{i_{j+1}}s_{i_{j+2}}\cdots s_{i_{k-1}}(\alpha_{i_k})).\]
\textup{(2)} We have  $A(k)=\operatorname{ht}(\beta_k)-1$ and $B(k)=N(-\beta_k).$
\end{lem}
\begin{proof}
For convenience of the reader, we give a proof.
We follow the argument in \cite[Proposition 31]{Lec:qchar}.
By \cite[37.2.4]{Lus:intro} 
we have $T'_{i, 1}(x)=(-1)^{\bracket{h_i, \xi}}q^{-(\alpha_i, \xi)}T_{i}(x)$ for $x\in \Uq(\mfr{g})_{\xi}$.
Then since $* $ and $\overline{\phantom{x}}$ are involution,
\begin{align*}
(* \circ\overline{\phantom{x}})(T_{i_1}T_{i_2}\cdots T_{i_{k-1}}(F_{i_k}))
=&(* \circ\overline{\phantom{x}} \circ T_{i}\circ \overline{\phantom{x}}\circ* )(* \circ\overline{\phantom{x}}(T_{i_2}\cdots T_{i_{k-1}}(F_{i_k}))) \\
=&(-1)^{A'(k)}q^{B'(k)} (* \circ\overline{\phantom{x}} \circ T_{i}\circ \overline{\phantom{x}}\circ* )(T_{i_2}\cdots T_{i_{k-1}}(F_{i_k}))\\
=&(-1)^{A'(k)}q^{B'(k)} T'_{i, 1}(T_{i_2}\cdots T_{i_{k-1}}(F_{i_k})) ,
\end{align*}
where 
\[A'(k)=\sum_{j=2}^{k-1}\bracket{h_{i_j}, s_{i_{j+1}}s_{i_{j+2}}\cdots s_{i_{k-1}}(\alpha_{i_k})}, ~~%
B'(k)=\sum_{j=2}^{k-1}(\alpha_{i_j}, s_{i_{j+1}}s_{i_{j+2}}\cdots s_{i_{k-1}}(\alpha_{i_k})).\]
Then we obtain the assertion.
\textup{(2)}
Since $N(s_i\xi)=N(\xi)-(\xi, \alpha_i)$ and $N(-\alpha_i)=0$,
we have $N(-\beta_k)=N(-s_{i_1}\cdots s_{i_{k-1}}\alpha_{i_k})=N(s_{i_2}\cdots s_{i_{k-1}}(-\alpha_{i_k}))-(s_{i_2}\cdots s_{i_{k-1}}(-\alpha_{i_k}), \alpha_{i_1})=\cdots=A(k)$.
\end{proof}

\begin{lem}
For $1\leq k\leq l$, $c\in \mbb{Z}_{\geq 1}$,
we set $\mbf{c}_k$ by $(\mbf{c}_k)_{j}=c\delta_{k, j}$, 
then we have 
\begin{equation}
F^{\up}(\mbf{c}_k):=\frac{1}{(F(\mbf{c}_k), F(\mbf{c}_k))_K}F(\mbf{c}_k)=q_{i_k}^{-\frac{c(c-1)}{2}}(F^{\up}_{\beta_k})^{c}\in \Uq^-(\mfr{g})^*_{\mca{A}}.\label{eq:dualpower}
\end{equation}
\end{lem}
\begin{proof}
By \propref{prop:innerprod} and \lemref{lem:KLforms}, we have 
\[(F_i^{(n)}, F_i^{(n)})_K=(1-q_i^2)^n/\prod_{j=1}^n (1-q_i^{2j})=\frac{1}{[n]_i!}q_i^{\frac{n(n-1)}{2}}.\]
We obtain \eqref{eq:dualpower}.
By its definition, we have $(F_i^{n})^*:=F_i^{(n)}/(F_i^{(n)}, F_i^{(n)})_K\in \Uq^-(\mfr{g})_{\mca{A}}^*$.
Then we have $F^{\up}(\mbf{c}_k)\in \Uq^-(\mfr{g})_{\mca{A}}^*$ by its integrality properties of the braid group action \cite[Chapter 41]{Lus:intro}.
\end{proof}

\begin{pf*}{Proof of \propref{prop:dual root}}
Here we show the characterization of dual canonical basis by \propref{prop:dualtriple}.
First we compute an action of dual bar involution on dual root vector.
We set $\kappa_{\beta_k}:=(F_{\beta_k}, F_{\beta_k})_K$ and $F_{\beta_k}^*:=F_{\beta_k}/\kappa_{\beta_k}$.
First we note that we have $(*\circ \overline{\phantom{x}})(F_{\beta_k})=(-1)^{ht(\beta_k)-1}q^{N(-\beta_k)}F_{\beta_k}$.
Since 
\[\kappa_{\beta_k}=\prod_{i\in I}(q_i^{-1}-q_i)^{-\xi_i}q_i^{-\xi_i}\frac{q_{i_k}^{-1}}{q_{i_k}^{-1}-q_{i_k}},\]
where $\beta_k=\sum \xi_i\alpha_i$ with $\xi\in\mbb{Z}_{\leq 0}$,
we have $\overline{\kappa_{\beta_k}}=(-1)^{ht(\beta_k)-1}q^{2N(-\beta_k)}\kappa_{\beta_k}$.
Then we have a condition (1), that is 
\[(*\circ \overline{\phantom{x}})(F^{\up}_{\beta_k})=q^{-N(-\beta_k)}F^{\up}_{\beta_k}.\]

Since the normalization factor $\kappa_{\beta_k}=\prod_{i\in I}(1-q_{i}^2)^{-\xi_i}(1-q_{i_k}^2)^{-1}$ is a rational function  in $q$ which is regular at $q=0$ and $\kappa_{\beta_k}(0)=1$,
then $F^{\up}_{\beta_k}\equiv F_{\beta_k}$.
Hence we have $F^{\up}_{\beta_k}\in \mbf{B}^{\up}$.
\end{pf*}
By \ref{eq:dualpower}, \propref{prop:dual root} and \eqref{eq:dualbar},
we have 
\begin{cor}\label{cor:multipleroot}
For $1\leq k\leq l$, 
we have 
$F^{\up}(\mbf{c}_k)=q_{i_k}^{-\frac{c(c-1)}{2}}(F^{\up}_{\beta_k})^{c}\in \mbf{B}^{\up}$
\end{cor}
\end{NB}
\begin{NB}In particular, (dual) root vector is (strongly) real.\end{NB}

\subsubsection{}
For the computation of the action of the dual bar involution $\sigma$,
we need an integrality property of the Levendorskii-Soibelman's formula for the dual root vectors and its multiple.
For $w\in W$, $\tw\in R(w)$ and $e\in \{\pm 1\}$,
we set
\[F^{\up}_e(\mbf{c}, \tw):=\frac{1}{(F_e(\mbf{c}, \tw), F_e(\mbf{c}, \tw))_K}F_e(\mbf{c}, \tw).\]
This is the dual basis  of $\{F_e(\mbf{c}, \tw)\}$ with respect to Kashiwara's bilinear form $(\cdot, \cdot)_{K}$.
As before, when we consider only the $e=1$ case, we omit the subscript $e$.
\begin{theo}[dual Levendorskii-Soibelman formula]\label{theo:dualLS}
For $j<k$,
we write
\[F^{\up}(c_k\beta_k)F^{\up}(c_j\beta_j)-q^{-(c_j\beta_j, c_k\beta_k)} F^{\up}(c_j\beta_j)F^{\up}(c_k\beta_k)=\sum_{}f^{*}_{\mbf{c}'}
F^{\up}(\mbf{c}', \tw).\]
Then $f^*_{\mbf{c}'}\in \mca{A}$ and %
if $f^*_{\mbf{c}'}\neq 0$, then $c'_{j}<c_{j}$ and $c'_{k}<c_{k}$
with $\sum_{j\leq m\leq k} c'_{m}\beta_m=c_{j}\beta_j+c_{k}\beta_k$.

\end{theo}
\begin{proof}
Firstly, a weaker statement that $f^{*}_{\mbf{c}'}\in \Qq$ with the above conditions follows from \thmref{theo:LSformula}
and \propref{prop:innerprod}.
Let us prove that $f^{*}_{\mbf{c}'}\in \mca{A}$.
Since the twisted coproduct $r$ preserves the $\mca{A}$-form $\Uq^-(\mfr{g})_{\mca{A}}$,
the dual integral form $\Uq^-(\mfr{g})_{\mca{A}}^{\up}$ is an $\mca{A}$-subalgebra of $\Uq^-(\mfr{g})$.
Therefore the left hand side belongs to $\Uq^{-}(\mfr{g})_{\mca{A}}^{\up}$ by  \propref{prop:dualmultroot}.
Taking the inner product with $F(\mbf{c}', \tw)$, we find $f^{*}_{\mbf{c}'}\in \mca{A}$ thanks to \thmref{thm:PBW}.
\end{proof}

In particular, we have the $\mca{A}$-subalgebra $\Uq^-(w, e)_{\mca{A}}^{\up}$ of $\Uq^-(w, e)$
generated by $\{F^{\up}(c\beta_k)\}_{1\leq k\leq l, c\geq 1}$.
This subalgebra has a free $\mca{A}$-basis $\{F^{\up}_e(\mbf{c}, \tw); \mbf{c}\in \mbb{Z}_{\geq 0}^l\}.$
We call this base the \emph{dual PBW basis}.
\begin{NB}
By the above theorem, we obtain the following integral version of Levendorskii-Soibelman's formula.
\begin{theo}[dual Levendorskii-Soibelman formula]\label{theo:dualLS}
For $j<k$,
we write
\[F(c_k\beta_k)F(c_j\beta_j)-q^{-(c_j\beta_j, c_k\beta_k)} F(c_j\beta_j)F(c_k\beta_k)=\sum_{}f_{\mbf{c}'}
F(\mbf{c}', \tw).\]
Then $f_{\mbf{c}'}\in \mca{A}$ and %
if $f^*_{\mbf{c}'}\neq 0$, then $c'_{j}<c_{j}$ and $c'_{k}<c_{k}$
with $\sum_{j\leq m\leq k} c'_{m}\beta_m=c_{j}\beta_j+c_{k}\beta_k$.
\end{theo}

\end{NB}

\subsubsection{}
We compute the action of the dual bar involution $\sigma$ on the dual PBW basis.
The following is straightforward generalization of \cite[2.1 Corollary (i)]{Cal:qflag},
and follows from \eqref{eq:dualbar} and \thmref{theo:dualLS}.
\begin{prop}\label{prop:dualbarPBW}
We have 
\begin{NB}\textup{(1)}
\[(*\circ \overline{\phantom{x}})
(F(\mbf{c}, \tw))= (-1)^{A(\mbf{c})}q^{B(\mbf{c})}F(\mbf{c}, \tw) + \sum_{\mbf{c'}<\mbf{c}}f_{\mbf{c}, \mbf{c'}}(q)F(\mbf{c}, \tw),\]
where $f_{\mbf{c}, \mbf{c'}}(q)\in \mca{A}$, $A(\mbf{c})=?$ and $B(\mbf{c})=?$.

\textup{(2)}
\[(*\circ \bar)(F^{\up}(\mbf{c}, \tw))=q^{-N(\mbf{c})}F^{\up}(\mbf{c}, \tw)+ \sum_{\mbf{c'}<\mbf{c}}f^*_{\mbf{c}, \mbf{c'}}(q)F^{\up}(\mbf{c}', \tw),\]
where $f^*_{\mbf{c}, \mbf{c'}}(q)\in \mca{A}$.
\end{NB}
\[\sigma(F^{\up}(\mbf{c}, \tw))=F^{\up}(\mbf{c}, \tw)+ \sum_{\mbf{c'}<\mbf{c}}f^*_{\mbf{c}, \mbf{c'}}(q)F^{\up}(\mbf{c}', \tw),\]
where $f^*_{\mbf{c}, \mbf{c'}}(q)\in \mca{A}$.
\end{prop}
\begin{NB}
\begin{proof}
First proof
In \textup{(2)}, we note that a weaker statement that $f^*_{\mbf{c}, \mbf{c'}}\in \Qq$ easily follows from \textup{(1)}.
So only we have to show is $f\in \mca{A}$. 
By the definition of the dual bar involution, we know that the dual $\mca{A}$-form of $\Uq^-(\mfr{g})$ is stable under the dual bar involution
and the integrality that $F(\mbf{c}, \tw)\in \Uq^-(\mfr{g})_{\mca{A}}$.
Then the coefficient $f^*_{\mbf{c}, \mbf{c'}}=((*\circ \bar)(F^{\up}(\mbf{c}, \tw)), F(\mbf{c}', \tw))_{K}\in \mca{A}$, so we have obtained \textup{(1)}.
So from here, we prove \textup{(1)}.
\begin{align*}
(*\circ \overline{\phantom{x}})(F_{\tw, \mbf{c}})
=&(*\circ \overline{\phantom{x}})(F_{\beta_1}^{(c_1)}F_{\beta_2}^{(c_2)}\cdots F_{\beta_l}^{(c_l)}) \\
=&(*\circ \overline{\phantom{x}})(F_{\beta_l}^{(c_l)})\cdots (*\circ \overline{\phantom{x}})(F_{\beta_2}^{(c_2)})(*\circ \overline{\phantom{x}})(F_{\beta_1}^{(c_1)}) \\
=&(-1)^{\sum_k c_k(\tr(\beta_k)-1)}q^{\sum_kc_kN(\beta_k)}F_{\beta_l}^{(c_l)}F_{\beta_{l-1}}^{(c_{l-1})}\cdots F_{\beta_1}^{(c_1)} \\
\equiv &(-1)^{\sum_k c_k(\tr(\beta_k)-1)}q^{\sum_kc_kN(\beta_k)+\sum_{k<l}c_kc_l(\beta_k, \beta_l)}F_{\beta_1}^{(c_1)}F_{\beta_2}^{(c_2)}\cdots F_{\beta_l}^{(c_l)}
\mod \mca{F}^{\tw}_{\leq \mbf{c}}
\end{align*}
We have 
\begin{align*}
\sigma(F^{\up}(\mbf{c}))=&q^{\sum_{i<j}(c_i\beta_i, c_j\beta_j)}\sigma(F^{\up}(\mbf{c}_l))\cdots \sigma(F^{\up}(\mbf{c}_1)) && \text{by \eqref{eq:dualbar}} \\
=&F^{\up}(\mbf{c})+\sum_{\mbf{c'}<\mbf{c}}f^*_{\mbf{c}, \mbf{c'}}(q)F^{\up}(\mbf{c}', \tw) && \text{by \lemref{lem:dualLS}}
\end{align*}
\end{proof}
\end{NB}

\subsubsection{}
\begin{theo}\label{theo:dualcan2}
\textup{(1)}
Let $w\in W$ and $\tw\in R(w)$.
Then there exists a unique $\mca{A}$-basis $\{B^{\up}(\mbf{c}, \tw) ; \mbf{c}\in \mbb{Z}_{\geq 0}^l\}$ of $\Uq^{-}(w, e)_{\mca{A}}^{\up}$ with the following properties: 
\begin{subequations}
\begin{align}
	&\sigma(B^{\up}(\mbf{c}, \tw))=B^{\up}(\mbf{c}, \tw), \label{eq:dualPBWcan1} \\
	&F^{\up}(\mbf{c}, \tw)=B^{\up}(\mbf{c}, \tw) + \sum_{\mbf{c'}<\mbf{c}}\varphi_{\mbf{c}, \mbf{c'}}B^{\up}(\mbf{c}', \tw), \varphi_{\mbf{c}, \mbf{c'}}\in q\mbb{Z}[q]. \label{eq:dualPBWcan2}
\end{align}
\end{subequations}
\textup{(2)}
	We have $B^{\up}(\mbf{c}, \tw)=\Gup(b(\mbf{c}, \tw)).$
\end{theo}
\begin{proof}
The proof of \textup{(1)} is the same as one for the existence of Kazhdan-Lusztig polynomials.
The only claim we need is \propref{prop:dualbarPBW}.

\textup{(2)}
Since we have $f_{\mbf{c}}(q)=(F(\mbf{c}, \tw), F(\mbf{c}, \tw))_K \in \mca{A}_0$ and  $f_{\mbf{c}}(0)=1$,
we obtain
\[B^{\up}(\mbf{c}, \tw)\equiv F^{\up}(\mbf{c}, \tw)\equiv b(\mbf{c}, \tw) \mod q\mscr{L}(\infty).\]
Therefore (2) follows from (1) and \propref{prop:dualtriple}.
\begin{NB}
\textup{(2)} follows easily from \textup{(1)} and the fact that $F^{\up}(\mbf{c}, \tw)=f_{\mbf{c}}(q)F(\mbf{c}, \tw)$
with $f_{\mbf{c}}(q)\in A_0$ and $f_{\mbf{c}}(0)=1$.
The proof of \textup{(1)}  is  the same as one for the existence of (dual) Kazhdan-Lusztig polynomials.
The uniqueness is clear from its proof.
We prove for convenience for the readers.
We use the (ascending) induction on the lexicographic order $\leq_{\tw} $.
Since the space $\Uq^-(w)$ has weight space decomposition, we construct the basis on each weight space.
For minimal $\mbf{c}$, we have $B^{\up}(\mbf{c}, \tw)=F^{\up}(\mbf{c}, \tw)$ by \propref{prop:dualbarPBW}.
We assume that $B^{\up}(\mbf{c}', \tw)$ is already obtained for $\mbf{c'}<\mbf{c}$.
Since \propref{prop:dualbarPBW}  and induction hypothesis, we have 
\begin{align}
(*\circ \overline{\phantom{x}})(F^{\up}(\mbf{c}, \tw))=&q^{-N(\wt(\mbf{c}))}F^{\up}(\mbf{c}, \tw)+%
\sum_{\mbf{c'}<_{\tw}\mbf{c}}a_{\mbf{c}, \mbf{c'}}F^{\up}(\mbf{c}', \tw) ~(\because \text{\propref{prop:dualbarPBW} }), \\
=&q^{-N(\wt(\mbf{c}))}F^{\up}(\mbf{c}, \tw)+\sum_{\mbf{c'}<_{\tw}\mbf{c}}b_{\mbf{c}, \mbf{c'}}B^{\up}(\mbf{c}', \tw) ~(\because \text{induction hypothesis}),
\end{align}
\beq  (*\circ \overline{\phantom{x}})(F^{\up}(\mbf{c}, \tw))=q^{-N(\wt(\mbf{c}))}F^{\up}(\mbf{c}, \tw)+\sum_{\mbf{c'}<_{\tw}\mbf{c}}f^*_{\mbf{c}, \mbf{c'}}B^{\up}(\mbf{c}', \tw)%
\label{eq:dualPBWcan3} \eeq 
where $f^*_{\mbf{c}, \mbf{c'}}\in \mca{A}$.
Here we apply $*\circ \bar$, then we have 
\beq F^{\up}(\mbf{c}, \tw)=q^{N(\wt(\mbf{c}))}F^{\up}(\mbf{c}, \tw)+\sum_{\mbf{c'}<_{\tw}\mbf{c}}\overline{f^*_{\mbf{c}, \mbf{c'}}}q^{-N(\wt(\mbf{c'}))}B^{\up}(\mbf{c}', \tw)%
\label{eq:dualPBWcan4}.\eeq
Since $\wt(\mbf{c})=\wt(\mbf{c'})$, then we have $N(\wt(\mbf{c}))=N(\wt(\mbf{c}'))$.
By comparing \eqref{eq:dualPBWcan3} and \eqref{eq:dualPBWcan4},
we obtain
\[q^{N(\wt(c))}f^*_{\mbf{c}, \mbf{c'}}=-q^{-N(\wt(c))}\overline{f^*_{\mbf{c}, \mbf{c'}}}=-\overline{q^{N(\wt(\mbf{c}))}f^*_{\mbf{c}, \mbf{c'}}},\]
then there exists $\varphi_{\mbf{c}, \mbf{c'}}\in q\mbb{Z}[q]$ such that 
$q^{N(\wt(c))}f^*_{\mbf{c}, \mbf{c'}}=\varphi_{\mbf{c}, \mbf{c'}}(q)-\varphi_{\mbf{c}, \mbf{c'}}(q^{-1})$.
Here we set 
\[B^{\up}(\mbf{c}, \tw)=F^{\up}(\mbf{c}, \tw)+\sum_{\mbf{c}'<_{\tw}\mbf{c}}\varphi_{\mbf{c}, \mbf{c'}}(q)B^{\up}(\mbf{c}', \tw),\]
then we can easily check that $B^{\up}(\mbf{c}, \tw)$ satisfies the conditions \eqref{eq:dualPBWcan1} and \eqref{eq:dualPBWcan2}.
This finishes the proof.
\end{NB}
\end{proof}
As a corollary, we have $\Uq^{-}(w, e)_{\mca{A}}^{\up}=\Uq^{-}(w, e)\cap \Uq^{-}(\mfr{g})_{\mca{A}}^{\up}$
since $\{\Gup(b)\}_{b\in \mscr{B}(\infty)}$ is an $\mca{A}$-basis of $\Uq^{-}(\mfr{g})_{\mca{A}}^{\up}$.
Together with this result, \thmref{theo:dualcan2} implies \thmref{theo:unipgp}.
\subsection{}
In this subsection, we study basic commutation relation among the dual canonical basis of $\Uq^{-}(w, e=1)$.
The following is a generalization of \cite[Proposition 4.2]{Rei:mult}, and follows from the characterization of dual canonical basis in terms of dual PBW basis.
For $\mbf{c}, \mbf{c}'\in \mbb{Z}_{\geq 0}^{l}$, we set
\[c_{\tw}(\mbf{c}, \mbf{c}')=\sum_{l<k}(c_k\beta_k, c'_{l}\beta_l)-\frac{1}{2}\sum_{k}c_{k}c'_{k}(\beta_{k}, \beta_{k}).\]
\begin{prop}\label{prop:dsum_formula}
We have 
\[\Gup(b(\mbf{c}, \tw))\Gup(b(\mbf{c}', \tw))=q^{-c_{\tw}(\mbf{c}, \mbf{c}')}\Gup(b(\mbf{c}+\mbf{c}', \tw))+\sum d_{\mbf{c}, \mbf{c'}}^{\mbf{d}}(q)\Gup(b(\mbf{d}, \tw)),\]
where $\mbf{d}< \mbf{c}+\mbf{c'}$ and  $d_{\mbf{c}, \mbf{c'}}^{\mbf{d}}(q)\in \mca{A}$.
\end{prop}
\begin{NB}
\begin{proof}
To prove the statement, it suffices  to compute the top term of the DeConcini-Kac filtration.
We have
\begin{align*}
&\gr^{\tw}(\Gup(b(\mbf{c}, \tw))\Gup(b(\mbf{c}', \tw)) \\
=&\gr^{\tw}(\Gup(b(\mbf{c}, \tw))\gr^{\tw}(\Gup(b(\mbf{c}', \tw)) \\
=&\gr^{\tw}(F^{\up}(\mbf{c}, \tw))\gr^{\tw}(F^{\up}(\mbf{c}', \tw))~&&\text{by \eqref{eq:dualPBWcan2}},\\
=&q^{-c_{\tw}(\mbf{c}, \mbf{c'})} \gr^{\tw}(F^{\up}(\mbf{c}+\mbf{c'}, \tw))~&&\text{by \thmref{theo:dualLS} and \propref{prop:dualmultroot}}, \\
=&q^{-c_{\tw}(\mbf{c}, \mbf{c'})} \gr^{\tw}(\Gup(b(\mbf{c}, \tw)) && \text{by \eqref{eq:dualPBWcan2}}.
\end{align*}
Then the assertion follows.
\end{proof}
\end{NB}
\begin{cor}\label{cor:dsum}
If $b(\mbf{c}, \tw)\bot b(\mbf{c}', \tw)$, we have
\[\Gup(b(\mbf{c}, \tw))\Gup(b(\mbf{c}, \tw))\simeq \Gup(b(\mbf{c}+\mbf{c}', \tw)),\]
that is $b(\mbf{c}, \tw)\circledast b(\mbf{c}', \tw)=b(\mbf{c}+\mbf{c}', \tw)$.
\end{cor}
\subsubsection{}
Using \propref{prop:dsum_formula}, we have the following expression of $q$-power 
of the $q$-commuting dual canonical basis elements in $\mscr{B}(w, e=1)$ as in \cite[Proposition 18]{LNT}.
\begin{prop}
If $G^{\up}(b(\mbf{c}, \tw))G^{\up}(b(\mbf{c}', \tw))=q^{-N_{\tw}(\mbf{c}, \mbf{c}')}G^{\up}(b(\mbf{c}', \tw))G^{\up}(b(\mbf{c}, \tw))$,
then we have 
\begin{equation}N_{\tw}(\mbf{c}, \mbf{c}')=c_{\tw}(\mbf{c}, \mbf{c}')-c_{\tw}(\mbf{c}', \mbf{c}).\label{eq:lambda}\end{equation}
\end{prop}

\begin{NB}
For $\mbf{c}, \mbf{c'}\in \mbb{Z}_{\geq 0}^l$,
we set 
\[U(\mbf{c}, \mbf{c}'):=\frac{q^{c_{\tw}(\mbf{c}, \mbf{c}')+1}B^{\up}(\mbf{c})B^{\up}(\mbf{c}')-q^{c_{\tw}(\mbf{c}', \mbf{c})-1}B^{\up}(\mbf{c}')B^{\up}(\mbf{c})}{q-q^{-1}}.\]
The following holds, this is same as in \cite[Proposition 18]{LNT}.
This follows from \propref{prop:dsum_formula} and \corref{cor:inverse}
\begin{lem}
Set $U_{\tw, e}(\mbf{c}, \mbf{c}')=\sum \beta_{\mbf{c}, \mbf{c'}}^{\mbf{d}}(q)\Gup(b_e(\mbf{d}, \tw))$,
then we have the following:

\textup{(1)} $\beta_{\mbf{c}, \mbf{c'}}^{\mbf{d}}(q) \in \mca{A}$ and $\beta_{\mbf{c}, \mbf{c'}}^{\mbf{c}+\mbf{c'}}(q)=1$.

\textup{(2)} $\beta_{\mbf{c}, \mbf{c'}}^{\mbf{d}}(q)\neq 0$, then  $\mbf{d}\leq \mbf{c}+\mbf{c'}$.

\textup{(3)} $\overline{\beta_{\mbf{c}, \mbf{c}'}(q)}=\beta_{\mbf{c}, \mbf{c}'}(q)$.
\end{lem}
\end{NB}

\subsection{}\label{sec:quantumunip}
In this subsection, we recall the specialization of $\Uq^{-}(w, e)$ at $q=1$.
\subsubsection{}
We have the following property of the specialization of $\Uq^{-}$ at $q=1$.
\begin{theo}[{\cite[\S 33.1]{Lus:intro}}]\label{prop:spec}
There is an isomorphism of algebras:
\[\Phi\colon U(\mfr{n})\xrightarrow{\sim} \mbb{C}\otimes_{\mca{A}}\Uq^{-}(\mfr{g})_{\mca{A}}\]
which sends $f_{i}$ to $f_{i}$.
\end{theo}
Let $r\colon U(\mfr{n})\to U(\mfr{n})\otimes U(\mfr{n})$ be the coproduct defined by $r(f)=f\otimes 1+1\otimes f$ for $f\in \mfr{n}$.
Here we note that $U(\mfr{n})$ is generated by $\{f_{i}\}_{i\in I}$ as algebra.
Since the specializatioin of the twisted coproduct satisfies this relation on the generators, the above is an isomorphism of bialgebras.
\subsubsection{}
Let $\mbb{C}[N]$ be the restricted dual of the universal enveloping algebra $U(\mfr{n})$ of the Lie algebra $\mfr{n}$,
that is
\[\mbb{C}[N]:=\bigoplus_{\xi\in Q}U(\mfr{n})_\xi^*.\]
We take the dual $\Uq^{-}(\mfr{g})_{\mca{A}}^{\up}$ of $\Uq^-(\mfr{g})_{\mca{A}}$ as before.
Since the multiplication of $\Uq^{-}(\mfr{g})$ preserves $\Uq^{-}(\mfr{g})_{\mca{A}}$, 
the twisted coproduct $r$ preserves the dual integral form $\Uq^{-}(\mfr{g})_{\mca{A}}^{\up}$,
that is $r(\Uq^{-}(\mfr{g})_{\mca{A}}^{\up})\subset \Uq^{-}(\mfr{g})_{\mca{A}}^{\up}\otimes \Uq^{-}(\mfr{g})_{\mca{A}}^{\up}$.

Let  $r^{*}\colon \mbb{C}[N]\otimes \mbb{C}[N]\to \mbb{C}[N]$ be a product so that 
$\bracket{r^{*}(\varphi\otimes \varphi'), x}=\bracket{\varphi\otimes \varphi', r(x)}$ holds for any $x\in U(\mfr{n})$ and 
 $\mu^{*}\colon \mbb{C}[N]\to \mbb{C}[N]\otimes \mbb{C}[N]$ be a coproduct so that
$\bracket{\mu^{*}(\varphi), x\otimes x'}=\bracket{\varphi, \mu(x\otimes x')}$ holds for any $x, x'\in U(\mfr{n})$,
where $\mu\colon U(\mfr{n})\otimes U(\mfr{n})\to U(\mfr{n})$ is the product on $U(\mfr{n})$.
The above isomorphism $\Phi$ induces the following.

\begin{prop}\label{prop:dualspec}
There is an isomorphism of bialgebras
\[\Phi^{\up}\colon \mbb{C}\otimes_{\mca{A}}\Uq^{-}(\mfr{g})_{\mca{A}}^{\up}\xrightarrow{\sim}\mbb{C}[N],\]
that is we have 
\begin{align*}
\mu^{*}\circ \Phi^{\up}&=(\Phi^{\up}\otimes \Phi^{\up})\circ r, \\
r^{*}\circ (\Phi^{\up}\otimes \Phi^{\up})&=\Phi^{\up}\circ \mu.
\end{align*}
\end{prop}

\subsubsection{}
Let 
\begin{align*}
\sigma_{i}:=&\;\exp(-f_{i})\exp(e_{i})\exp(-f_{i})\\
=&\;\exp(e_{i})\exp(-f_{i})\exp(e_{i}),
\end{align*}
for $i\in I$.
Then we have 
\begin{align*}
(\sigma_{i})^{-1}=&\;\exp(f_{i})\exp(-e_{i})\exp(f_{i})\\
=&\;\exp(-e_{i})\exp(f_{i})\exp(-e_{i}).
\end{align*}
(This $(\sigma_{i})^{-1}$ is equal to $\overline{s}_{i}$ used in \cite[7.1]{GLS:KacMoody}.)
The action of $\sigma_{i}$ is well-defined on integrable $\mfr{g}$-modules, especially on the adjoint representation of $\mfr{g}$.
Under the specialization at $q=1$, we have $\sigma_{i}=S_{i}|_{q=1}$.
\subsubsection{}
For $\tw\in R(w)$ and $e\in \{\pm 1\}$, 
let 
\[f_e(\beta_k):=\sigma_{i_{1}}^{e}\cdots \sigma_{i_{k-1}}^{e}(f_{i_{k}}).\]
Then we have $f_e(\beta_k)\in \mfr{g}_{-\beta_{k}}$ and 
\[\mfr{n}(w)=\bigoplus_{1\leq k\leq l}\mbb{C}f_e(\beta_k).\]
By the definition, $f_e(\beta_k)$ is the specialization of $F_e(\beta_k)$.
\begin{NB}
signs can be differ by the choice of $e\in \{\pm 1\}$
\end{NB}
\subsubsection{}
Let $\mbb{C}[N(w)]$ be the restricted dual of the universal enveloping algebra $U(\mfr{n}(w))$ associated with $\mfr{n}(w)$.
We consider a basis of $\mfr{n}(w)$ given by $\{f_e(\beta_k)\}_{1\leq k\leq l}$ and also
a basis $\{f_e(\beta_k)\}_{1\leq k\leq l}\cup \{f'_{k}\}$ of $\mfr{g}$ which includes $\{f_e(\beta_k)\}_{1\leq k\leq l}$ as in \cite[4.3]{GLS:KacMoody}.
Here we fix a total order on the basis of $\mfr{g}$ by
\[f_e(\beta_1)<\cdots<f_e(\beta_k)<f'_1<f'_2<\cdots.\]
By the Poincar\'{e}-Birkhoff-Witt basis theorem, we have a basis of $U(\mfr{n})$ given by 
\[f_e((\mbf{c}, \mbf{d}), \tw):=\begin{cases}f_{e}(\beta_1)^{(c_1)}\cdots f_{e}(\beta_l)^{(c_l)}{f'_1}^{(d_1)}\cdots & \text{~when~} e=1, \\%
\cdots {f'_1}^{(d_1)}f_{e}(\beta_l)^{(c_l)}\cdots f_{e}(\beta_1)^{(c_1)} &\text{~when~} e=-1 ,\end{cases}\]
and also a basis of $U(\mfr{n}(w))$ given by 
\[f_e(\mbf{c}, \tw):=\begin{cases}f_{e}(\beta_1)^{(c_1)}\cdots f_{e}(\beta_l)^{(c_l)} & \text{~when~} e=1, \\%
f_{e}(\beta_l)^{(c_l)}\cdots f_{e}(\beta_1)^{(c_1)} &\text{~when~} e=-1 ,\end{cases}\]
where $x^{(c)}=x^c/c!$ for $x\in \mfr{g}$ and $c\in \mbb{Z}_{\geq 0}$.
We have $\Phi(f_{e}(\mbf{c}, \tw))=F_{e}(\mbf{c}, \tw)|_{q=1}$.
\subsubsection{}
Let $\{f_{e}^{*}(\mbf{c}, \tw)\}$ (resp.\ $\{f_{e}^{*}((\mbf{c}, \mbf{d}), \tw)\}$) be the dual basis of $\{f_{e}(\mbf{c}, \tw)\}$
(resp.\ $\{f_{e}((\mbf{c}, \mbf{d}), \tw)\}$).
Using these, we obtain a section of $\mbb{C}[N]\to \mbb{C}[N(w)]$ as algebras.
\begin{lem}
Let $\widetilde{\pi}_{w}^{*}\colon \mbb{C}[N(w)]\to \mbb{C}[N]$ be a $\mbb{C}$-linear homomorphism defined by 
\[\widetilde{\pi}_{w}^{*}(f^{*}_{e}(\mbf{c}, \tw)):=f_{e}^{*}((\mbf{c}, 0), \tw).\]
Then it is an algebra embedding.
\end{lem}
\begin{proof}
First $\bracket{\widetilde{\pi}_{w}^{*}(f_{e}^{*}(\mbf{c}_{1}, \tw))\cdot \widetilde{\pi}_{w}^{*}(f_{e}^{*}(\mbf{c}_{2}, \tw)), f_{e}((\mbf{c'}, \mbf{d}'), \tw)}$
is equal to 
\[\bracket{\widetilde{\pi}_{w}^{*}(f_{e}^{*}(\mbf{c}_{1}, \tw))\otimes\widetilde{\pi}_{w}^{*}(f_{e}^{*}(\mbf{c}_{2}, \tw)), r(f_{e}((\mbf{c'}, \mbf{d}'), \tw))}.\]
We note that 
\begin{equation}
r(f_{e}((\mbf{c'}, \mbf{d}'), \tw))=\sum_{\mbf{c'}_{1}+\mbf{c'}_{2}=\mbf{c'}, \mbf{d}'_{1}+\mbf{d}'_{2}=\mbf{d}'}f_{e}((\mbf{c'}_{1}, \mbf{d}'_{1}), \tw))\otimes f_{e}((\mbf{c'}_{2}, \mbf{d}'_{2}), \tw)).\label{eq:PBW}\end{equation}

Hence the above is equal to $\delta_{\mbf{c}_{1}+\mbf{c}_{2}, \mbf{c}'}\delta_{0, \mbf{d}'}$.
On the other hand, we consider 
\[\bracket{\widetilde{\pi}_{w}^{*}(f_{e}^{*}(\mbf{c}_{1}, \tw)\cdot f_{e}^{*}(\mbf{c}_{2}, \tw)), f_{e}((\mbf{c'}, \mbf{d}'), \tw)}.\]
By \eqref{eq:PBW}, we have $f_{e}^{*}(\mbf{c}_{1}, \tw)\cdot f_{e}^{*}(\mbf{c}_{2}, \tw):=r^{*}(f_{e}^{*}(\mbf{c}_{1}, \tw)\otimes f_{e}^{*}(\mbf{c}_{2}, \tw))=f_{e}^{*}(\mbf{c}_{1}+\mbf{c}_{2}, \tw)$.
Then the above is equal to $\bracket{\widetilde{\pi}_{w}^{*}(f_{e}^{*}(\mbf{c}_{1}+\mbf{c}_{2}, \tw)), f_{e}((\mbf{c'}, \mbf{d}'), \tw)}=\delta_{\mbf{c}_{1}+\mbf{c}_{2}, \mbf{c}'}\delta_{0, \mbf{d}'}$.
Then the assertion holds.
\end{proof}
By \cite[Proposition 8.2]{GLS:KacMoody}, this embedding does not depend on the choice of $\tw\in R(w)$ and of the basis of $\mfr{g}$.

\subsubsection{}
We study the image of $\Uq^{-}(w, e)_{\mca{A}}^{\up}\otimes_{\mca{A}}\mbb{C}$ under the isomorphism $\Phi^{\up}$.
\begin{lem}\label{lem:rootvector}
Let $f\in \mfr{g}_{\alpha}$ with $\alpha\in \Delta_{+}\setminus \Delta_+(w)$, we have 
\[\bracket{f, \Phi^{\up}(\Gup(b)|_{q=1})}=0\]
for $b\in \mscr{B}(w, e)$.
\end{lem}
\begin{proof}
Suppose that $b\in \mscr{B}(w, e)$ and $f\in \mfr{g}_{\alpha}$ with $\bracket{f, \Phi^{\up}(\Gup(b)|_{q=1})}\neq 0$.
Then we have
\[\alpha=\sum_{1\leq k\leq l} a_{k}\beta_{k}\]
for some $a_{k}\in \mbb{Z}_{\geq 0}$.
By the definition of $\Delta_{+}(w)$, we have $w^{-1}\alpha\in \Delta_{+}$ and $w^{-1}(\sum_{1\leq k\leq l} a_{k}\beta_{k})\in Q_{-}$.
This is a contradiction.
Hence we get the assertion.
\begin{NB}
The following argument is not correct.
Since $\Delta^{+}(w)$ is a bracket closed subset, we have $\alpha=\beta_{k}$ for some $1\leq k\leq l$.
\end{NB}
\end{proof}
\subsubsection{}\label{sec:partialR}
We have the following formula of the (twisted) coproduct of the root vectors $F(\beta_{k})$, see \cite[3.5 Corollary 3]{DamDec}.
\begin{prop}\label{prop:coprod_root}
We have the following expansion:
\[r(F(\beta_{k}))-(1\otimes F(\beta_{k})+F(\beta_{k})\otimes 1)=\sum_{\mbf{c}}x_{\mbf{c}}\otimes F(\mbf{c}, \tw),\]
where $x_{\mbf{c}}\in \Uq^{-}(\mfr{g})$ and if $x_{\mbf{c}}\neq 0$, then $c_{k'}=0$ for $k'\geq k$.
\end{prop}

We have the compatibility of the twisted coproduct $r$ with $\Uq^{-}(w, e)$ (cf.\ \cite[2.4.2 Theorem c)]{LevSoi:qWeyl}).
\begin{NB}The following is a consequence of \cite[Theorem 4 (3)]{Damiani} and also stated in \cite[2.4.2 (c)]{LevSoi:qWeyl}.\end{NB}

\begin{prop}\label{prop:comodule}
We have 
\begin{align*}
r(\Uq^{-}(w, +1)_{\mca{A}}^{\up})&\subset \Uq^{-}(\mfr{g})_{\mca{A}}^{\up}\otimes \Uq^{-}(w, +1)_{\mca{A}}^{\up}, \\
r(\Uq^{-}(w, -1)_{\mca{A}}^{\up})&\subset \Uq^{-}(w, -1)_{\mca{A}}^{\up}\otimes \Uq^{-}(\mfr{g})_{\mca{A}}^{\up},
\end{align*}
that is $\Uq^{-}(w, +1)_{\mca{A}}^{\up}$ (resp.\ $\Uq^{-}(w, -1)_{\mca{A}}^{\up}$) is a left (resp.\ right) $\Uq^{-}(\mfr{g})_{\mca{A}}^{\up}$-comodule.
\end{prop}
\begin{proof}
Recall that we proved $\Uq^{-}(w, e)_{\mca{A}}^{\up}=\Uq^{-}(w, e)\cap \Uq^{-}(\mfr{g})_{\mca{A}}^{\up}$ during the proof of \thmref{theo:unipgp}.
Since $r$ preserves the dual $\mca{A}$-form $\Uq^{-}(\mfr{g})_{\mca{A}}^{\up}$, it suffices to prove a weaker statement,
that is 
\begin{align*}
r(\Uq^{-}(w, +1))&\subset \Uq^{-}(\mfr{g})\otimes \Uq^{-}(w, +1), \\
r(\Uq^{-}(w, -1))&\subset \Uq^{-}(w, -1)\otimes \Uq^{-}(\mfr{g}).
\end{align*}
Moreover if we apply the $*$-involution, we obtain the claim for the $e=-1$ case from the claim for the $e=1$ case.
So it is enough to prove the $e=1$ case.
This assertion is a consequence of \propref{prop:DKP} and \propref{prop:coprod_root}.
\end{proof}
\subsubsection{}
\begin{theo}\label{theo:spec}
Under the algebra homomorphism $\Phi^{\up}$, we have 
\[\mbb{C}\otimes_{\mca{A}}\Uq^-(w, e)_{\mca{A}}^{\up}\simeq \mbb{C}[N(w)].\]
\end{theo}
In view of this theorem, the quantum nilpotent subalgebra $\Uq^{-}(w, e)$ can be considered as
the ``quantum coordinate ring'' of the corresponding unipotent subgroup $N(w)$,
so we call it the \emph{quantum unipotent subgroup} and denote it by $\mca{O}_{q}[N(w)]$.
\begin{NB}In the compatibility with the quantum closed unipotent cell, we give a precise definition of $\mca{O}_{q}[N(w)]$ in \secref{subsec:nilpDemazure}.\end{NB}
\begin{proof}
We compute the following inner product: 
\[\bracket{\Phi^{\up}(F^{\up}_{e}(\mbf{c}, \tw)|_{q=1}), f_{e}((\mbf{c}', \mbf{d}'), \tw)}.\]
First we have
\begin{align*}
&\;\bracket{\Phi^{\up}(F^{\up}_{e}(\mbf{c}, \tw)|_{q=1}), f_{e}((\mbf{c}', \mbf{d}'), \tw)}\\
=&\;\begin{cases}
\bracket{\mu^{*}(\Phi^{\up}(F^{\up}_{e}(\mbf{c}, \tw)|_{q=1}),  f_{e}((\mbf{c}', 0), \tw)\otimes f_{e}((0, \mbf{d}'), \tw)} & \text{~when~} e=1 \\
\bracket{\mu^{*}(\Phi^{\up}(F^{\up}_{e}(\mbf{c}, \tw)|_{q=1}),  f_{e}((0, \mbf{d}'), \tw)\otimes f_{e}((\mbf{c}', 0), \tw)} & \text{~when~} e=-1 \\
\end{cases}\\
=&\;\begin{cases}%
\bracket{(\Phi^{\up}\otimes \Phi^{\up})(r(F^{\up}_{e}(\mbf{c}, \tw)|_{q=1})),  f_{e}((\mbf{c}', 0), \tw)\otimes f_{e}((0, \mbf{d}'), \tw)} & \text{~when~} e=1 \\
\bracket{(\Phi^{\up}\otimes \Phi^{\up})(r(F^{\up}_{e}(\mbf{c}, \tw)|_{q=1})),  f_{e}((0, \mbf{d}'), \tw)\otimes f_{e}((\mbf{c}', 0), \tw)} & \text{~when~} e=-1 \\
\end{cases}\\
=&\;0
\end{align*}
if $\mbf{d}'\neq 0$.
This follows from \lemref{lem:rootvector} and \propref{prop:comodule}.
Hence it suffices to compute the following form
\[\bracket{\Phi^{\up}(F^{\up}_{e}(\mbf{c}, \tw)|_{q=1}), f_{e}(\mbf{c}', \tw)}.\]
This is equal to $\bracket{F^{\up}_{e}(\mbf{c}, \tw)|_{q=1}, \Phi(f_{e}(\mbf{c}', \tw))}=\bracket{F^{\up}(\mbf{c}, \tw)|_{q=1}, F_{e}(\mbf{c}', \tw)|_{q=1}}=\delta_{\mbf{c}, \mbf{c}'}$.
Hence we have $\Phi^{\up}(F_{e}^{\up}(\mbf{c}, \tw)|_{q=1})=f_{e}^{*}((\mbf{c}, 0), \tw)$ and the assertion.
\end{proof}

\section{Quantum closed unipotent cell and the dual canonical basis}\label{sec:Demazure}
\subsection{Demazure-Schubert filtration $\mbf{U}_w^-$}\label{subsec:U_w}
We recall the definition of the Demazure-Schubert filtration $\U_w^-$ associated with a Weyl group element $w\in W$.
\subsubsection{}
\begin{NB}
{1st version}
We recall about the result on the Demazure-Schubert filtration.
First we recall the result of the compatibilities with the canonical basis.
Let $\mbf{i}=(i_1, i_2, \cdots, i_l)\in I^l$ be any sequence in $I$.
$\Uq^-(\mfr{g})_{\mbf{i}}$ be the subspace which is spanned by the monomials 
$F_{i_1}^{(a_1)}\cdots F_{i_l}^{(a_l)}$ for various $a_1, \cdots, a_l\in \mbb{Z}_{\geq 0}$.
Then the following is the compatibility with the canonical basis.
\begin{prop}\cite[4.2]{Lus:problem}
The subspace $\Uq^-(\mfr{g})_{\mbf{i}}$ is a subcoalgebra of $\Uq^-(\mfr{g})$
and compatible with the basis $\mbf{B}$,
that is there exists a subset $\mscr{B}_{\mbf{i}}(\infty)$ of $\mscr{B}(\infty)$ such that:
\[\Uq^-(\mfr{g})_{\mbf{i}}=\bigoplus_{b\in \mscr{B}_{\mbf{i}}(\infty)}\Qq \Glow(b).\]
\end{prop}
\end{NB}
Let $\mbf{i}=(i_1, \cdots, i_l)$ be a sequence in $I$ and $\mbf{U}_{\mbf{i}}^-$
the $\Qq$-linear subspace spanned by the monomials $F_{i_1}^{(a_1)}\cdots F_{i_1}^{(a_l)}$ for
all $(a_1, a_2, \cdots, a_l)\in \mbb{Z}_{\geq 0}^l$,
that is 
\[\mbf{U}_{\mbf{i}}^-:=\sum_{a_1, a_2, \cdots, a_l\in \mbb{Z}_{\geq 0}}\Qq F_{i_1}^{(a_1)}\cdots F_{i_1}^{(a_l)}.\]
By its definition, this is a $\Qq$-subcoalgebra of $\Uq^-$.
We have the following compatibility with the canonical base.
\begin{prop}[{\cite[4.2]{Lus:problem}}]
The subcoalgebra $\mbf{U}_{\mbf{i}}^-$ is compatible with the canonical basis $\mbf{B}$,
that is there exists a subset $\mscr{B}_{\mbf{i}}(\infty)$ of $\mscr{B}(\infty)$ such that
\[\mbf{U}_{\mbf{i}}^-=\bigoplus_{b\in \mscr{B}_{\mbf{i}}(\infty)}\Qq \Glow(b).\]
\end{prop}

\begin{rem}
If we consider the $\mca{A}$-subspace $(\mbf{U}_{\mbf{i}}^-)_{\mca{A}}$ spanned by the monomials $F_{i_1}^{(a_1)}\cdots F_{i_l}^{(a_l)}$,
then $(\mbf{U}_{\mbf{i}}^-)_{\mca{A}}$ is a $\mca{A}$-subcoalgebra of $\Uq^-$ and we have 
\[(\mbf{U}_{\mbf{i}}^-)_{\mca{A}}=\bigoplus_{b\in \mscr{B}_{\mbf{i}}(\infty)}\mca{A} \Glow(b).\]
\end{rem}
\begin{NB}
\subsubsection{Geometric meaning of $\U_{\mbf{i}}^-$}
By Lusztig's construction, the subspace $\Uq^-(\mfr{g})_{\mbf{i}}$ has the following geometric description.
First we fix any sequence $(i_1, i_2, \cdots, i_l)\in I^l$.
We consider the projective morphism 
\[\pi_{\mbf{i}, \mbf{a}}^{\Omega} \colon \widetilde{\Fl}_{\mbf{i}, \mbf{a}}^{\Omega}\to \mbf{E}_{\Omega, \mbf{V}}\]
for various $\mbf{a}=(a_1, \cdots, a_l)\in \mbb{Z}_{\geq 0}$ with $\dimv \mbf{V}=\sum_{j}a_ji_j$.
Since $\pi_{\mbf{i}, \mbf{a}}^{\Omega}$ is proper and $\widetilde{\Fl}_{\mbf{i}, \mbf{a}}^{\Omega}$ is smooth,
then 
\[L^{\Omega}_{\mbf{i}, \mbf{a}}:=\mbb{R}(\pi_{\mbf{i}, \mbf{a}}^{\Omega})_!(\mbf{1}_{\widetilde{\Fl}_{\mbf{i}, \mbf{a}}^{\Omega}}[\dim \widetilde{\Fl}_{\mbf{i}, \mbf{a}}^{\Omega}])\]
is a semisimple complex on $\mbf{E}_{\Omega, \mbf{V}}$ by the decomposition theorem.
We denote by $\mscr{P}^{\Omega}_{\mbf{i}, \mbf{V}}$
the set of isomorphism classes of simple perverse sheaves which appears in the direct summands of $L^{\Omega}_{\mbf{i}, \mbf{a}}$
for various $\mbf{a}$. 
Let $\mscr{Q}^{\Omega}_{\mbf{i}, \mbf{V}}$ be the full subcategory whose objects are the semisimple complexes 
which is isomorphic to finite direct sums of objects $\mscr{P}^{\Omega}_{\mbf{i}, \mbf{V}}$ with some shifts.

Let $\mca{K}(\mscr{Q}^{\Omega}_{\mbf{i}, \mbf{V}})$ be the abelian group with generators $(L)$ for
each isomorphism class of object $L$ in $\widetilde{\Fl}_{\mbf{i}, \mbf{a}}^{\Omega}$.
We define relations $(L')+(L'')=(L)$ whenever $L\simeq L'\+L''$.
We can define $\mca{A}=\mbb{Z}[q^\pm]$-action by $q^\pm(L)=(L[\pm 1])$.
By its construction, $\mca{K}(\mscr{Q}^{\Omega}_{\mbf{i}, \mbf{V}})$ is a free $\mca{A}$-module with the basis given by $\mscr{P}^{\Omega}_{\mbf{i}, \mbf{V}}$.
\end{NB}
\begin{rem}
By the construction of $\mbf{U}^-_{\mbf{i}}$, it is clear that
\begin{align*}
*(\mbf{U}^-_{\mbf{i}})&=\mbf{U}^-_{\mbf{i}^{\operatorname{opp}}}, \\
*(\mscr{B}_{\mbf{i}}(\infty))&=\mscr{B}_{\mbf{i}^{\operatorname{opp}}}(\infty),
\end{align*}
where $\mbf{i}^{opp}=(i_l, i_{l-1}, \cdots, i_1)$ for $\mbf{i}=(i_1, i_2, \cdots, i_l)$.
\end{rem}
\subsubsection{}
For $w\in W$, we consider $\U^-_{\tw}$ associated with $\tw=(i_{1}, \cdots, i_{l})\in R(w)$.
Then it is known that $\U_{\tw}^-$ does not depend on the choice of the reduced expression $\tw$ 
(\cite[5.3]{Lus:problem}).
Therefore we denote $\U_{\tw}^-$ by $\U_w^-$ and  also $\mscr{B}_{\mbf{i}}(\infty)$ by $\mscr{B}_w(\infty)$
by abuse of notations.
By their constructions, we have 
\begin{subequations}
\begin{align}
*(\U_{w}^-)&=\U_{w^{-1}}^-,\\
*(\mscr{B}_w(\infty))&=\mscr{B}_{w^{-1}}(\infty).
\end{align}
\end{subequations}
\begin{NB2}
In the case $\mbf{i}$ is a reduced expression for some Weyl group element $w\in W$,
then it is known \cite[5.3]{Lus:problem} that the subspace $\Uq^-(\mfr{g})_{\mbf{i}}$ does not depend 
on a choice of reduced expression.
So we denote it by $\Uq^-(\mfr{g})_{w}$.
\end{NB2}
\subsubsection{}
Following \cite[9.3]{BZ:qcluster}, we define the \emph{quantum closed unipotent cell} $\mca{O}_q[\overline{N_{w}}]$ 
associated with $w$ by
\begin{NB}
This definition does not work. (10/10/15)
\[\mca{O}_q[\overline{N_{w}}]:=\bigoplus_{b\in \mscr{B}_{w}(\infty)}\Qq \Gup(b).\]
\end{NB}
\[\mca{O}_q[\overline{N_{w}}]:=\Uq^{-}(\mfr{g})/(\mbf{U}_{w}^{-})^{\bot}=\Uq^{-}(\mfr{g})/\bigoplus_{b\notin \mscr{B}_{w}(\infty)}\Qq \Gup(b).\]
Let $\iota^{*}_{w}\colon \mbf{U}_{q}^{-}(\mfr{g})\to \mca{O}_q[\overline{N_{w}}]$ be the natural projection.
Since $(\mbf{U}_{w}^{-})^{\bot}=\bigoplus_{b\notin \mscr{B}_{w}(\infty)}\Qq \Gup(b)$ is compatible with $\mbf{B}^{\up}$,
the natural projection induces an bijection 
$\{\Gup(b); b\in \mscr{B}_{w}(\infty)\}\simeq \{\iota^{*}_{w}(\Gup(b)); b \in \mscr{B}_{w}(\infty)\}$.
Moreover, $(\mbf{U}_{w}^{-})^{\bot}$ is a two-sided ideal since $\mbf{U}_{w}^{-}$ is a subcoalgebra.
Thus $\mca{O}_q[\overline{N_{w}}]$ has an induced algebra structure.
\subsection{Demazure module and its crystal}
In this subsection, we recall the definition of the extremal vector $u_{w\lambda}$ and the associated Demazure module $V_w(\lambda)$.
In particular, we remind that $\mscr{B}_w(\infty)$ can be considered as a certain limit of 
the Demazure crystal.
\subsubsection{}
For $i\in I$, we consider the subalgebra $\Uq(\mfr{g})_i$ generated by $e_i, f_i, t_i$.
Consider the $(l+1)$-dimensional irreducible representation of $\Uq(\mfr{g})_i$
with a highest weight vector $u_{0}^{(l)}$, let $u_{k}^{(l)}:=f_i^{(k)}u_0^{(l)}~(1\leq k\leq l)$.
We have 
\begin{equation}S_i(u_k^{(l)})=(-1)^{l-k}q_i^{(l-k)(k+1)}u_{l-k}^{(l)}.\end{equation}
In particular, we have 
\begin{subequations}
\begin{align}
	S_i(u_l^{(l)})&=u_0^{(l)}, \label{eq:refl}\\
	S_{i}(u_0^{(l)})&=(-q_i)^{l}u_l^{(l)}.
\end{align}
\end{subequations}
\subsubsection{}
For $\lambda\in P$ and $w\in W$, let us denote by $u_{w\lambda}$
the canonical basis element of weight $w\lambda$.
We have the following description (\cite[3.2]{Kas:Demazure} and \cite[Lemma 39.1.2]{Lus:intro}):
\begin{align*}
	u_{w\lambda}&=u_\lambda && \text{~if~} w=1, \\
	u_{s_iw\lambda}&=f_i^{(m)}u_{w\lambda}=S_i^{-1}u_{w\lambda} && \text{~if~} m=\bracket{h_i, w\lambda}\geq 0.
\end{align*}%
\begin{NB}Since we have $\mscr{B}^{\low}(\lambda)_{W\lambda}=\mscr{B}^{\up}(\lambda)_{W\lambda}$, \end{NB}
Recall that $u_{w\lambda}$ is also the dual canonical basis element.
For $\tw\in R(w)$, we have 
\begin{equation}u_{w\lambda}=S_{i_1}^{-1}\cdots S_{i_l}^{-1}u_{\lambda}.\label{eq:extreme}\end{equation}%
\begin{NB2}%
For $(i_1, i_2, \cdots, i_l)\in R$, we set 
\[u_{w\lambda}:=S_{i_1}^{-1}\cdots S_{i_l}^{-1}u_{\lambda},\]
where $u_\lambda$ is the highest weight vector.
Then it is known that $u_{w\lambda}$ does not depend on a choice of reduced expression
and both lower and upper global basis element.
\end{NB2}%
\subsubsection{}\label{sec:Demazure_module}
We recall basic properties of the Demazure module,
see \cite[\S 3]{Kas:Demazure} or \cite[Chapitre 9]{Kas:bases} for more details. 
Let $\lambda\in P_+$ and $V(\lambda)$ be the integrable highest weight $\Uq(\mfr{g})$-module with a highest weight vector $u_\lambda$
of weight $\lambda$.
Let $V_w(\lambda):=\mbf{U}_q^+(\mfr{g})u_{w\lambda}$.
This $\Uq^+(\mfr{g})$-module is called the \emph{Demazure module} associated with $w$ and $\lambda$.
We have the following properties of the Demazure module $V_w(\lambda)$.
\begin{prop}
Let $w\in W$ and $\tw=(i_1, \cdots, i_l)\in R(w)$ be a reduced expression of $w$.

\textup{(1)}
We have
\[V_w(\lambda)=\sum_{a_1, \cdots, a_l\in \mbb{Z}_{\geq 0}}\Qq F_{i_1}^{(a_1)}\cdots F_{i_l}^{(a_l)}u_\lambda.\]

\textup{(2)}
We define $\mscr{B}_w(\lambda)\subset \mscr{B}(\lambda)$ by 
\begin{subequations}
\begin{align}
\mscr{B}_w(\lambda):=&\{\fit{i_1}^{a_1}\cdots \fit{i_r}^{a_l}u_\lambda\in \mscr{B}(\lambda); (a_1, \cdots, a_l)\in \mbb{Z}_{\geq 0}^{l}\setminus \{0\}\} \\
=&\{b\in \mscr{B}(\lambda); \eit{i_l}^{\max}\cdots \eit{i_1}^{\max}b=u_\lambda\}.
\end{align}
\end{subequations}
Then we have 
\[V_{w}(\lambda)=\bigoplus_{b\in \mscr{B}_w(\lambda)}\Qq\Glow_\lambda(b).\]

\textup{(3)}
For $i\in I$, we have 
\[\eit{i}\mscr{B}_w(\lambda)\subset \mscr{B}_w(\lambda)\sqcup \{0\}.\]
\end{prop}
We call $\mscr{B}_w(\lambda)$ the \emph{Demazure crystal}.

\subsubsection{}
We have a similar description of $\mscr{B}_{w}(\infty)$ as $\mscr{B}_w(\lambda)$.
Thus $\mscr{B}_w(\infty)$ can be interpreted as certain limit of the Demazure crystal.
%
\begin{prop}[{\cite[Corollary 3.2.2]{Kas:Demazure}}]\label{lem:Demazure}
Let $w\in W$ and $(i_1, \cdots, i_l)\in R(w)$ be its reduced expression.

\textup{(1)}
We have 
\begin{subequations}
\begin{align}
\mscr{B}_w(\infty)=&\{\fit{i_1}^{a_1}\cdots \fit{i_r}^{a_l}u_\infty\in \mscr{B}(\infty); (a_1, \cdots, a_l)\in \mbb{Z}_{\geq 0}^{l}\setminus \{0\}\} \\
=&\{b\in \mscr{B}(\infty); \eit{i_l}^{\max}\cdots \eit{i_1}^{\max}b=u_\infty\}.
\end{align}
\end{subequations}

\textup{(2)}
For $i\in I$, we have
\begin{equation}\eit{i}\mscr{B}_w(\infty)\subset \mscr{B}_w(\infty)\sqcup \{0\}. \end{equation}
\end{prop}
\subsection{}\label{subsec:nilpDemazure}
To study  multiplicative properties of $\Uq^-(w, e)$, we relate it to the quantum closed unipotent cell $\mca{O}_q[\overline{N_w}]$.
The following is a generalization of \cite[3.2 Lemma]{Cal:qflag}.
This can be considered as a quantum analogue of \cite[Corollary 15.7]{GLS:KacMoody}.
\begin{theo}\label{theo:emb}
For $w\in W$ and $e\in \{\pm 1\}$, we have the following embedding of algebras:
\[\Uq^-(w, e)\hookrightarrow \mca{O}_q[\overline{N_{w^{-e}}}].\]
\end{theo}
\begin{proof}
\begin{NB}Since both algebras $\Uq^-(w, e)$ and $\mca{O}_q[\overline{N_{w^{-e}}}]$ are compatible with the dual canonical base,\end{NB}
We consider the composite of the inclusion $\Uq^-(w, e)\hookrightarrow \Uq^{-}(\mfr{g})$ and the natural projection $\iota^{*}_{w^{-e}}\colon \Uq^{-}(\mfr{g})\to \mca{O}_{q}[\overline{N_{w^{-e}}}]$.
Since both homomorphisms are algebra homomorphisms, we obtain an algebra homomorphism
\[\Uq^-(w, e)\to\mca{O}_q[\overline{N_{w^{-e}}}].\]
Since $\Uq^{-}(w, e)$ is compatible with $\mbf{B}^{\up}$ and $\iota^{*}_{w^{-e}}$ induces an bijection
$\{\Gup(b); b\in \mscr{B}_{w^{-e}}(\infty)\}\simeq \{\iota^{*}_{w^{-e}}(\Gup(b)); b\in \mscr{B}_{w^{-e}}(\infty)\}$,
it suffices to prove the corresponding assertion for
the crystals, that is $\mscr{B}(w, e)\hookrightarrow \mscr{B}_{w^{-e}}(\infty)$.
Since we have $*(\mscr{B}(w, e))=\mscr{B}(w, -e)$ and $*(\mscr{B}_w(\infty))=\mscr{B}_{w^{-1}}(\infty)$,
it is enough to prove the claim for the $e=1$ case.

We prove $\mscr{B}(w, 1)\subset \mscr{B}_{w^{-1}}(\infty)$ by the induction on $l=\ell(w)$.
For $l=1$ case, by the constructions of $\mscr{B}(s_i, e)$ and $\mscr{B}_{s_i}(\infty)$,
we have $\mscr{B}(s_i, e)=\mscr{B}_{s_i}(\infty)$ for any $i\in I$ and $e\in \{\pm 1\}$.
Let $\tw=(i_{1}, \cdots, i_{l})\in R(w)$ be a reduced expression.
For $l\geq 2$, we can assume that,
for $w_{\geq 2}:=s_{i_2}\cdots s_{i_l}\in W$, $\tw_{\geq 2}=(i_2, i_3, \cdots, i_l)\in R(w_{\geq 2})$ and 
$\mbf{c}_{\geq 2}=(c_2, c_3, \cdots, c_l)\in \mbb{Z}_{\geq 0}^{l-1}$,
we have 
\[b_{\geq 2}:=F(\mbf{c}_{\geq 2}, \tw_{\geq 2})\mod q\mscr{L}(\infty)\in \mscr{B}_{w_{\geq 2}^{-1}}(\infty)\]
by the induction hypothesis.
Note that $\eit{i_1}^{\max}b_{\geq 2}\in \mscr{B}_{w_{\geq 2}^{-1}}(\infty)$ by \lemref{lem:Demazure} (2).
Since $*(\mscr{B}_{w}(\infty))=\mscr{B}_{w^{-1}}(\infty)$, we have $\fit{i_{1}}^{*\vphi_{i_{1}}(b_{\geq 2})}\eit{i_1}^{\max}b_{\geq 2}\in \mscr{B}_{w^{-1}}(\infty)$.
In view of \cite[Theorem 3.3.2]{Kas:Demazure}, it suffices to prove $\fit{i_{1}}(\fit{i_{1}}^{*\vphi_{i_{1}}(b_{\geq 2})}\eit{i_1}^{\max}b_{\geq 2})\in \mscr{B}_{w^{-1}}(\infty)$.
We consider the image of it under the Kashiwara embedding $\Psi_{i_{1}}\colon \mscr{B}(\infty)\to \mscr{B}(\infty)\otimes \mscr{B}_{i_{1}}$
and show that the image of it is contained in $\mscr{B}_{w_{\geq 2}^{-1}}(\infty)\otimes \mscr{B}_{i_{1}}$.
Since $\Psi_{i_{1}}$ is a strict embedding, we have 
\begin{align*}
\Psi_{i_{1}}(\fit{i_{1}}(\fit{i_{1}}^{*\vphi_{i_{1}}(b_{\geq 2})}\eit{i_1}^{\max}b_{\geq 2}))
=&\fit{i_{1}}(\Psi_{i_{1}}(\fit{i_{1}}^{*\vphi_{i_{1}}(b_{\geq 2})}\eit{i_1}^{\max}b_{\geq 2}))\\
=&\fit{i_{1}}(\eit{i_1}^{\max}b_{\geq 2}\otimes \fit{i_{1}}^{\vphi_{i_{1}}(b_{\geq 2})}b_{i_{1}}).
\end{align*}
If $\vphi_{i_{1}}(\eit{i_{1}}^{\max}b_{\geq 2})\leq \vep_{i_{1}}(\fit{i_{1}}^{\vphi_{i_{1}}(b_{\geq 2})})$,
we have $\fit{i_{1}}(\eit{i_1}^{\max}b_{\geq 2}\otimes \fit{i_{1}}^{\vphi_{i_{1}}(b_{\geq 2})}b_{i_{1}})=\eit{i_1}^{\max}b_{\geq 2}\otimes \fit{i_{1}}^{\vphi_{i_{1}+1}(b_{\geq 2})}b_{i_{1}}$.
This is contained in $\mscr{B}_{w_{\geq 2}^{-1}}(\infty)\otimes \mscr{B}_{i_{1}}$.

Suppose $\vphi_{i_{1}}(\eit{i_{1}}^{\max}b_{\geq 2})> \vep_{i_{1}}(\fit{i_{1}}^{\vphi_{i_{1}}(b_{\geq 2})})$.
This means that $\vep_{i_{1}}(b_{\geq 2})>0$.
Let $S$ be the $i_{1}$-string which contains $b_{\geq 2}$ and $\eit{i_{1}}^{\max}(b_{\geq 2})$.
Note that $b_{\geq 2}\neq \eit{i_{1}}^{\max}(b_{\geq 2})$.
Since both $b\geq 2$ and $\eit{i_{1}}^{\max}b_{\geq 2}$ are in $\mscr{B}_{w_{\geq 2}^{-1}}(\infty)$,
we have $S\cap \mscr{B}_{w_{\geq 2}^{-1}}(\infty)=S$ by \cite[Proposition 3.3.4]{Kas:Demazure}.
Hence we get the assertion.
\begin{NB}
The following proof does not work (10/10/15).
We follow the argument in \cite[4.2 Example]{Saito:PBW}.
For a reduced expression $(i_1, i_2, \cdots, i_l)\in R(w)$,  we note that if we have 
$\Psi_{(i_1, i_2, \cdots, i_r)}(b)=u_{\infty}\otimes \fit{i_r}^{m_r}b_{i_r}\otimes \cdots \otimes \fit{i_{1}}^{m_1}b_{i_1}$
for some $m_1, m_2, \cdots, m_r$, then we have 
$\eit{i_r}^{*\max}\cdots \eit{i_1}^{*\max}(b)=u_\infty$ and $b\in \mscr{B}_{w^{-1}}(\infty)$
by the description of the Kashiwara embedding.
Conversely the image of $b\in \mscr{B}_{w^{-1}}(\infty)$ under $\Psi_{(i_{1}, \cdots, i_{r})}$ has an above form.
So we show 
$\Psi_{(i_1, i_2, \cdots, i_r)}(b)=u_{\infty}\otimes \fit{i_r}^{m_r}b_{i_r}\otimes \cdots \otimes \fit{i_1}^{m_1}b_{i_1}$
for any $b\in \mscr{B}(w, e=1)$.

We prove by the induction $l=\ell(w)$.
For $l=1$ case, by the constructions of $\mscr{B}(s_i, e)$ and $\mscr{B}_{s_i}(\infty)$,
we have $\mscr{B}(s_i, e)=\mscr{B}_{s_i}(\infty)$ for any $i\in I$ and $e\in \{\pm 1\}$.
For $l\geq 2$, we can assume that,
for $w_{\geq 2}:=s_{i_2}\cdots s_{i_l}\in W$, $\tw_{\geq 2}=(i_2, i_3, \cdots, i_l)\in R(w_{\geq 2})$ and 
$\mbf{c}_{\geq 2}=(c_2, c_3, \cdots, c_l)\in \mbb{Z}_{\geq 0}^{l-1}$,
we have 
\[b(\mbf{c}_{\geq 2}, \tw_{\geq 2})=F(\mbf{c}_{\geq 2}, \tw_{\geq 2})\mod q\mscr{L}(\infty)\in \mscr{B}_{w_{\geq 2}^{-1}}(\infty)\]
by the induction hypothesis.
Note that $\eit{i_1}^{\max}b(\mbf{c}_{\geq 2}, \tw_{\geq 2})\in \mscr{B}_{w_{\geq 2}^{-1}}(\infty)$.
Therefore 
\[\Psi_{(i_{2}, \cdots, i_{r})}(\eit{i_1}^{\max}b(\mbf{c}_{\geq 2}, \tw_{\geq 2}))=u_{\infty}\otimes \fit{i_{l}}^{m'_{l}}b_{i_{l}}\otimes \cdots \otimes \fit{i_{2}}^{m'_{2}}b_{i_{2}}\]
for some $m'_{2}, \cdots, m'_{l}\in \mbb{Z}_{\geq 0}$.
Now we consider
\begin{align*}
b(\mbf{c}, \tw)\;&=F(\mbf{c}, \tw) \mod q\mscr{L}(\infty)=\fit{i_1}^{c_1}\Lambda_{i_1}b(\mbf{c}_{\geq 2}, \tw_{\geq 2})\\
\;&=\fit{i_1}^{c_1}\fit{i_1}^{*\vphi_{i_1}(b(\mbf{c}_{\geq 2}, \tw_{\geq 2}))}\eit{i_1}^{\max}b(\mbf{c}_{\geq 2}, \tw_{\geq 2}).
\end{align*}
We have 
\begin{align*}
\;&\Psi_{i_l}\circ \cdots \circ \Psi_{i_1}(b(\mbf{c}; \tw)) \\
=\;&\fit{i_1}^{c_1}\Psi_{i_l}\circ \cdots \circ \Psi_{i_1}(\fit{i_1}^{*\varphi_{i_1}(b(\mbf{c}_{\geq 2}; \tw_{\geq 2}))}\eit{i_1}^{\max}b(\mbf{c}_{\geq 2}; \tw_{\geq 2})) \\
=\;&\fit{i_1}^{c_1}((\Psi_{i_l}\circ \cdots \circ \Psi_{i_2})(\eit{i_1}^{\max}b(\mbf{c}_{\geq 2}; \tw_{\geq 2}))\otimes \fit{i_1}^{\varphi_{i_1}(b(\mbf{c}_{\geq 2}; \tw_{\geq 2}))}b_{i_1}).
\end{align*}
Hence the above is equal to
\[\fit{i_1}^{c_1}((u_{\infty}\otimes \fit{i_l}^{m'_l}b_{i_l}\otimes \cdots \otimes \fit{i_2}^{m'_2}b_{i_2})\otimes \fit{i_1}^{\varphi_{i_1}(b(\mbf{c}_{\geq 2}; \tw_{\geq 2}))}b_{i_1}).\]
Note that $\vphi_{i_{1}}(u_{\infty})=\vep_{i_{1}}(u_{\infty})+\wt_{i_{1}}(u_{\infty})=0$
and $\vep_{i_{1}}(\fit{i_l}^{m'_l}b_{i_l}\otimes \cdots \otimes \fit{i_2}^{m'_2}b_{i_2}\otimes \fit{i_1}^{\varphi_{i_1}(b(\mbf{c}_{\geq 2}; \tw_{\geq 2}))}b_{i_1})=\vep_{i_{1}}(u_{\infty}\otimes\fit{i_l}^{m'_l}b_{i_l}\otimes \cdots \otimes \fit{i_2}^{m'_2}b_{i_2}\otimes \fit{i_1}^{\varphi_{i_1}(b(\mbf{c}_{\geq 2}; \tw_{\geq 2}))}b_{i_1})=\vep_{i_{1}}(\Lambda_{i_{1}}b(\mbf{c}_{\geq 2}, \tw_{\geq 2}))=0$.
Here we used \eqref{eq:vepmult} and \thmref{theo:Kasemb}.
So the above is equal to 
\[u_{\infty}\otimes \fit{i_1}^{c_1}(\fit{i_l}^{m'_l}b_{i_l}\otimes \cdots \otimes \fit{i_2}^{m'_2}b_{i_2}\otimes \fit{i_1}^{\varphi_{i_1}(b(\mbf{c}_{\geq 2}; \tw_{\geq 2}))}b_{i_1})\]
by \eqref{eq:fitn}.
Thus we obtain the claim for $l$.
\end{NB}
\end{proof}
\begin{NB}
In view of this theorem, 
we call  $\Uq^{-}(w, e=-1)$ the quantum unipotent subgroup $\mca{O}_{q}[N(w)]$ and we set $\mscr{B}(w)$
by $\mscr{B}(w, -1)$.
So we have the following embedding:
\[\mca{O}_{q}[N(w)]\hookrightarrow \mca{O}_{q}[\overline{N}_{w}].\]
This result can be considered as quantum analogue of \cite[Corollary 15.7]{GLS:KacMoody}.\end{NB}

\begin{NB2}To study the multiplicative properties of $\Uq^-(w)$, we relate it with the Demazure-Schubert filtration.
The following is a generalization of \cite[3.2 Lemma]{Cal:qflag}.
\begin{prop}
For $w\in W$ and $e\in \{\pm 1\}$, we have 
\[\mscr{B}(w, e)\subset \mscr{B}_{w^{-e}}(\infty).\]
\end{prop}
\begin{proof}
We prove the case with $e=1$.
For the case $e=-1$, we obtain the claim by applying $*$-involution and uses the properties.
We follow the argument in \cite[4.2 Example]{Saito:PBW}.
To prove the statement, we consider the image of the Kashiwara embedding along $\tw=(i_1, i_2, \cdots, i_l)$.
It suffices for us that $\Psi_{i_l}\circ \cdots \circ \Psi_{i_1}(b(\mbf{c}; \tw))=u_{\infty}\otimes \fit{i_l}^{m_l}b_{i_l}\otimes \cdots \fit{i_1}^{m_1}b_{i_1}$
for some non-negative integers $m_1, \cdots, m_l$.
We prove by the induction on $l=l(w)$.
For $l=1$ case, we have $\mscr{B}(s_{i}, e)=\mscr{B}_{s_i}$ by the constructions.
We take $\tw=(i_1, i_2, \cdots, i_l)\in R(w)$.
By the induction hypothesis, we have that the statement for $w_{\geq 2}:=s_{i_2}\cdots s_{i_l}$ and $\tw_{\geq 2}=(i_2, \cdots, i_l)\in R(w_{\geq 2})$.
In particular, we have 
$b(\mbf{c}_{\geq 2}; \tw_{\geq 2}):=F_{i_2}^{(c_2)}\cdots T_{i_2}\cdots T_{i_l}(F_{i_l}^{(c_l)}) \mod q\mscr{L}(\infty)\in \mscr{B}_{w_{\geq 2}^{-1}}(\infty)$.
We consider an element
\[b(\mbf{c}; \tw)=\fit{i_1}^{c_1}\fit{i_1}^{*\varphi_{i_1}(b(\mbf{c}_{\geq 2}; \tw_{\geq 2}))}\eit{i_1}^{\max}b(\mbf{c}_{\geq 2}; \tw_{\geq 2}).\]
Here we have $0\neq \eit{i_1}^{\max}b(\mbf{c}_{\geq 2}; \tw_{\geq 2})\in \mscr{B}_{w_{\geq 2}^{-1}}(\infty)$
by \cite[Proposition 3.2.5 (iii)]{Kas:Demazure}.
\begin{align*}
&\Psi_{i_l}\circ \cdots \circ \Psi_{i_1}(b(\mbf{c}; \tw)) \\
=&\fit{i_1}^{c_1}\Psi_{i_l}\circ \cdots \circ \Psi_{i_1}(\fit{i_1}^{*\varphi_{i_1}(b(\mbf{c}_{\geq 2}; \tw_{\geq 2}))}\eit{i_1}^{\max}b(\mbf{c}_{\geq 2}; \tw_{\geq 2})) \\
=&\fit{i_1}^{c_1}((\Psi_{i_l}\circ \cdots \circ \Psi_{i_2})(\eit{i_1}^{\max}b(\mbf{c}_{\geq 2}; \tw_{\geq 2}))\otimes \fit{i_1}^{\varphi_{i_1}(b(\mbf{c}_{\geq 2}; \tw_{\geq 2}))}b_{i_1}) \\
=&\fit{i_1}^{c_1}((u_{\infty}\otimes \fit{i_l}^{m'_l}b_{i_l}\otimes \cdots \otimes \fit{i_2}^{m'_2}b_{i_2})\otimes \fit{i_1}^{\varphi_{i_1}(b(\mbf{c}_{\geq 2}; \tw_{\geq 2}))}b_{i_1})\\
=&u_{\infty}\otimes \fit{i_1}^{c_1}(\fit{i_l}^{m'_l}b_{i_l}\otimes \cdots \otimes \fit{i_2}^{m'_2}b_{i_2}\otimes \fit{i_1}^{\varphi_{i_1}(b(\mbf{c}_{\geq 2}; \tw_{\geq 2}))}b_{i_1})
\end{align*}
Hence the claim holds.
\end{proof}
\end{NB2}

\begin{NB}
\subsection{Feigin map as quantization of Gei\ss-Leclerc-Schr\"{o}er's $\varphi$-map}
In this subsection, we introduce Feigin map $\chi_{\tw}$ following \cite{Ber:Fei} and \cite{Rei:Fei}.
This is $q$-commutative ``specialzation'' of $q$-character or $q$-shuffle algebra 
\subsubsection{}
Let $w\in W$ and $\tw=(i_1, i_2, \cdots, i_l)\in R(w)$.
We define $\chi_{\tw}\colon \Uq^-\to \Qq[Z_1, Z_2, \cdots, Z_n]$ as follows.
Let $r\colon \Uq^-\to \Uq^-\otimes \Uq^-$ be the twisted coproduct.
Let $r^{\otimes l}:=(r\otimes 1^{\otimes l-2})(r\otimes 1^{\otimes l-2})\cdots r\colon \Uq^-\to (\Uq^-)^{\otimes l}$.
If we define an $\Qq$-algebra structure by 
\[(x_1\otimes x_2 \otimes \cdots \otimes x_l)(y_1\otimes y_2 \otimes \cdots \otimes y_l):=q^{-\sum_{k<l}(\wt(x_l),\wt( y_l)))} (x_1y_1\otimes x_2y_2 \otimes \cdots \otimes x_ly_l),\]
then $r^{\otimes l}$ is $\Qq$-algebra homomorphism.
\subsubsection{}
\[\mca{J}_{w}:=\{u\in \Uq^-(\mfr{g}) ; (u, \mbf{U}^-_{q, w})=0 \}=\bigoplus_{b\in B(\infty)\setminus B_w(\infty)}\Qq\Gup(b)\]
Since we have $r(\mbf{U}^-_{q, w})\subset \mbf{U}^-_{q, w}\otimes \mbf{U}^-_{q, w}$.
Then $\mca{J}_{w}$ is a two-sided ideal in $\Uq^-(\mfr{g})$.
\end{NB}

\section{Construction of initial seed: Quantum flag minors}\label{sec:qflag}
In this section, we give a construction of the quantum initial seed in Conjecture \ref{conj:qconj}
which corresponds to the initial seed in \cite{GLS:KacMoody}.
We only consider the $e=-1$ case, but the other case follows by applying the $\ast$-involution.

\subsection{Quantum generalized minors}
\subsubsection{}
We first define a quantum generalized minor.
This is a $q$-analogue of a (restricted) generalized minor 
$D_{w\lambda}=D_{w\lambda, \lambda}$ which is defined in \cite[7.1]{GLS:KacMoody}.

\begin{defn}[quantum generalized minor]\label{def:qminor}
For $\lambda\in P_+$ and $w\in W$, let
\[\Delta_{w\lambda}=\Delta_{w\lambda, \lambda}:=\jbar_\lambda(u_{w\lambda}).\]
We call it a \emph{quantum generalized minor}.
When $\lambda$ is a fundamental weight, we call it a \emph{quantum flag minor}.
\end{defn}
By its definition, it is given by a matrix coefficient as
\[(\Delta_{w\lambda, \lambda}, P)=(u_{w\lambda}, Pu_\lambda).\]
\subsubsection{}
The following result for extremal vectors is well-known. 
\begin{lem}[{\cite[Lemma 8.6]{Nak:CBMS}}]\label{lem:extremal}
For $\lambda, \mu\in P_+$ and $w\in W$,
we have 
\[\Phi(\lambda, \mu)(u_{w(\lambda+\mu)})=u_{w\lambda}\otimes u_{w\mu}.\]
It follows that  
\[q_{\lambda, \mu}(u_{w\lambda}\otimes u_{w\mu})=u_{w\lambda+w\mu}.\]
Therefore we get 
\[q^{(w\mu-\mu, \lambda)}\Delta_{w\lambda}\Delta_{w\mu}=\Delta_{w(\lambda+\mu)}\]
by \propref{prop:prod_rep}.
In particular, $\Delta_{w, \lambda}$ is strongly real for any $w\in W$ and $\lambda\in P_{+}$.
\end{lem}
\subsubsection{}
We describe extremal vectors in terms of the PBW basis.
This is a straight forward generalization of \cite{Cal:unity}.
For $1\leq k\leq l$, we define the following operations as in \cite[9.8]{GLS:KacMoody},
\begin{align*}
	k^-&:=\max(0, \{1\leq s\leq k-1; i_{s}=i_k\}),\\
	k_{\max}&:=\max\{1\leq s\leq l; i_s=i_k\}. 
\end{align*}
\begin{prop}\label{prop:extremePBW1}
For $0\leq k\leq l$, we define $\mbf{n}_k$ by 
\[\mbf{n}_{k}(j):=\begin{cases}1 & \text{~if~} i_j=i_k,  j\leq k, \\ 0 & \text{~otherwise}.\end{cases}\]
If $i=i_k$ (here we understand  $i=i_k$ holds for any $i$ if $k=0$), we have 
$F_{e=-1}(m\mbf{n}_{k}; \tw)u_{m\varpi_{i}}=u_{s_{i_1}\cdots s_{i_{k}}m\varpi_{i}}$
for $m\geq 1$.
\end{prop}
\begin{proof}
We follow the argument in \cite[2.1 Lemma]{Cal:unity}.
We prove the assertion by an induction on $k$.
The assertion is trivial when $k=0$.
Note that 
\[F_{-1}(m\mbf{n}_k, \tw)=T_{i_1}^{-1}\cdots T_{i_{k-1}}^{-1}(F_{i_k}^{(m)})F_{-1}(m\mbf{n}_{k^-}, \tw).\]
Therefore we have 
\[F_{-1}(m\mbf{n}_k, \tw)u_{m\varpi}=T_{i_1}^{-1}\cdots T_{i_{k-1}}^{-1}(F_{i_k}^{(m)})u_{s_{i_1}\cdots s_{i_{k^-}}m\varpi_i}\]
{by the induction hypothesis}.
By \eqref{eq:braidadj}, this is equal to 
\begin{align*}
&S_{i_1}^{-1}\cdots S_{i_{k-1}}^{-1}(F_{i_k}^{(m)})S_{i_{k-1}}\cdots S_{i_{1}}S_{i_1}^{-1}\cdots S_{i_{k^-}}^{-1}u_{m\varpi_i} \\
=\;&S_{i_1}^{-1}\cdots S_{i_{k-1}}^{-1}(F_{i_k}^{(m)})S_{i_{k-1}}\cdots S_{i_{k^-+1}}u_{m\varpi_i}.
\end{align*}
Since none of $i_{k^-+1}, \cdots, i_{k-1}$ is $i$, this is equal to 
\[S_{i_1}^{-1}\cdots S_{i_{k-1}}^{-1}(F_{i_k}^{(m)})u_{m\varpi_i}.\]
By \eqref{eq:extreme}, this is nothing but $S_{i_1}^{-1}\cdots S_{i_k}^{-1}u_{m\varpi_i}$.
Therefore the assertion also holds for $k$.
\begin{NB}\begin{align*}
& F_{-1}(m\mbf{n}_{k}; \tw)u_{m\varpi_i}=T_{i_1}^{-1}\cdots T_{i_{k-1}}^{-1}(F_{i_k}^{(m)})F_{-1}(m\mbf{n}_{k^-}; \tw)u_{m\varpi_i} \\
=\;& T_{i_1}^{-1}\cdots T_{i_{k-1}}^{-1}(F_{i_k}^{(m)})u_{s_{i_1}\cdots s_{i_{k^-}}(m\varpi_i)} ~(\text{by induction hypothesis})\\
=\;& S_{i_1}^{-1}\cdots S_{i_{k-1}}^{-1}(F_{i_k}^{(m)})S_{i_{k-1}}\cdots S_{i_1}S_{i_1}^{-1}\cdots S_{i_{k^-}}^{-1}u_{m\varpi_{i}}\\
=\;& S_{i_1}^{-1}\cdots S_{i_{k-1}}^{-1}(F_{i_k}^{(m)})S_{i_{k-1}}\cdots S_{i_{k^-+1}}u_{m\varpi_{i}} \\
=\;&S_{i_1}^{-1}\cdots S_{i_{k-1}}^{-1}(F_{i_k}^{(m)})u_{m\varpi_{i}} \\
=\;& S_{i_1}^{-1}\cdots S_{i_{k-1}}^{-1}S_{i_k}^{-1}u_{m\varpi_{i}}=u_{s_{i_1}\cdots s_{i_k}(m\varpi_{i})}
\end{align*}
Then we obtain the claim.
\end{NB}
\end{proof}
By the above proposition, we have 
$\pi_{m\varpi_{i_k}}(b_{-1}(m\mbf{n}_{k}; \tw))\neq 0$ for any $1\leq k\leq l$ and $m\geq 1$.
\begin{NB}
Hence 
$\Gup(\jbar_{\varpi_{i_k}}(u_{s_{i_1}\cdots s_{i_k}\varpi_{i}}))=b_{-1}(\mbf{c}_{k}; \tw)$.
\end{NB}
Hence $\jbar_{m\varpi_{i_k}}(u_{s_{i_1}\cdots s_{i_k}m\varpi_{i_k}})=\Gup(b_{-1}(m\mbf{n}_k, \tw))$
for any $1\leq k\leq l$.
As a special case, we obtain the following result.
\begin{cor}
For $w\in W$ and fix $\tw\in R(w)$.
For $i\in I$, we set $\mbf{n}^i$ by $\mbf{n}_{k_{\max}}$ with $i_k=i$.
For $\lambda\in P_+$, we set $\mbf{n}^\lambda:=\sum_{i\in I}\lambda_i\mbf{n}^i\in \mbb{Z}_{\geq 0}^l$.
Then we have
\begin{equation}\Delta_{w\lambda}=\Gup(b_{-1}(\mbf{n}^\lambda, \tw)).\end{equation}
\end{cor}
\begin{proof}
By \propref{prop:extremePBW1},  we have 
\begin{equation}\Delta_{wm\varpi_i}=\Gup(b_{-1}(m\mbf{n}^i, \tw)) \label{eq:fundweight}\end{equation}
for any $i\in I$.
Then by \eqref{eq:fundweight}, \lemref{lem:extremal} and \corref{cor:dsum}, we obtain the assertion.
\end{proof}
\subsection{Commutativity relations}
In this subsection, we prove that quantum generalized minors $\{\Delta_{w\lambda}\}$ $q$-commute with $\Gup(b)$
for $b\in \mscr{B}_{w}(\infty)$ in the quotient $\mca{O}_{q}[\overline{N_{w}}]$.
It means that $\Delta_{w\lambda}$ and $\Gup(b)$ $q$-commute up to $(\mbf{U}_{w}^{-})^{\bot}$.
 By \thmref{theo:emb}, they literally $q$-commute when $b\in \mscr{B}_{w}(w, -1)$.
We denote the projection of $\Gup(b)$ to $\mca{O}_{q}[\overline{N_{w}}]$ also by $\Gup(b)$ for brevity.
\subsubsection{}\label{sec:quasiR}
For the proof of certain $q$-commutativity relation, we need to use the quasi $\mca{R}$-matrix.
We recall its properties.

First we consider another coproduct $\overline{\Delta}$
defined by $(\overline{\phantom{x}}\otimes \overline{\phantom{x}})\circ \Delta \circ \overline{\phantom{x}}$.
We have an analogue of \lemref{lem:rcoproduct}:
\begin{subequations}
\begin{align}
	\overline{\Delta}(q^h)&=q^h\otimes q^h, \\
	\overline{\Delta}(e_i)&=e_i\otimes t_i +1\otimes e_i, \\
	\overline{\Delta}(f_i)&=f_i \otimes 1+t_i^{-1}f_i.
\end{align}
\end{subequations}
We consider the following completion 
\[\Uq^+(\mfr{g})\widehat{\otimes}\Uq^-(\mfr{g})=\bigoplus_{\xi\in Q}\prod_{\xi=\xi'+\xi''}\Uq^+(\mfr{g})_{\xi'}\otimes \Uq^-(\mfr{g})_{\xi''}.\]
Note that the counit $\vep$ extends to the completion.
In \cite[Chapter 4]{Lus:intro}, Lusztig has shown that there exists a unique intertwiner $\Xi \in \Uq^+(\mfr{g})\widehat{\otimes}\Uq^-(\mfr{g})$ 
such that 
\begin{equation}\Xi \circ \Delta(x)=\overline{\Delta}(x) \circ \Xi \text{~for any~} x\in \Uq(\mfr{g}), \label{eq:Rmatrix}\end{equation}
$\vep(\Xi)=1$, and $\Xi \circ \overline{\Xi}=\overline{\Xi}\circ \Xi =1$. 
We have an analogue of \lemref{lem:rcoproduct}:
\begin{equation}
\overline{\Delta}(x)=\sum q^{-(\wt x_{(1)}, \wt x_{(2)})}x_{(2)}t_{\wt x_{(1)}}\otimes x_{(1)}\label{eq:deltabar_r},\end{equation}
for any $x\in \Uq^-(\mfr{g})$ with $r(x)=\sum x_{(1)}\otimes x_{(2)}$.
In particular, we have 
\begin{equation}
\overline{\Delta}(x)(u_\lambda\otimes u_\mu)=\sum q^{-(\wt x_{(1)}, \wt x_{(2)})}x_{(2)}t_{\wt x_{(1)}}u_\lambda \otimes x_{(1)}u_\mu,
\end{equation}
for such $x\in \Uq^-(\mfr{g})$.
\subsubsection{}\label{sec:q-center}
\begin{prop}\label{prop:q-center}
For $b\in \mscr{B}_w(\mu)$ and $u_{w\lambda}\in \mscr{B}_w(\lambda)$,
we have the following $q$-commutation relation in $\mca{O}_{q}[\overline{N_{w}}]$:
\[\jbar_\lambda(u_{w\lambda})\Gup(\jbar_{\mu}(b))\simeq \Gup(\jbar_\mu(b))\jbar_\lambda(u_{w\lambda}).\]
\end{prop}
\begin{proof}
Since we only consider the equality in quantum closed unipotent cell $\mca{O}_q[\overline{N}_w]$,
it is enough to check that inner products with $x\in \U_w^-$ are the same up to some $q$-shifts, and the $q$-shifts do not depend on choice of $x$.
By \propref{prop:prod_rep}, this is equal to 
\[(u_{w\lambda}\otimes \Gup_\mu(b), \Delta(x)(u_{\lambda}\otimes u_\mu))_{\lambda, \mu},\]
where $(\cdot, \cdot)_{\lambda, \mu}$ denotes the inner product on $V(\lambda)\otimes V(\mu)$
defined by $(u\otimes u', v\otimes v')_{\lambda, \mu}:=(u, v)_\lambda (u', v')_{\mu}$.
We use the quasi $\mca{R}$-matrix to rewrite this as
\[(u_{w\lambda}\otimes \Gup_\mu(b), (\overline{\Xi}\circ \overline{\Delta}(x) \circ \Xi)(u_{\lambda}\otimes u_\mu))_{\lambda, \mu}.\]
Since the action of the quasi $\mca{R}$-matrix is trivial on the highest weight vector, this is equal to 
\[(u_{w\lambda}\otimes \Gup_\mu(b), (\overline{\Xi}\circ \overline{\Delta}(x))(u_{\lambda}\otimes u_\mu))_{\lambda, \mu}.\]
Since the inner product has an adjoint property for $\varphi$, this is equal to 
\[((\varphi\otimes \varphi)(\overline{\Xi})(u_{w\lambda}\otimes \Gup_\mu(b)), (\overline{\Delta}(x))(u_{\lambda}\otimes u_\mu))_{\lambda, \mu}.\]
Note that $(\overline{\Delta}(x))(u_{\lambda}\otimes u_\mu)$ is contained in the tensor product 
of Demazure modules $V_w(\lambda)\otimes V_w(\mu)$ by \secref{sec:Demazure_module}.
By the form of quasi $\mca{R}$-matrix \eqref{eq:Rmatrix} and the definition of $\varphi$, 
the nontrivial part of $(\varphi\otimes \varphi)(\overline{\Xi})(u_{w\lambda}\otimes \Gup_\mu(b))$ is not contained in 
the tensor product $V_w(\lambda)\otimes V_w(\mu)$, therefore the above is equal to 
\[(u_{w\lambda}\otimes \Gup_\mu(b), (\overline{\Delta}(x))(u_{\lambda}\otimes u_\mu))_{\lambda, \mu}.\]
By \eqref{eq:deltabar_r}, this is equal to 
\[(u_{w\lambda}\otimes \Gup_\mu(b), (\flip \circ \Delta(x))(u_{\lambda}\otimes u_\mu))_{\lambda, \mu},\]
and also to
\[(\jbar_{\lambda}(u_{w\lambda})\otimes \Gup(\jbar_\mu b), (\flip \circ r(x))_K\]
up to some $q$-shifts, where $\flip(P\otimes Q):=Q\otimes P$.
Therefore we get 
\[(\jbar_\lambda(u_{w\lambda})\Gup(\jbar_\mu(b)), x)_K\simeq (\Gup(\jbar_\mu(b))\jbar_\lambda(u_{w\lambda}), x)_K\]
for any $x\in \mbf{U}_w^-$.
Here we note that $q$-shifts depend only on the weights of $u_{w\lambda}$ and $\jbar_\mu(b)$,
and is independent of $x$.
Then we obtain the assertion.
\end{proof}
Restricting the above equality, we obtain the following $q$-commutativity relations in $\mca{O}_{q}[N(w)]$.
\begin{cor}
For $\mbf{c}\in \mbf{Z}_{\geq 0}^{l}$, we have 
\[\Gup(b_{-1}(\mbf{c}, \tw))\Delta_{w\lambda}\simeq \Delta_{w\lambda}\Gup(b_{-1}(\mbf{c}, \tw)).\]
\begin{NB}\textup{(2)}
For $\lambda, \lambda'\in P_+$ and $w\leq w'$,
we have 
\[\Delta_{w'\lambda'}\Delta_{w\lambda}\simeq \Delta_{w\lambda}\Delta_{w'\lambda'}.\]
\end{NB}
\end{cor}
\subsection{Factorization of $q$-center}
In this subsection,
we prove the multiplicative property of 
$\mscr{B}_w(\infty)$ with respect to the quantum minors $\{\Delta_{w\lambda}\}_{\lambda\in P_+}$ in $\mca{O}_{q}[\overline{N_{w}}]$.
This is a generalization of \cite[3.1]{Cal:adapted}, \cite{Cal:qflag} and \cite[Lemma 4.2]{Lampe}.
This result can be considered as a $q$-analogue of \cite[Lemma 15.8]{GLS:KacMoody}.
\subsubsection{}
Using \corref{cor:string2} and \eqref{eq:string3} inductively, we obtain the following lemma.
\begin{lem}\label{lem:string3}
Let $w\in W$  and  $\tw=(i_{1}, \cdots, i_{l})\in R(w)$ as above.
We define
\[\vep_{\tw}(b):=(\vep_{i_{1}}(b), \vep_{i_{2}}(\eit{i}^{\max}b), \cdots, \vep_{i_{l}}(\eit{i_{l-1}}^{\max}\cdots \eit{i_{1}}^{\max}b))\]
for $b\in \mscr{B}(\infty)$.
For $b_{1}, b_{2}\in \mscr{B}(\infty)$, let us write
\[\Gup(b_{1})\Gup(b_{2})=\sum d_{b_{1}, b_{2}}^{b}(q)\Gup(b)\]
with $d_{b_{1}, b_{2}}^{b}(q)\in \mca{A}$.
If $d_{b_{1}, b_{2}}^{b}(q)\neq 0$, then  
$\vep_{\tw}(b)\leq \vep_{\tw}(b_{1})+\vep_{\tw}(b_{2})$, where $\leq$ is the lexicographic order on $\mbb{Z}_{\geq 0}^{l}$
as in \secref{subsec:lex}.

Let $b\in \mscr{B}_{w}(\infty)$ with $\vep_{\tw}(b)=\vep_{\tw}(b_{1})+\vep_{\tw}(b_{2})$ for $b_{1}, b_{2}\in \mscr{B}_{w}(\infty)$.
Then we have $d_{b_{1}, b_{2}}^{b}(q)=q^{N}$ for some $N\in \mbb{Z}$.
\end{lem}
\subsubsection{}
\begin{prop}
Let $w\in W$ and $\lambda, \mu\in P_+$.
For $b\in \mscr{B}_w(\mu)$ and $u_{w\lambda}\in \mscr{B}_w(\lambda)$, 
there exists $b'\in \mscr{B}_w(\lambda+\mu)$ such that 
\[\Phi(\lambda, \mu)(b')=u_{w\lambda}\otimes b,\]
and we have an equality in $\mca{O}_{q}[\overline{N}_{w}]$:
\[\Delta_{w\lambda}\Gup(\jbar_{\mu}(b))\simeq \Gup(\jbar_{\lambda+\mu}(b')).\]
\end{prop}
\begin{proof}
Fix $\tw=(i_{1}, \cdots, i_{l})\in R(w)$, we have 
\[\eit{i_1}^{\max}{(u_{w\lambda}\otimes b)}=u_{w_{\geq 2}\lambda}\otimes \eit{i_{1}}^{\max}b\]
by the tensor product rule \eqref{eq:eitn} for crystal operators and $\vphi_{i_{1}}(u_{w\lambda})=0$.
Using this recursively, we get
\[\eit{i_l}^{\max}\cdots \eit{i_1}^{\max}{(u_{w\lambda}\otimes b)}=u_{\lambda}\otimes u_{\mu}.\]
In particular, there exists $b'\in \mscr{B}_w(\lambda+\mu)$ such that $\Phi(\lambda, \mu)(b')=u_{w\lambda}\otimes b$.
By \propref{prop:prod_rep} and \propref{prop:prod_cry},
we have 
\begin{equation}q^{(\wt b-\mu, \lambda)}\Delta_{w\lambda}\Gup(\jbar_{\mu}{b})=\Gup(\jbar_{\lambda+\mu}(b'))+\sum_{}f_{b, w\lambda}^{b''}(q)\Gup(b'')
\label{eq:eq_1}\end{equation}
for some $f_{b, w\lambda}^{b''}(q)\in q\mbb{Z}[q]$.
By the second assertion of \lemref{lem:string3}, we have $f_{b, w\lambda}^{\jbar_{\lambda+\mu}(b')}(q)=0$ as in \cite[1.8 Proposition (i)]{Cal:adapted}.

Applying the dual bar-involution $\sigma$, we obtain 
\begin{equation}
q^{-(\wt b-\mu, \lambda)+(\wt b-\mu, w\lambda-\lambda)}\Gup(\jbar_{\mu}{b})\Delta_{w\lambda}=\Gup(\jbar_{\lambda+\mu}(b'))+\sum_{}f_{b, w\lambda}^{b''}(q^{-1})\Gup(b'').
\end{equation}
By \propref{prop:q-center}, we have $\Gup(\jbar_{\mu}{b})\Delta_{w\lambda}=q^m\Delta_{w\lambda}\Gup(\jbar_{\mu}{b})$
for some $m\in \mbb{Z}$ in $\mca{O}_{q}[\overline{N}_{w}]$
It is equal to 
\begin{equation}
q^{-(\wt b-\mu, \lambda)+(\wt b-\mu, w\lambda-\lambda)+m}\Delta_{w\lambda}\Gup(\jbar_{\mu}{b})=\Gup(\jbar_{\lambda+\mu}(b'))+\sum_{}f_{b, w\lambda}^{b''}(q^{-1})\Gup(b''). \label{eq:eq_2}
\end{equation}
Therefore we obtain $f_{b, w\lambda}^{b''}(q)=0$ for any $b''$ by comparing \eqref{eq:eq_1} and \eqref{eq:eq_2}.
\begin{NB}This is not true: In particular, $u_{w\lambda}\bot \jbar_{\mu}(b)$ for any $\lambda, \mu\in P_+$ and $b\in \mscr{B}_w(\mu)$.\end{NB}
\end{proof}
Since there exists $\mu\in P_+$ such that $\pi_\mu(b)\neq 0$, we obtain the following theorem. 
\begin{theo}\label{theo:Demazure}
Let $b\in \mscr{B}_{w}(\infty)$ and $\lambda\in P_{+}$.
There exists $b'\in \mscr{B}_{w}(\infty)$ such that 
\[\Delta_{w\lambda}\Gup(b)\simeq \Gup(b')\]
in $\mca{O}_{q}[\overline{N_{w}}]$.
\end{theo}
Taking $b$ from $\mscr{B}(w, -1)$, we obtain the following theorem by \corref{cor:dsum}.
\begin{theo}
For $\mbf{c}\in \mbf{Z}_{\geq 0}^{l}$ and $\lambda \in P_+$, 
we have 
\[\Delta_{w\lambda}\Gup(b_{-1}(\mbf{c}, \tw))\simeq \Gup(b_{-1}(\mbf{c}+\mbf{n}^{\lambda}, \tw)).\]
\end{theo}

\subsubsection{}\label{sec:adapted}
The following is a generalization of Caldero's result \cite[2.1 Lemma, 2.2 Theorem]{Cal:adapted}.
It follows from \thmref{theo:Demazure} by an induction on the length of $w$.

\begin{theo}\label{theo:seed}
	Let $w\in W$ and fix $\tw\in R(w)$.
	We set 
	\[\Delta_{\tw, k}:=\Delta_{s_{i_{1}}\cdots s_{i_{k}}\varpi_{i_{k}}}\]
	for $1\leq k\leq l$.
	Then $\{\Delta_{\tw, k}\}_{1\leq k\leq l}$ forms a strongly compatible subset.
\end{theo}

\subsubsection{}\label{sec:factorization}
Following \cite[15.5]{GLS:KacMoody},
we call $\mbf{c}\in \mbb{Z}_{\geq 0}^l$ \emph{interval-free} if $\mbf{c}$ satisfies the following conditions:
\[c^{(i)}:=\min\{c_k; i_k=i\}=0\]
for any $i\in I$.
By definition, $^\varphi\mbf{c}:=\mbf{c}-\sum_{i\in I}c^{(i)}\mbf{n}^i\in \mbb{Z}_{\geq 0}^l$ is interval free.
We have the following factorization property with respect to the extremal vectors $\{\Delta_{w\lambda}\}_{\lambda\in P_+}$.
\begin{theo}\label{theo:unip}
For $\mbf{c}\in \mbb{Z}_{\geq 0}^{l}$, we set $\lambda(\mbf{c}):=\sum_{i\in I}c^{(i)}\varpi_{i}\in P_+$.
Then we have 
\[\Gup(b_{-1}(\mbf{c}, \tw))\simeq \Gup(b_{-1}({^\varphi\mbf{c}}, \tw))\Delta_{w\lambda(\mbf{c})}.\]
\end{theo}

\begin{NB}
The following sections are far from complete.
\section{Conjecture and its consequence}\label{sec:conj}
In this section, we clarify a conjecture which concerns the quantum cluster algebra \cite{BZ:qcluster} and the quantum unipotent subgroup $\mca{O}_{q}[N(w)]$.

\subsection{Quantization conjecture}
\subsubsection{}
For $w\in W$, we consider the quantum unipotent subgroup $\mca{O}_{q}[N(w)]$ and define its $\mca{A}$-form
\[\mca{O}_{q}[N(w)]_{\mca{A}}:=\bigoplus_{b\in \mscr{B}(w)}\mca{A}\Gup(b)\]
by using the dual canonical basis.
Let $\mca{F}(\mca{O}_{q}[N(w)])$ be the fraction field of $\mca{O}_{q}[N(w)]$.
\subsubsection{}
Let $R$ be an (commutative) integral domain,
and $v\in R$ be an invertible element.
Let $L$ be a free $\mbb{Z}$-module  with a $\mbb{Z}$-valued skew symmetric bilinear form 
$\lambda\colon L\times L\to \mbb{Z}$.
\begin{defn}
The \emph{quantum torus} $\mca{T}=\mca{T}(L, \lambda)$ over $(R, v)$ is the $R$-algebra
generated by $X^{g}~(g\in L)$ with the following relations:
\[X^{g}X^{g'}=v^{\lambda(g, g')}X^{g+g'},\]
for any $g, g'\in L$.
\end{defn}

\subsubsection{}
For $\tw\in R(w)$, let $L$ be a free $\mbb{Z}$-module of rank $l$.
We set $\Lambda_{\tw}$ be the skew-symmetric bilinear form defined by
\[\Lambda_{\tw}(e_{i}, e_{j}):=N_{\tw}(\mbf{n}_{i}, \mbf{n}_{j}).\]
Here $N_{\tw}(\mbf{n}_{i}, \mbf{n}_{j})$ is defined in \eqref{eq:lambda}, and $\{e_{i}\}_{1\leq i\leq l}$ is a free basis of $L$.
\subsubsection{}
For $\tw\in R(w)$, we have a subalgebra $\mca{A}_{\tw}$ of $\mca{O}_q[N(w)]$ which is generated by $\{\Delta_{\tw, k}\}_{1\leq k\leq l}$.
Here $\mca{A}_{\tw}$ is a quantum polynomial algebra.
Let $\mca{T}_{\tw}$ by a quantum subtorus of $\mca{F}(\mca{O}_q[N(w)])$ which is generated by 
$\{\Delta_{\tw, k}\}_{1\leq k\leq l}$ and their inverses.
For the construction of quantum cluster algebra, we give a quantum seed.
To give a quantum seed, we need exchange matrix and the compatible pair.
Since the construction of ``initial seed'' is same as in \cite[Theorem 10.1]{BZ:qcluster} (and \cite{GLS:KacMoody}),
we already have the compatible pair of it.
We denote it  by $(\Gamma_{\tw}, \Lambda_{\tw})$.
Then we can associate a quantum cluster algebra with coefficient $\mscr{A}^{q}(\Gamma_{\tw}, \Lambda_{\tw})$ over $(\mca{A}, q_{s})$
in the sense of \cite[Definition 2.1]{Qin}.
As in \cite[Proposition 10.9]{BZ:qcluster}, we have the following: 
\begin{prop}
The correspondence $X_k \mapsto \Delta_{\tw, k}$ uniquely extends to 
the $\mbb{Q}(q_s)$-algebra isomorphism $\mca{T}(\Lambda_{\tw})\simeq \mca{T}_{\tw}$,
in particuar we have a $\Qq$-algebra embedding $\Phi_{\tw}\colon \mca{F}(\mca{T}(\Lambda_{\tw}))\hookrightarrow \mca{F}(\mca{O}_{q}(N(w))).$
\end{prop}
The following is a main conjecture of this paper.
\begin{conj}[quantization conjecture]

\textup{(1)}
By restricting the isomorphism $\Phi_{\tw} \colon \mca{T}(\Lambda_{\tw})\simeq \mca{T}_{\tw}$,
we obtain an $\mca{A}$-algebra isomorphism: 
\[\mscr{A}^{q}(\Gamma_{\tw}, \Lambda_{\tw})\otimes_{\mbb{Z}}\mbb{Q}\simeq \mca{O}_{q}[N(w)]_{\mca{A}}.\]
\textup{(2)}
The image of the set of quantum cluster monomials is contained in the dual canonical base $\mbf{B}^{\up}(w)$
up to some $q$-shifts.
\end{conj}

If we specialize at $q=1$, we obtain the following conjecture.
Since the cluster algebra has natural Poisson algebra structure and this can be ``semiclassical limit'' of  the quantum cluster algebra.
Also the coordinate ring of unipotent subgroup have natural Poisson structure which comes from the Poisson-Lie group structure on $G$.
So the isomorphism is in the classical limit should be considered the the isomorphism of \emph{Poisson algebras}.
\begin{conj}[weak quantization conjecture]
If we specialize the isomorphism at $q=1$, then we have an algebra isomorphism:
\[\mscr{A}(\Gamma_{\tw})\otimes_{\mbb{Z}}\mbb{C}\simeq \mbb{C}[N(w)].\]
\textup{(2)}
The image of the set of cluster monomials is contained in the specialized dual canonical base $\mbf{B}^{\up}(w)|_{q=1}$ at $q=1$.
\end{conj}
The definition of $\mbb{Z}$-form of the coordinate ring is not clear if we don't use the (dual) canonical basis.
This $\mbb{Z}$-form can be defined by using the Kostant's $\mbb{Z}$-form ?

\subsubsection{}
In \cite{GLS:KacMoody}, the initial seed associated with a reduced expression $\tw$ is also given by 
(restricted) generalized minor, so (1) can be considered as a quantum analogue.
Since ``coefficient'' is contained in any cluster, we should have multiplicative property with respect to any dual canonical basis element in view of the Conjecutre\ref{conj:qconj}.
Hence (3) can be considered as a quantum analogue of \cite[Lemma 15.8]{GLS:KacMoody}.

\subsection{Positivity conjecture}
For symmetric Kac-Moody case, the dual canonical basis has the positivity of structure constants for the multiplication.
In particular, the \emph{positivity conjecture} for (quantum) cluster variable is only a consequence of it.
Here, we say the positivity conjecture  for a  seed $\mbf{x}=\{x_{1}, \cdots, x_{l}\}$
(resp.\ a quantum seed $M\colon \mbb{Z}^{l}\to \mca{F}$) mean the following:

For (quantum) cluster variable $y$,
we consider the (quantum) Laurent expansion at the given seed $\mbf{x}=\{x_{1}, \cdots, x_{l}\}$
(resp.\ the given quantum seed $X\colon \mbb{Z}^{l}\to \mca{F}$):
\[y=\sum_{\mbf{c}\in \mbb{Z}^{l}}f_{\mbf{c}}x_{1}^{c_{1}}\cdots x_{l}^{c_{l}} , (\text{resp.\ }\!\!\sum_{\mbf{c}\in \mbb{Z}^{l}}f_{\mbf{c}}(q)X(\mbf{c}))\]
then we have $f_{\mbf{c}}\in \mbb{Z}_{\geq 0}$
(resp.\ $f_{\mbf{c}}(q_{s})\in \mbb{Z}_{\geq 0}[q_{s}^{\pm}]$).
\subsubsection{}

\subsection{Monoidal categorification}
\subsubsection{}
Hernandez and Leclerc \cite{HerLec} has introduced a notion of {monoidal categorification}.
We state their definition in slightly different form.
For a monoidal abelian category $\mfr{M}$ and its complexified Grothendieck ring $K_{0}(\mfr{M})\otimes_{\mbb{Z}}\mbb{C}$, we have a natural basis which is given by simple modules.
We say that a simple object of a monoidal abelian category $\mfr{M}$ is \emph{prime} if there exists no nontrivial factorization $x\simeq x_{1}\otimes x_{2}$.
We say that a simple object is (strongly) real object if $x\otimes x$ (resp.\ $x^{\otimes m}$ for $m\geq 2$) remains simple.

\subsubsection{}

For a frozen quiver $\Gamma$, we say \emph{monoidal categorification} of cluster algebra $\mscr{A}_{\Gamma}$ with coefficient in the sense of \cite{FZ:cluster4}
by a monoidal abelian category $\mfr{A}_{\Gamma}$ such that 
\begin{enumerate}
	\item There exists an algebra isomorphism between complexified Grothendieck ring of $\mfr{A}_{\Gamma}$,
	that is 
	\[\mscr{A}(\Gamma)\otimes \mbb{C}\simeq K_{0}(\mfr{A}_{\Gamma})\otimes_{\mbb{Z}}\mbb{C},\]
	\item the cluster monomials are the classes of all the (strongly) real simple objects of $\mfr{M}$.
        \item the cluster variables are the classes of all the (strongly) real prime simple objects of $\mfr{M}$ 
\end{enumerate}
If such categorification exists, then the positivity conjectures of $\mscr{A}(\Gamma)$ can be shown, see \cite[Proposition 2.2]{HerLec}

\subsubsection{}
Under the monoidal categorification of the cluster algebra, 
we have the following \emph{simple product theorem} as in \cite{HerLec} and \cite{Her}
\begin{prop}[simple product conjecture]\label{conj:simpleproduct}
If the monoidal categorification is obtained, then we have the following theorem:

For real elements  $\{b_{1}, \cdots, b_{l}\}\subset \mscr{B}(w)$, the followings are equivalent:

\textup{(1)} $\{b_{1}, \cdots, b_{l}\}$ is compatible,

\textup{(2)} $\{b_{1}, \cdots, b_{l}\}$ is strongly compatible, in particular, $b_{i}$ is strongly real for $1\leq i\leq l$.
\end{prop}

\subsubsection{}
We propose a natural framework which gives quantum cluster algebras.
Since Khovanov-Lauda-Rouquier's algebras are graded algebras, and their graded module categories give the quantized enveloping algebra and furthermore
the dual canonical basis is given by the class of graded simple modules with some specified shifts \cite{VV:KLR}.
As a consequence of the quantization conjecture for symmetric case, we obtain the following graded analogue of monoidal categorification
(we call \emph{graded monoidal categorification}).
For a monoidal abelian category $(\mfr{M}, [1])$ with a shift functor $[1]$ , we define the corresponding quantum Grothendieck ring as follows:
Let $K_{0}(\mfr{M})$ be the Grothendieck ring of $\mfr{M}$ and define $\mbb{Z}[q^{\pm}]$-action by 
$q[L]=[L[1]]$. 
As same as monoidal categorification, graded monoidal categorification of quantum cluster algebra yields the positivity conjecture for it.
In particular, cluster expansion formulae gives the graded decomposition numbers for the tensor product of mutation pair

For a compatible pair $(\Gamma, \Lambda)$, we say a \emph{graded monoidal categorification} of the quantum cluster algebra with coefficient $\mscr{A}^{q}(\Gamma, \Lambda)$
by a monoidal abelian category with a shift functor $(\mfr{A}_{\Gamma, \Lambda}, [1])$ such that 
\begin{enumerate}
	\item There exists an $\mbb{Z}[q^{\pm}]$-algebra isomorphism between the quantum Grothendieck ring of $\mfr{A}_{\Gamma, \Lambda}$,
	that is we have 
	\[\mscr{A}^{q}(\Gamma, \Lambda)\simeq K_{0}(\mfr{A}_{\Gamma, \Lambda}),\]
	
	\item the quantum cluster monomials are contained in the classes of (strongly) real simple objects of $\mfr{A}_{\Gamma, \Lambda}$,
	
	\item the quantum cluster variables are contained in the classes of (strongly) real prime objects of $\mfr{A}_{\Gamma, \Lambda}$.
\end{enumerate}

\subsubsection{}
In \cite{HerLec}, they introduced the monoidal subcategory $\mca{C}_l$
of finite dimensional representations of quantm affine algebra $\Uq(\mbf{L}\mfr{g})$.
By using $\mca{C}_{1}$ of type $A_{n}$ and $D_{4}$, they gave a monoidal categorification.
Nakajima \cite{Nak:cluster} proved their conjecture for any symmetric Kac-Moody type with some bipartite orientations, in particular for finite ADE cases.
He has shown that the cluster algebra can be embed into its Grothendieck ring of $\mca{C}_1$ and 
the set of cluster monomial forms a subset of the basis given by the classes of simple modules in it.
In particular, under some modification, it is shown that the truncation of (specialization) 
truncated $q$-character gives a cluster character.
In its construction, the $t$-analogue of $q$-character of simple modules can be regard as the dual canonical basis or its $t$-character.
Since there is a $t$-deformed Grothendieck ring by $t$-analogue of $q$-character,
then the $t$-analogue of cluster algebra (quantum cluster algebra) should be given by it under some modification.

For the generalization of the his construction, it is natural to consider subalgebras of $\Uq^-(\mfr{g})$ which are compatible with the dual canonical basis.
So this is precisely the original motivation of the study of the cluster algebras !

\subsubsection{}
By specializing the conjecture \ref{conj:qcat} at $q=1$,
and \cite{VV:KLR}, we have the following conjecture.
Let $\mfr{N}(w)$ be the extension-closed full subcategory of finite dimensional modules over the Khovanov-Lauda-Rouquier's algebras
which is generated by $\mscr{B}(w, -1)$, that is the tensor subcategory 
of finite dimensional module over Khovanov-Lauda-Rouquier's algebras
whose composition factors are contained in $\mscr{B}(w, -1)$.
This is already a corollary of Gei\ss-Leclerc-Schr\"{o}er's additive cateogorification
and theie {open orbit conjecture}.
\begin{prop}
If the (weak) quantization conjecture is true,
then the monoidal category $\mfr{N}(w)$ gives a monoidal categorification of the cluster algebra $\mscr{A}(\Gamma_{\tw})\mbb{C}[N(w)]$.
In particular, we have an isomorphism of algebras:
\[K_0(\mfr{N}(w))\otimes_{\mbb{Z}}\mbb{C}\simeq \mbb{C}[N(w)],\]
and the set of cluster monomials of $\mbb{C}[N(w)]$ coincides with the class of real simple objects.
\end{prop}

\subsection{Gei\ss-Leclerc-Schr\"{o}r's open orbit conjecture}
\subsubsection{}
For the canonical basis $\mbf{B}$, we have a crystal structure $\mscr{B}(\infty)$.
In \cite{KasSai}, they constructed the crystal $\mscr{B}(\infty)$ by using the set $\operatorname{Irr}\Lambda$ of irreducible components of certain Lagrangian variety $\Lambda_{\mbf{V}}$ of cotangent bundle 
$\mbf{E}_{\mbf{V}}=T^*\mbf{E}_{\Omega, \mbf{V}}$ of representation variety $\mbf{E}_{\Omega, \mbf{V}}$ of quiver $(I, \Omega)$.
This Lagrangian variety $\Lambda_{\mbf{V}}$ is called Lusztig's quiver variety.
By using their result, Lusztig introduced the \emph{semicanonical basis} which is a basis of universal enveloping algebra $U(\mathfrak{n}_-)$ in \cite{Lus:aff, Lus:semican}.
He studied the space of constructible functions on $\Lambda_{\mbf{V}}$ and consider a linear form $\rho_b$ which evaluates ``generic value'' of constructible functions for each irreducible component $\Lambda_b\in \Irr\Lambda$.
Here $\{\rho_{b}\}_{b\in \mscr{B}(\infty)}$ is the \emph{dual semicanonical basis}.
The semicanonical basis $\{f_b\}_{b\in \mscr{B}(\infty)}$ is defined by its generic value for each irreducible components
by $\rho_{b}(f_{b'})=\delta_{b, b'}$, so this is canonically parametrized by $\operatorname{Irr}\Lambda$.
It is known that this basis has same crystal structure as same as canonical basis \cite[Theorem 3.1, Theorem 3.8]{Lus:semican}.
\subsubsection{}
In \cite{GLS:adaptable, GLS:KacMoody}, 
Gei\ss, Leclerc, Schr\"{o}er studied a connection between the dual semicanonical basis and cluster algebras.
In particular they gave \emph{additive categorification} of the cluster algebra of 
unipotent subgroups and unipotent cells by using Frobenius $2$-Calabi-Yau category and its cluster structure
in the sense of \cite{BIRSc}.

\subsubsection{}
Furthermore, they give the following conjecture which connects the dual semicanonical basis and the specialization of the dual canonical basis.
\begin{conj}[Open orbit conjecture {\cite[Conjecture 18.1]{GLS:KacMoody}}]\label{conj:naive_ooc}
For $b\in \mscr{B}(\infty)$, let $\Gup(b)|_{q=1}$ and $\rho_{b}$ be the associated the specialization of the dual canonical basis element and the dual semicanonical basis element.
If the associated irreducible component $\Lambda_b$ contains an open orbit, then we have 
\[\Gup(b)|_{q=1}=\rho_{b}.\]
\end{conj}
The above conjecture states that the set of cluster monomials is contained the intersection of the specialized dual canonical basis $\mca{B}^*:=\mbf{B}^{\up}|_{q=1}$ 
and the dual semicanonical basis $\mca{S}^*$.

\subsubsection{}
If we assume the {open orbit conjecture}, we obtain the following proposition.
\begin{prop}
If $b$ is rigid, then $b$ is strongly real.
\end{prop}
\begin{proof}
Let $b$ a rigid element, then 
the corresponding irreducible component $\Lambda_b$ contains an open orbit.
We denote its module by $M_b$.
By its definition, $M_b^{\+m}$ is also rigid for any $m\geq 2$.
Then we have a corresponding irreducible component which arises as its orbit closure.
We denote it $b^{[m]}$. Then we can show that $b^{[m]}=S_m(b)$.
Since $b^{[m]}$ is also rigid, we have $\rho_b^m=\rho_{S_m(b)}$.
\emph{Here we use the {open orbit conjecture}}, then we have $\Gup(b)^m|_{q=1}=\Gup(S_m(b))|_{q=1}$.
Then we obtain $\Gup(b)^m=q^N\Gup(S_m(b))$ for some $N\in \mbb{Z}$ by the \emph{positivity of the dual canonical basis}.
In particular, $b$ is strongly real.
\end{proof}

\subsubsection{}
If we assume the conjecture \cite[Conjecture II 5.3]{BIRSc}, any rigid object in the Frobenius $2$ Calabi-Yau category $\mca{C}_{w}$ can be obtained by iterations of mutations.
In view of the monoidal categorification, 
\begin{prop}\label{conj:refinedOOC}
If the (weak) quantization conjecture is true, then we obtain the following propositon:

For $b\in \mscr{B}(w)$, the followings are equivalent:

\textup{(1)} $b$ is rigid, that is the associated irreducible component containes an open orbit,

\textup{(2)} $b$ is strongly real.

Moreover, the dual semicanonical basis element $\rho_{b}$ and the specialized dual canonical basis element coincides for such elements.
\end{prop}

\subsubsection{}
In the similarity with the \cite{HerLec} and \cite{Her}, the following conjecture should be expected from our Conjecture \ref{conj:qconj}.
As a collary, we can prove that real element is strongly real, and compatible family which consists of real elements is strongly compatible family.
\begin{conj}\label{conj:simpleproduct}
For a sequence of elements $b_1, \cdots , b_l\in \mscr{B}(\infty)$,
the followings are equivalent: 

\textup{(1)}
$\Gup(b_1)\cdots \Gup(b_l)\in q^{\mbb{Z}}\mbf{B}^{\up}$.

\textup{(2)}
$\Gup(b_i)\Gup(b_j)\in q^{\mbb{Z}}\mbf{B}^{\up}$ for any $i<j$.
\end{conj}

As in \cite{Her}, a direction (1) to (2) is clear.
So the main part is another direction (2) to (1).

\subsection{Simple product conjecture}
In view of the {open orbit conjecture} and the {quantized categorification conjecure} \ref{conj:qconj},
the following conjecture is expected in the framework of the monoidal categorification \cite[Theorem 8.1]{HerLec} and \cite{Her}.
\begin{conj}\label{conj:simple_product}

\textup{(1)}
Let $b_1, \cdots, b_m$ be a family of elements in $\mscr{B}(\infty)$.
Suppose that $b_i \bot b_j$ for every $1\leq i<j \leq m$.
Then $\{b_1, \cdots, b_m\}$ is a compatible family.

\textup{(2)}
Let $b\in \mscr{B}(\infty)$ be a real element.
Then $b$ is strongly real element and $b^{[m]}=S_m(b)$ for any $m$.

\textup{(3)}
Let $\{b_1, \cdots, b_l\}$ be a compatible family of real elements in $\mscr{B}(\infty)$.
Then $\{b_{1}, \cdots, b_{l}\}$ is a strongly compatible family.
 \end{conj}

\begin{rem}
 The converse of (1) follows from the categorification.
 \textup{(2)} and \textup{(3)} are corollaries of \textup{(1)}.
 By construction, the subalgebra generated by a compatible familty of real elements
 are adapted algebra in the sense of \cite{Cal:adapted}, and for such adapted algebra
 $A_t\subset A$ we have $\operatorname{GKdim}(A_t)=l$. 
 \end{rem}

To relate some part of the dual canonical basis and (quantum) cluster algebra,
we introduce a notion which should correspond to the notion of (basic) cluster-tilting object
under the GLS's open orbit conjecture. 
 \begin{defn}
Let $\mscr{A}\subset \Uq^-(\mfr{g})$ be a $\Qq$-algebra which is compatible with upper global basis,
that is we have 
$\mscr{A}=\bigoplus_{b\in \Gup(\mscr{B}(\infty))\cap \mscr{A}}\Qq \Gup(b)$.
Here we set the corresponding subset $\mscr{B}(\mscr{A})\subset \mscr{B}(\infty)$.

A compatible family of real prime elements $\underline{\delta}:=\{\delta_1, \cdots, \delta_l\}\subset  \mscr{B}(\mscr{A})$
is called a \emph{multiplicative cluster} if the following holds:
\begin{align*}
\mscr{B}_{\underline{\delta}}:=
& \{b\in \mscr{B}(\mscr{A}) ; b \bot \delta_{i}~(1\leq i\leq l) \} \\
=&\{b\in \mscr{B}(\mscr{A}); \Gup(b)\simeq \Gup(\delta_1)^{n_1}\cdots \Gup(\delta_l)^{n_l}, n_1,\cdots, n_l\in \mbb{Z}_{\geq 0}\}.
 \end{align*}
 \end{defn}
By construction,
we have $\operatorname{GKdim}(\mscr{A}_{\underline{\delta}})=l$ and $l\leq \operatorname{GKdim}(\mscr{A})$,
where  $\mscr{A}_{\underline{\delta}}$ is the subalgebra of $\mscr{A}$ generated by $\underline{\delta}$.

Our conjecture is the following: 

\begin{conj}
Under the GLS'S open orbit conjecture, 
for the subalgebra $\Uq^-(w)$,
we can define a quantum cluster algebra structure and 
cluster corresponds to a multiplicative cluster in the above sense.
\end{conj}

\begin{defn}\cite[2.2 Definition 1]{Cal:adapted}
\textup{(1)}
Let  $A_{t}\subset A\subset \Uq^-(\mfr{g})$ be $\Qq$-subalgebras which are compatible with the dual canonical basis $\mbf{B}^{\up}$, i.e.,
which are spanned by subsets $\mbf{B}^{\up}(A_t)\subset \mbf{B}^{\up}(A)$ of $\mbf{B}^{\up}$.

$A_t$ is called \emph{an adapted algebra} of $A$ if for any pair of $b^*, b^{*'}\in \mbf{B}^{\up}(A_t)$ is multiplicative and 
there exists a finite subset $\mbf{P}_t^*\subset \mbf{B}^{\up}(A_t)$ such that
a multiplicative subset generated by $\mbf{P}^*_t$ is an Ore set in $A_t$ and coincides with $q^{\mbb{Z}}\mbf{B}^{\up}(A_t)$.

\textup{(2)}
An adapted algebra $A_t\subset A$ is called \emph{maximal} if $A_t\subset A_{t'}\subset A$ and $A_{t'}$ is adapted
implies $A_{t'}=A_{t}$
\end{defn}
\begin{rem}
In the connection of the quantum cluster algebra,
we consider ambient algebra $A'$ as cluster algebra and maximal adapted subalgebra $A\subset A'$ as certain quantum seed (or corresponding toric frame).
In particular, we should have $\Frac(A)=\Frac(A')$.
More precisely, 
the subalgebra generated by maximal adaptable algebras which are ``reachable'' via ``quantum mutation''
from certain seed corresponds to the ambient algebra
and the intersection of  reachable adapted algebras should coincides with the coefficient of quantum cluster algebra and the $q$-center of
the ambient algebra $A'$.

\centering
\begin{tabular}{|c|c|}
cluster algebra & ambient algebra $A$ \\
quantum seed & adapted algebra $A_t$\\
coefficient & $q$-center $Z_q(A)=\cap_{t\in \mbb{T}}A_t$
\end{tabular}
\label{default}

\end{rem}

\subsection{Lusztig's problem on characteristic variety}
In \cite[13.7]{Lus:quiver}, he stated a following conjecture.
For a simple perverse sheaves (or regular holonomic $\mca{D}$-module) $\mscr{P}_{\Omega, b}$ which arising in the canonical basis,
the characteristic varieties $\operatorname{Ch}(\mscr{P}_{\Omega, b}):=\Supp(\mca{E}\otimes_{\pi^{-1}\mca{D}}\mscr{P}_{\Omega, b})$ of 
them are independent of a choice of $\Omega$ and depend only on $b\in \mscr{B}(\infty)$
this can be proved by the Fourier-Deligne-Sato transformation.
\begin{enumerate}
\item For graph of type ADE, the characteristic varieties of simple perverse sheaves (or regular holonomic $\mca{D}$-module) which arising in the canonical basis are irreducible
varieties
\item For general graph, there is a unique bijection $s$ between $\Irr\Lambda$ and the set $\mscr{P}_{\Omega}$ of simple perverse sheaves 
which satisfies $\Lambda_{s(b)}\subset \operatorname{Ch}(\mscr{P}_{\Omega, b})$.
\end{enumerate}
For the first question, a counter example is given in \cite{KasSai} in type $A_5$.
For the second question, existence of such bijection ($s=id$) is given in \cite{KasSai} but uniqueness is still open except for ADE case and affine ADE case.
``Similarity'' of counter example is mentioned in \cite[Remark 1]{Lec:imaginary} and 
\cite[1.3]{GLS:semican1}.
\end{NB}

\bibliographystyle{plain}
\bibliography{cluster}

\begin{thebibliography}{10}

\bibitem{Beck:convex}
J.~Beck.
\newblock Convex bases of {PBW} type for quantum affine algebras.
\newblock {\em Comm. Math. Phys.}, 165(1):193--199, 1994.

\bibitem{BFZ:cluster3}
A.~Berenstein, S.~Fomin, and A.~Zelevinsky.
\newblock Cluster algebras. {III}. {U}pper bounds and double {B}ruhat cells.
\newblock {\em Duke Math. J.}, 126(1):1--52, 2005.

\bibitem{BZ:string}
A.~Berenstein and A.~Zelevinsky.
\newblock String bases for quantum groups of type {$A_r$}.
\newblock In {\em I. {M}. {G}elfand {S}eminar}, volume~16 of {\em Adv. Soviet
  Math.}, pages 51--89. Amer. Math. Soc., Providence, RI, 1993.

\bibitem{BZ:qcluster}
A.~Berenstein and A.~Zelevinsky.
\newblock Quantum cluster algebras.
\newblock {\em Adv. Math.}, 195(2):405--455, 2005.

\bibitem{BIRSc}
A.~B. Buan, O.~Iyama, I.~Reiten, and J.~Scott.
\newblock Cluster structures for 2-{C}alabi-{Y}au categories and unipotent
  groups.
\newblock {\em Compos. Math.}, 145(4):1035--1079, 2009.

\bibitem{BuaMar}
A.~B. Buan and R.~Marsh.
\newblock Cluster-tilting theory.
\newblock In {\em Trends in representation theory of algebras and related
  topics}, volume 406 of {\em Contemp. Math.}, pages 1--30. Amer. Math. Soc.,
  Providence, RI, 2006.

\bibitem{Cal:unity}
P.~Caldero.
\newblock On the {$q$}-commutations in {$U_q(\mathfrak{n})$} at roots of one.
\newblock {\em J. Algebra}, 210(2):557--576, 1998.

\bibitem{Cal:adapted}
P.~Caldero.
\newblock Adapted algebras for the {B}erenstein-{Z}elevinsky conjecture.
\newblock {\em Transform. Groups}, 8(1):37--50, 2003.

\bibitem{Cal:qflag}
P.~Caldero.
\newblock A multiplicative property of quantum flag minors.
\newblock {\em Represent. Theory}, 7:164--176 (electronic), 2003.

\bibitem{DamDec}
I.~Damiani and C.~De~Concini.
\newblock Quantum groups and {P}oisson groups.
\newblock In {\em Representations of {L}ie groups and quantum groups ({T}rento,
  1993)}, volume 311 of {\em Pitman Res. Notes Math. Ser.}, pages 1--45.
  Longman Sci. Tech., Harlow, 1994.

\bibitem{DKP:solvable}
C.~De~Concini, V.~G. Kac, and C.~Procesi.
\newblock Some quantum analogues of solvable {L}ie groups.
\newblock In {\em Geometry and analysis ({B}ombay, 1992)}, pages 41--65. Tata
  Inst. Fund. Res., Bombay, 1995.

\bibitem{FG:qcluster}
V.~V. Fock and A.~B. Goncharov.
\newblock The quantum dilogarithm and representations of quantum cluster
  varieties.
\newblock {\em Invent. Math.}, 175(2):223--286, 2009.

\bibitem{FG:cluster1}
V.V. Fock and A.B. Goncharov.
\newblock Cluster ensembles, quantization and the dilogarithm.
\newblock {\em {Ann. Sci. \'{E}c. Norm. Sup\'{e}r. (4)}}, 42(6):865--930, 2009.

\bibitem{FG:cluster2}
V.V. Fock and A.B. Goncharov.
\newblock Cluster ensembles, quantization and the dilogarithm. {II}. the
  intertwiner.
\newblock In {\em Algebra, arithmetic, and geometry: in honor of {Y}u. {I}.
  {M}anin. {V}ol. {I}}, volume 269 of {\em Progr. Math.}, pages 655--673.
  {Birkh\"{a}user Boston Inc.}, Boston, MA, 2009.

\bibitem{FZ:cluster1}
S.~Fomin and A.~Zelevinsky.
\newblock Cluster algebras. {I}. {F}oundations.
\newblock {\em J. Amer. Math. Soc.}, 15(2):497--529 (electronic), 2002.

\bibitem{FZ:cluster2}
S.~Fomin and A.~Zelevinsky.
\newblock Cluster algebras. {II}. {F}inite type classification.
\newblock {\em Invent. Math.}, 154(1):63--121, 2003.

\bibitem{FZ:cluster4}
S.~Fomin and A.~Zelevinsky.
\newblock Cluster algebras. {IV}. {C}oefficients.
\newblock {\em Compos. Math.}, 143(1):112--164, 2007.

\bibitem{GLS:semican1}
C.~Gei{\ss}, B.~Leclerc, and J.~Schr{\"o}er.
\newblock Semicanonical bases and preprojective algebras.
\newblock {\em Ann. Sci. \'Ecole Norm. Sup. (4)}, 38(2):193--253, 2005.

\bibitem{GLS:Verma}
C.~Gei{\ss}, B.~Leclerc, and J.~Schr{\"o}er.
\newblock Verma modules and preprojective algebras.
\newblock {\em Nagoya Math. J.}, 182:241--258, 2006.

\bibitem{GLS:adaptable}
C.~Gei{\ss}, B.~Leclerc, and J.~Schr{\"o}er.
\newblock Cluster algebra structures and semicanonical bases for unipotent
  groups.
\newblock E-print arXiv \url{http://arxiv.org/abs/math/0703039}, 2007.

\bibitem{GLS:semican2}
C.~Gei{\ss}, B.~Leclerc, and J.~Schr{\"o}er.
\newblock Semicanonical bases and preprojective algebras. {II}. {A}
  multiplication formula.
\newblock {\em Compos. Math.}, 143(5):1313--1334, 2007.

\bibitem{GLS:chamber}
C.~Gei{\ss}, B.~Leclerc, and J.~Schr{\"o}er.
\newblock Generic bases for cluster algebras and the chamber ansatz.
\newblock E-print arXiv \url{http://arxiv.org/abs/1004.2781}, 2010.

\bibitem{GLS:KacMoody}
C.~Gei{\ss}, B.~Leclerc, and J.~Schr{\"o}er.
\newblock Kac-moody groups and cluster algebras.
\newblock E-print arXiv \url{http://arxiv.org/abs/1001.3545}, 2010.

\bibitem{GroLus}
I.~Grojnowski and G.~Lusztig.
\newblock A comparison of bases of quantized enveloping algebras.
\newblock In {\em Linear algebraic groups and their representations ({L}os
  {A}ngeles, {CA}, 1992)}, volume 153 of {\em Contemp. Math.}, pages 11--19.
  Amer. Math. Soc., Providence, RI, 1993.

\bibitem{Kas:crystal}
M.~Kashiwara.
\newblock On crystal bases of the {$Q$}-analogue of universal enveloping
  algebras.
\newblock {\em Duke Math. J.}, 63(2):465--516, 1991.

\bibitem{Kas:Demazure}
M.~Kashiwara.
\newblock The crystal base and {L}ittelmann's refined {D}emazure character
  formula.
\newblock {\em Duke Math. J.}, 71(3):839--858, 1993.

\bibitem{Kas:global}
M.~Kashiwara.
\newblock Global crystal bases of quantum groups.
\newblock {\em Duke Math. J.}, 69(2):455--485, 1993.

\bibitem{Kas:bases}
M.~Kashiwara.
\newblock {\em Bases cristallines des groupes quantiques}, volume~9 of {\em
  Cours Sp\'ecialis\'es [Specialized Courses]}.
\newblock Soci\'et\'e Math\'ematique de France, Paris, 2002.
\newblock Edited by Charles Cochet.

\bibitem{KhoLau:II}
M.~Khovanov and A.~D. Lauda.
\newblock A diagrammatic approach to categorification of quantum groups. {II}.
\newblock E-print arXiv \url{http://arxiv.org/abs/0804.2080}, 2008.

\bibitem{KhoLau:I}
M.~Khovanov and A.~D. Lauda.
\newblock A diagrammatic approach to categorification of quantum groups. {I}.
\newblock {\em Represent. Theory}, 13:309--347, 2009.

\bibitem{Kum}
S.~Kumar.
\newblock {\em Kac-{M}oody groups, their flag varieties and representation
  theory}, volume 204 of {\em Progress in Mathematics}.
\newblock {Birkh\"{a}user Boston Inc.}, Boston, MA, 2002.

\bibitem{Lampe}
P.~Lampe.
\newblock {A quantum cluster algebra of Kronecker type and the dual canonical
  basis}.
\newblock E-print arXiv \url{http://arxiv.org/abs/1002.2762}, 2010.

\bibitem{Lec:imaginary}
B.~Leclerc.
\newblock Imaginary vectors in the dual canonical basis of
  {$U_q(\mathfrak{n})$}.
\newblock {\em Transform. Groups}, 8(1):95--104, 2003.

\bibitem{Lec:qchar}
B.~Leclerc.
\newblock Dual canonical bases, quantum shuffles and {$q$}-characters.
\newblock {\em Math. Z.}, 246(4):691--732, 2004.

\bibitem{HerLec}
B.~Leclerc and D.~Hernandez.
\newblock Cluster algebras and quantum affine algebras.
\newblock {\em Duke Math. J.}, 154:265--341, 2010.

\bibitem{LNT}
B.~Leclerc, M.~Nazarov, and J.-Y. Thibon.
\newblock Induced representations of affine {H}ecke algebras and canonical
  bases of quantum groups.
\newblock In {\em Studies in memory of {I}ssai {S}chur ({C}hevaleret/{R}ehovot,
  2000)}, volume 210 of {\em Progr. Math.}, pages 115--153. {Birkh\"auser
  Boston}, Boston, MA, 2003.

\bibitem{LevSoi:qWeyl}
S.~Levendorski{\u\i} and Y.~So{\u\i}belman.
\newblock Some applications of the quantum {W}eyl groups.
\newblock {\em J. Geom. Phys.}, 7(2):241--254, 1990.

\bibitem{LevSoi}
S.~Levendorski{\u\i} and Y.~So{\u\i}belman.
\newblock Algebras of functions on compact quantum groups, {S}chubert cells and
  quantum tori.
\newblock {\em Comm. Math. Phys.}, 139(1):141--170, 1991.

\bibitem{Lus:quiver}
G.~Lusztig.
\newblock Quivers, perverse sheaves, and quantized enveloping algebras.
\newblock {\em J. Amer. Math. Soc.}, 4(2):365--421, 1991.

\bibitem{Lus:aff}
G.~Lusztig.
\newblock Affine quivers and canonical bases.
\newblock {\em Inst. Hautes \'Etudes Sci. Publ. Math.}, (76):111--163, 1992.

\bibitem{Lus:intro}
G.~Lusztig.
\newblock {\em Introduction to quantum groups}, volume 110 of {\em Progress in
  Mathematics}.
\newblock {Birkh\"{a}user Boston Inc.}, Boston, MA, 1993.

\bibitem{Lus:problem}
G.~Lusztig.
\newblock Problems on canonical bases.
\newblock In {\em Algebraic groups and their generalizations: quantum and
  infinite-dimensional methods ({U}niversity {P}ark, {PA}, 1991)}, volume~56 of
  {\em Proc. Sympos. Pure Math.}, pages 169--176. Amer. Math. Soc., Providence,
  RI, 1994.

\bibitem{Lus:braid}
G.~Lusztig.
\newblock Braid group action and canonical bases.
\newblock {\em Adv. Math.}, 122(2):237--261, 1996.

\bibitem{Lus:semican}
G.~Lusztig.
\newblock Semicanonical bases arising from enveloping algebras.
\newblock {\em Adv. Math.}, 151(2):129--139, 2000.

\bibitem{MSW:positivity}
G.~Musiker, R.~Schiffler, and L.~Williams.
\newblock {Positivity for cluster algebras from surfaces}.
\newblock E-print arXiv \url{http://arxiv.org/abs/0906.0748}.

\bibitem{NaiSag}
S.~Naito and D.~Sagaki.
\newblock Crystal of {L}akshmibai-{S}eshadri paths associated to an integral
  weight of level zero for an affine {L}ie algebra.
\newblock {\em Int. Math. Res. Not.}, (14):815--840, 2005.

\bibitem{Nak:cluster}
H.~Nakajima.
\newblock {Quiver varieties and cluster algebras}.
\newblock E-print arXiv \url{http://arxiv.org/abs/0905.0002}, 2009.

\bibitem{Nak:CBMS}
H.~Nakajima.
\newblock Quiver varieties and canonical bases of quantum affine algebras.
\newblock Available at \url{http://www4.ncsu.edu/~jing/conf/CBMS/cbms10.html},
  2010.

\bibitem{Qin}
F.~Qin.
\newblock Quantum cluster variables via {S}erre polynomials.
\newblock E-print arXiv \url{http://arxiv.org/abs/1004.4172}, 2010.

\bibitem{Rei:mult}
M.~Reineke.
\newblock Multiplicative properties of dual canonical bases of quantum groups.
\newblock {\em J. Algebra}, 211(1):134--149, 1999.

\bibitem{Rouq:2KM}
R.~Rouquier.
\newblock {2-Kac-Moody algebras}.
\newblock E-print arXiv \url{http://arxiv.org/abs/0812.5023}, 2008.

\bibitem{Saito:PBW}
Y.~Saito.
\newblock {PBW} basis of quantized universal enveloping algebras.
\newblock {\em Publ. Res. Inst. Math. Sci.}, 30(2):209--232, 1994.

\bibitem{VV:KLR}
M.~Varagnolo and E.~Vasserot.
\newblock {Canonical bases and Khovanov-Lauda algebras}.
\newblock E-print arXiv \url{http://arxiv.org/abs/0901.3992}, 2009.

\end{thebibliography}
\end{document}